\documentclass[12pt]{amsart}

\usepackage{amsthm}
\newtheorem{theorem}{Theorem}[section]
\newtheorem{lemma}[theorem]{Lemma}
\newtheorem{proposition}[theorem]{Proposition}
\theoremstyle{definition}
\newtheorem{definition}[theorem]{Definition}

\theoremstyle{remark}
\newtheorem{remark}[theorem]{Remark}

\counterwithin*{section}{part}

\usepackage{amssymb}
\usepackage{empheq}
\usepackage{stmaryrd}
\usepackage{enumerate}
\usepackage[shortlabels]{enumitem}
\usepackage{calc}
\usepackage{url}
\usepackage{comment}
\usepackage{hyperref}
\usepackage{tikz}
\usetikzlibrary{calc}

\usepackage{mathtools}

\def\PP{\mathbb{P}}
\def\RR{\mathbb{R}}

\def\NN{\mathbb{N}}

\newcommand{\WEHAomega}{\mathsf{WE}\mbox{-}\mathsf{HA}^\omega}
\newcommand{\WEPAomega}{\mathsf{WE}\mbox{-}\mathsf{PA}^\omega}
\newcommand{\QFAC}{\mathsf{QF}\mbox{-}\mathsf{AC}}
\newcommand{\QFER}{\mathsf{QF}\mbox{-}\mathsf{ER}}
\newcommand{\DC}{\mathsf{DC}}
\newcommand{\AC}{\mathsf{AC}}
\newcommand{\IPU}{\mathsf{IP}^\omega_\forall}
\newcommand{\IPEF}{\mathsf{IP}^\omega_{ef}}
\newcommand{\MARK}{\mathsf{M}^\omega}
\newcommand{\Algebra}{S}
\newcommand{\UBOmega}[1]{\exists\text{-}\mathsf{UB}^{\Omega,#1}}
\newcommand{\UBF}[1]{\exists\text{-}\mathsf{UB}^{\Algebra,#1}}
\newcommand{\FullUBOmega}[1]{\mathsf{UB}^{\Omega,#1}}
\newcommand{\FullUBF}[1]{\mathsf{UB}^{\Algebra,#1}}
\newcommand{\CCOmega}[1]{\mathsf{QF}\text{-}\mathsf{CC}^{\Omega,#1}}
\newcommand{\CCF}[1]{\mathsf{QF}\text{-}\mathsf{CC}^{\Algebra,#1}}

\newenvironment{claim}[1]{\par\noindent\underline{Claim:}\space#1}{}
\newenvironment{claimproof}[1]{\par\noindent\underline{Proof:}\space#1}{\leavevmode\unskip\penalty9999 \hbox{}\nobreak\hfill\quad\hbox{$\blacksquare$}}

\usepackage[left=2.0cm,%
                right=2.0cm,%
                top=2.5cm,%
                bottom=3.5cm,%
                headheight=12pt,%
                a4paper]{geometry}%

\begin{document}

\title[A systematic way of analysing proofs in probability theory]{A systematic way of analysing proofs in probability theory}

\author[M. Neri, P. Oliva, N. Pischke]{Morenikeji Neri${}^{\MakeLowercase a}$, Paulo Oliva${}^{\MakeLowercase b}$, Nicholas Pischke${}^{\MakeLowercase c}$}
\date{\today}
\maketitle
\vspace*{-5mm}
\begin{center}
{\scriptsize 
${}^a$ Department of Mathematics, Technische Universit\"at Darmstadt,\\
Schlossgartenstra\ss{}e 7, 64289 Darmstadt, Germany,\\ 
${}^b$ School of Electronic Engineering and Computer Science, Queen Mary University of London,\\
Mile End Road, London, E1 4NS, United Kingdom,\\
${}^c$ Department of Computer Science, University of Bath,\\
Claverton Down, Bath, BA2 7AY, United Kingdom,\\
E-mails: neri@mathematik.tu-darmstadt.de, p.oliva@qmul.ac.uk, nnp39@bath.ac.uk}
\end{center}

\begin{abstract}
Over extended systems of finite type arithmetic, we utilize a formal representation of the outer measure to define a translation which allows for the systematic formalization of probabilistic statements. As a main result, this translation gives rise to novel \emph{probabilistic logical metatheorems} in the style of proof mining, guaranteeing the extractability of computable bounds from (non-effective) proofs of probabilistic existence statements. We further show how the set-theoretically false principle of uniform boundedness due to Kohlenbach can be used to replicate logically strong continuity properties of probability measures in the context of these bound extraction theorems in a tame way, i.e.\ without affecting the computational complexity of the resulting bounds in question, all the while guaranteeing the validity of those bounds even over finitely additive probability spaces. This in particular provides a formal perspective on the elimination of the principle of $\sigma$-additivity during bound extraction, as previously only observed ad hoc in the practice of proof mining. In that context, we for the first time provide a proof-theoretic treatment of higher-type uniform boundedness principles and related contra-collection principles via Kohlenbach's monotone variant of G\"odel's functional interpretation, which is of independent interest. All together, these new metatheorems provide a systematic proof-theoretic approach towards extracting various types of quantitative information for probabilistic theorems considered in the literature, justifying a range of recent applications to probability theory and stochastic optimization. This paper represents a major logical contribution to a recent advance of bringing the methods of proof mining to bear on probability theory, significantly extending previous work by the first and third author [Forum Math.\ Sigma, 13, e187 (2025)] in that direction.
\end{abstract}

\maketitle 

\section{Introduction}

\subsection{Motivation and background}

The nature and description of the computational content of a mathematical theorem is one of the fundamental questions in proof theory, and the development of tools and frameworks to exhibit and categorize this content is a key motivation behind many of the modern developments in the field. In the present paper, we are concerned with the computational content of theorems that are of a probabilistic nature, i.e.\ which, instead of making absolute statements, make \emph{almost sure} statements.

Indeed, many theorems from probability theory and statistics, including their most striking results, often take the form of almost sure theorems, asserting that a given property holds with probability one. Examples range from convergence results for sampling-based algorithms in statistics and operations research to zero-one laws, which state that a property can only hold almost always or almost never, with nothing in between. Beyond their qualitative content, such theorems, each in their own right, can clearly be made to offer a related computational challenge. This challenge commonly takes the form of asking whether they can be strengthened quantitatively, showing that an approximate version of the property in question holds up to a chosen degree of probability. This transforms what might seem like a computationally empty statement, that is, the property in question holding with probability one, into a statement with a meaningful computational interpretation, such that the complexity of the related computational information yields a sensible measure of the internal complexity of those statements and, in particular, their proofs.

To illustrate the above with a more precise example, we turn to almost sure convergence results, perhaps the most central and prevalent of such examples due to their ubiquity in statistics and stochastic approximation. Concretely, if $(X_n)_{n \in \mathbb{N}}$ is a stochastic process on a probability space $(\Omega, \mathcal{F}, \PP)$, one can show that the probabilistic statement
\[
\mbox{``$(X_n)_{n \in \mathbb{N}}$ converges with probability one''}
\]
is equivalent to (see e.g.\ \cite{NeriPowell2025}) the property
\[
\forall \varepsilon, \lambda >0\ \exists N \in \NN \left(\PP(\forall i,j \ge N\left(|X_i-X_j|\leq \varepsilon\right))>1-\lambda\right).
\]
A function $\Phi(\varepsilon, \lambda)$ providing a bound (and hence a witness) for the $N$ in the above expression is exactly what is isolated in the probability-theoretic literature as a \emph{rate of almost sure convergence}. From a logical point of view, analogous to rates of convergence for deterministic sequences, such rates are generally only effectively derivable in special circumstances, such as (but not limited to) semi-constructive contexts (see e.g.\ the discussion in \cite{Kohlenbach2008}). Correspondingly, there are further equivalent phrasings of almost sure convergence which give rise to quantitatively weaker readings, such as \emph{metastable rates of almost sure convergence} introduced by Avigad, Dean, and Rute \cite{AvigadDeanRute2012} (extending the work of Tao \cite{Tao2008b}), that is, functions $\Phi(\varepsilon, \lambda, g)$ bounding the $N$ in the statement
\[
\forall \varepsilon, \lambda >0\ \forall g:\mathbb{N}\to\mathbb{N}\ \exists N \in \NN \left(\PP(\exists n\leq N\ \forall i,j \in [n;n+g(n)]\left(|X_i-X_j|\leq\varepsilon\right))>1-\lambda\right).
\]
These can be seen as probabilistic generalizations of ordinary rates of metastability in the sense of Tao \cite{Tao2008}, which over the years have become a central way to assign a computational meaning to deterministic convergence statements in the absence of computable rates of convergence. Indeed, contrary to the above rates of almost sure convergence, these metastable rates (and also more uniform variants as discussed in detail later) are widely available for large classes of classically proven almost sure convergence theorems (as will, in fact, be formally justified for the first time in this paper) and they have been, and continue to be, obtained for major theorems such as the dominated convergence theorem \cite{AvigadDeanRute2012}, the martingale convergence theorem \cite{NeriPowell2025} and the Robbins-Siegmund theorem \cite{NeriPowell2024}, as well as further applied results from stochastic approximation theory \cite{NeriPischkePowell2026}, among others. 

Such computational considerations are certainly not isolated to convergence theorems, with further examples being discussed later on. The fundamental question of the present paper is then how such probabilistic statements can be formally recognized in the context of logical systems that facilitate the extraction of computational content from proofs, and in particular how the extractability of computational information commonly associated with such statements, as illustrated above, can be guaranteed therein.

\subsection{The proof mining program and probability theory}

The approach of this paper is based on the formal systems and methods employed in the proof mining program. This research program originates with the work of Kohlenbach in the 1990s, systematizing and expanding the ideas of Kreisel's unwinding of proofs \cite{Kreisel1951,Kreisel1952}, and aims at providing the computational content of theorems from mathematics by analysing their (prima facie) non-effective proofs as they are found in the usual literature. While applied in nature, this enterprise is firmly logically supported and substantiated through a range of proof-theoretic results called \emph{logical metatheorems}, mainly based on central proof-theoretic machinery like  G\"odel's functional interpretation \cite{Goedel1958}, Kreisel's modified realizability interpretation \cite{Kreisel1959}, as well as Howard's majorizability \cite{Howard1973}, and their variants. Concretely, the logical metatheorems used in proof mining are theorems associated with a logical system at hand such that
\begin{enumerate}
\item the system is suitably designed so that it allows for the formalization of large classes of objects and proofs from the respective area of application,
\item the associated logical metatheorem then guarantees that from large classes of proofs carried out in this system, potentially involving a wide range of non-computational principles, one can extract effective and highly uniform computational information for the respective theorem.
\end{enumerate}
Substantiated and supported by these metatheorems, proof mining has now led to hundreds of new applications in core mathematics. We refer to the central monograph \cite{Kohlenbach2008} for a comprehensive overview of the field, as well as to the surveys \cite{Kohlenbach2019,KohlenbachOliva2003} for further references and discussions.

This logical machinery was recently extended to probability theory by the first and third author in \cite{NeriPischke2025}, where corresponding systems and logical metatheorems on bound extraction were presented. While discussed in more detail later on, we already want to highlight here that these systems follow an abstract approach to both the sample and event space, treating them in particular as bounded spaces (we refer to \cite{NeriPischke2025} for further discussions on that point, which is of fundamental importance to this enterprise), and that crucially, the main system places strong emphasis on so-called probability contents, that is finitely additive probability measures (see in particular \cite{RaoRao1983}), as the right underlying notion for extending proof mining methods to this area. Concretely, probability contents are functions $\PP:\Algebra\to [0,1]$ on an algebra of sets $\Algebra$ over a base set $\Omega$ (that is $\emptyset\in \Algebra$ and $\Algebra$ is closed under complements and finite unions) which satisfy $\PP(\Omega)=1$ and are additive, i.e.\ $\PP(A\cup B)=\PP(A)+\PP(B)$ for disjoint $A,B\in\Algebra$. An algebra $\Algebra$ becomes a $\sigma$-algebra if it is closed under countable unions and over such a $\sigma$-algebra, $\PP$ becomes a probability measure if it is additionally $\sigma$-additive, i.e.\ $\PP\left(\bigcup_{i\in\mathbb{N}}A_i\right)=\sum_{i\in\mathbb{N}}\PP(A_i)$ for a disjoint collection $(A_i)_{i\in\mathbb{N}}\subseteq \Algebra$. 

In parallel to and since the work \cite{NeriPischke2025}, various case studies of proof mining in probability theory have been carried out, which include, beyond those related to probabilistic convergence theorems mentioned before, also work on strong laws of large numbers \cite{Neri2024,Neri2023,Neri2025a} and ergodic theory \cite{Neri2025b}, on stochastic optimization \cite{NeriPischkePowell2025a,Pischke2025,Pischke2026,PischkePowell2024} and on zero-one laws \cite{PowellWan2025}. 

\subsection{The contributions of this paper}

The main system introduced in \cite{NeriPischke2025} ``only'' axiomatises probability contents. It is already argued therein that this restriction does not appear to be a real limitation on the applicability of the system, as it seemed sufficient for essentially all applications of that approach developed at the time of its release. These were however quite limited, and further developments (in particular the case studies mentioned above) quickly showed that the applicability of this system is a subtle issue, not least through the system's restricted vocabulary, which required substantial additional work tailored ad hoc to each such case study. The fundamental question thereby arose of how theorems and their proofs, coming from ordinary probability theory, can be systematically recognized in the language of that system to render these metatheorems applicable and to hence guarantee the extractability of related quantitative information. Or, in other words: How can the property (1) of logical metatheorems highlighted above be justified for the systems presented in \cite{NeriPischke2025}?

In this paper, we provide an answer to this question by presenting a systematic approach to the general task of assigning a computational meaning, both classically and constructively, to probabilistic statements which covers the previously discussed examples, as well as many more, in a uniform manner, and moreover facilitates a similarly systematic approach for the extraction of that content. In particular, the results of this paper substantiate the previously mentioned case studies released since as applications of this general logical approach initiated in \cite{NeriPischke2025}. Hence, while first proposed \cite{NeriPischke2025}, it is only through the present work that we can now really recognize the systematic applicability and suitability of this approach via the tame theory of probability contents.

More concretely, the present paper proceeds as follows: In the language of those systems from \cite{NeriPischke2025}, we utilize a formal representation of the outer measure to define a translation which allows for the systematic formalization of probabilistic statements. In its most special case, this translation associates to any formula $\varphi$ another formula $\PP[\varphi]=1$ which expresses that $\varphi$ holds almost surely. This translation can then be combined with the various proof-theoretic tools employed previously to establish the logical metatheorems for the systems for contents in \cite{NeriPischke2025} to associate with that derived statement $\PP[\varphi]=1$ a precise computational meaning.

In particular, this gives rise to dedicated new \emph{probabilistic} logical metatheorems for the extraction of computational information from proofs of probabilistic theorems, extending the \emph{basic} logical metatheorems previously derived in \cite{NeriPischke2025}. In similarity to the bound extraction theorems from \cite{NeriPischke2025}, also these new metatheorems presented here guarantee a high degree of uniformity of the extractable bounds, in the sense that they will be independent of all parameters relating to the underlying probability space, and in particular of the measure. Establishing some of these new metatheorems is highly non-trivial, and makes crucial use of essentially all known techniques and devices from the proof mining program, in particular of extensions of a set-theoretically false principle of uniform boundedness introduced by Kohlenbach \cite{Kohlenbach2006} as further discussed below. Throughout, we also simultaneously focus both on metatheorems tailored to classical as well as constructive reasoning, the latter of which are considered in the context of probability theory here for the first time and are derived by following the work by Gerhardy and Kohlenbach \cite{GerhardyKohlenbach2006}. We will then in particular discuss how these probabilistic metatheorems give rise to various well-known quantitative notions from probability theory, covering the examples and case studies mentioned before, and thereby illustrating that the new approach of this paper is in particular practically natural and useful. 

All of these considerations are, in particular, made over the systems for probability contents which guarantees that the resulting extracted bounds will be valid in that context. This has particular consequences in the light of our second set of contributions, where we provide an algebra for expressions of the form $\PP[\varphi]=1$ based on the logical structure, and in particular the quantifier structure, of $\varphi$. In that context, under the use of a principle of uniform boundedness due to Kohlenbach \cite{Kohlenbach1999,Kohlenbach2006}, we show that the quantifiers internal to $\varphi$ can be, in some cases, ``pulled out'' of the expression $\PP[\varphi]=1$. These prenexations of probabilistic statements are intimately related to continuity principles of the set-function, and in general only valid in the context of probability measures. However, as the aforementioned uniform boundedness principle is admissible in the context of the usual logical metatheorems employed in proof mining (such as those developed in this paper), even though it is generally set-theoretically false, these prenexation results may be used in proofs, while still guaranteeing that the resulting bounds are true for general probability contents. In that way, we illustrate how uniform boundedness principles can be employed to use facts about probability measures in proofs for probability contents, highlighting great potential importance of such principles for the future of proof mining and probability. For that reason, all the logical metatheorems discussed in the present paper accommodate such uniform boundedness principles explicitly. In that context, we in particular for the first time provide a proof-theoretic treatment of these uniform boundedness principles in higher types, which allow for such prenexation principles even for uncountable quantifiers, which in that way are false even for probability measures. To accomodate such results also in a semi-constructive context, we not only rely on the principle of uniform boundedness, but also on the dual so-called ``contra-collection'' principle, which prominently features in the context of the bounded functional interpretation of the second author and Ferreira \cite{FerreiraOliva2005}, and is treated here for the first time in the context of the monotone functional interpretation. These results, extending the reach of the monotone functional interpretation, are in particular of idenpendent interest even outside applications to probability theory. Moreover, as hinted on before, we want to emphasize that these principles feature crucially in our approach to some of the new probabilistic metatheorems. Lastly, we provide a type of conservativity result for the systems considered here regarding the principle of $\sigma$-additivity. Concretely, we show that in a classical context, $\sigma$-additivity can actually be derived from the uniform boundedness principles, which provides a formal perspective on the elimination of $\sigma$-additivity during bound extraction in proof mining, as was previously only observed ad hoc in practice. All these results in particular further affirm the suitability of probability contents, not measures, as the fundamental measure-theoretic notion in proof mining and probability theory, as mentioned before.

\subsection{Related and future work}

Beyond previously mentioned contributions, we of course want to highlight that present work will allow for the development of many further applications of proof mining to probability theory. Indeed, some of the previously mentioned case studies (see in particular \cite{Neri2025a,Neri2025b,NeriPischkePowell2025a,NeriPischkePowell2026,Pischke2025,Pischke2026}) were developed in tandem with the present work, where the logical perspectives developed here were crucial in guiding the formulation of quantitative variants of probabilistic statements and the extraction of the respective computational content. In the context of \cite{NeriPischkePowell2026} in particular, the present logical approach was crucial for developing the associated notion of a finitary martingale, as well as the associated quantitative theory. Further applications stemming from the present work, in particular to stochastic optimization, are already underway. Aside from those concrete applications, we want to mention two avenues for future work stemming from this paper.

At first, we want to highlight that the boundedness of the sample and event space as well as their associated uniform boundedness and contra-collection principles play a crucial role in this work. While not discussed here in any great detail, since we generally follow the formal setup of \cite{NeriPischke2025}, this highlights the very interesting question for future work of whether the present topic can benefit from the use of the \emph{bounded functional interpretation} due to Ferreira and the second author \cite{FerreiraOliva2005} (see also \cite{EngraciaFerreira2020,Ferreira2009}), a proof interpretation inspired by the paradigm of the monotone functional interpretation of Kohlenbach \cite{Kohlenbach1996b} (and conceptually already going back to \cite{Kohlenbach1992}) but where such aspects are dealt with more fundamentally, in particular the contra-collection principle.

Similarly, while also not discussed here in any further detail, we want to mention that while the present results are formulated as tailored to the proof-theoretic framework from \cite{NeriPischke2025}, various ideas of the present paper could perhaps also be used in the context of model-theoretic approaches to uniform bounds, in particular as, e.g., articulated in the context of positive bounded logic \cite{AvigadIovino2013,DuenezIovino2017} (see also \cite{NeriPischke2025} for a more detailed discussion on the differences and similarities of the present proof-theoretic methodology to such results from model theory). In particular, the general translation induced by the outer measure, mapping formulas to probabilistic formulas, could be used in such a context to broaden the applicability of the related model-theoretic results on the existence of uniform bounds, as it can be phrased in the language of positive bounded logic. In any case however, we want to highlight that these model-theoretic results can only guarantee existence of such uniform bounds and, in difference to the proof-theoretic approach, do not offer considerations on complexity or an approach to extracting them from proofs.

Further, it would be very interesting to compare the present use of the outer measure to similar such uses in model theory (see e.g.\ \cite{Keisler1987}), and especially to probability quantifiers (see in particular \cite{Keisler1985}, among many others) or probabilistic logic (see in particular \cite{ScottKrauss1966}). However, we want to emphasize that the goal of the present paper is quite distinct from these and related works, and is particularly rooted in the perspective of (applied) proof theory, where such a notion is, to the best of our knowledge, employed for the first time in the present paper.

Lastly, it would be interesting to consider the connection between uniform boundedness principles and continuity theorems for probability measures, as developed in this paper, also from a model-theoretic perspective, in particular in light of the observation that uniform boundedness represents a sort-of proof-theoretic version of saturation for ultraproducts in model theory, as e.g.\ centrally discussed in \cite{GuenzelKohlenbach2016} where many such applications of saturation are shown to be captured from a proof-theoretic perspective via uniform boundedness. In particular, it would be interesting to investigate the relationship between the present uses of uniform boundedness in probability theory and the construction of the Loeb measure in nonstandard analysis (see the seminal work \cite{Loeb1975} as well as \cite{DuenezIovino2017}), where saturation is similarly employed to extend finitely additive measures to $\sigma$-additive ones, and hence used to establish continuity properties. Investigations on this last aspect are already underway.

\subsection{The outline of the paper}

The paper is now organised as follows: Section \ref{sec:prelim} provides the necessary details on the systems from \cite{NeriPischke2025} as well as the central proof-theoretic devices used in this paper and the associated \emph{basic} logical metatheorems, in particular their extensions to incorporate higher-type uniform boundedness and contra-collection. Section \ref{sec:trans} motivates and introduces our general translation based on a formal representation of the outer measure in the language of the systems from \cite{NeriPischke2025}. Section \ref{sec:algebra} provides the various results on the algebra of our translation, including the prenexation results based on the principles of uniform boundedness and contra-collection. Moreover, in that section we highlight in particular how the uniform boundedness principles provide a formal perspective on the elimination of $\sigma$-additivity from proofs in the course of a proof-theoretic analysis. Section \ref{sec:meta} then provides the general discussion of the computational content of probabilistic statements as suggested by our translation, resulting in metatheorems specially tailored to these probabilistic circumstances. The final Section \ref{sec:examples} discusses how these respective probabilistic logical metatheorems proved earlier give rise to various natural quantitative notions from the literature, and beyond, relating them in particular to the previously mentioned case studies.

\section{The formal framework}\label{sec:prelim}

The present section serves effectively as preliminaries to the present paper, laying out the essential details of the formal approach towards proof mining and probability theory developed in \cite{NeriPischke2025}, but also develops crucial new material like further variations arising by considering semi-constructive versions and extensions via uniform boundedness and contra-collection principles, as required here. While we will only rely on rather little probability theory, which we will introduce as needed, we refer to \cite{Klenke2008} for a general reference on probability theory and to \cite{RaoRao1983} for the central reference for the theory of contents.

\subsection{Systems for probability contents}

We begin with the essential details of the formal approach towards a system geared for bound extractions from proofs in the theory of probability contents, as introduced in \cite{NeriPischke2025}. The main system defined for that purpose in \cite{NeriPischke2025} is the system $\mathcal{F}^\omega[\PP]$ which, as common for logical systems employed in the proof mining program, arises as an extension of systems of finite type arithmetic.

Concretely, as in \cite{Kohlenbach2008} and \cite{Troelstra1973}, let $\WEPAomega$ be a weakly extensional variant of Peano arithmetic over all \emph{finite types} $T$ defined by
\[
0\in T \quad \mbox{and} \quad \mbox{if $\rho,\tau\in T$ then $\rho(\tau)\in T$}.
\]
We denote the \emph{pure types} using natural numbers, by recursively setting $n+1:=0(n)$. The only primitive relation symbol of $\WEPAomega$ is $=_0$, representing equality at type $0$. In particular, equality at higher types is defined via the abbreviation
\[
x^{\tau(\xi)}=_{\tau(\xi)}y^{\tau(\xi)} \;\; := \;\; \forall z^{\xi}\left( xz=_\tau yz\right)
\] 
and for higher type inequalities, we similarly set $x^{\tau(\xi)}\geq_{\tau(\xi)}y^{\tau(\xi)}:=\forall z^{\xi}\left( xz\geq_\tau yz\right)$.\footnote{Note that $\geq_0$ is definable from the only primitive relation $=_0$ using cut-off subtraction.} Besides equality at type $0$, the system contains constants for successor and zero, constants for the combinators of Sch\"onfinkel \cite{Schoenfinkel1924} allowing for lambda abstraction, and lastly, constants allowing for simultaneous primitive recursion in the sense of G\"odel \cite{Goedel1958} and Hilbert \cite{Hilbert1926}, all with their corresponding natural defining axioms. We do not spell out these constants or axioms as they do not feature explicitly in this work. Instead, we refer the reader to \cite{Kohlenbach2008}.

Furthermore, as opposed to the full axiom of extensionality, the system crucially only contains the following \emph{quantifier-free extensionality rule}
\[
  \frac{\varphi_0\to s=_\rho t}{\varphi_0\to r[s/x^\rho]=_\tau r[t/x^\rho]}\tag{$\QFER$}
\]
where $\varphi_0$ is a quantifier-free formula, $s$ and $t$ are terms of type $\rho$ and $r$ is a term of type $\tau$. We refer the reader to \cite{Kohlenbach2008} for further discussions on the (necessary) restrictions on extensionally in systems amenable to program extraction.

Over $\WEPAomega$ and its extensions, we rely on a representation of real numbers as fast converging Cauchy sequences of rational numbers with a fixed Cauchy modulus $2^{-n}$ (see Chapter 4 in \cite{Kohlenbach2008} for details), that is as objects of type $1=0(0)$. We will not need any precise details of this representation and just recall that in that context, the usual arithmetical operations like $+_\mathbb{R}$, $\cdot_\mathbb{R}$, $\vert\cdot\vert_\mathbb{R}$, etc., are primitive recursively definable through closed terms and the relations $=_\mathbb{R}$ and $<_\mathbb{R}$, etc., are representable via formulas of the underlying language. Furthermore, these relations are not decidable but are given by $\Pi^0_1$- and $\Sigma^0_1$-formulas, respectively. We generally drop subscripts from arithmetical operations and often also reduce typing in the notation throughout for readability.

Instead of relying on representations of probability content spaces in finite type arithmetic, which would come with all sorts of additional questions and complications, the central approach used in \cite{NeriPischke2025} is to represent the sample and event space using new, additional abstract types, following the central modern paradigm of proof mining developed by Kohlenbach, starting with the seminal work \cite{Kohlenbach2005}. Concretely, consider now the extended set of finite types $T^{\Omega,\Algebra}$ by
\[
0,\Omega,\Algebra\in T^{\Omega,\Algebra} \quad \mbox{and} \quad \mbox{if $\rho,\tau\in T^{\Omega,\Algebra}$ then $\rho(\tau)\in T^{\Omega,\Algebra}$},
\]
with the types $\Omega$ and $\Algebra$ serving as abstract representations of the sample and event space, respectively. The language of the system $\mathcal{F}^\omega[\PP]$ for a probability content over an algebra of events now arises from $\WEPAomega$ (lifted to this extended set of types) by adding the constants $\mathrm{eq}$ with type $0(\Omega)(\Omega)$ and $\in$ with type $0(\Algebra)(\Omega)$ for the sample space, $\cup$ with type $\Algebra(\Algebra)(\Algebra)$, $(\cdot)^c$ with type $\Algebra(\Algebra)$ and $\emptyset$ with type $\Algebra$ for the event space, and lastly $\PP$ of type $1(\Algebra)$ for the probability content.\footnote{The system $\mathcal{F}^\omega[\PP]$ as defined in \cite{NeriPischke2025} also contains a constant $c_\Omega$ of type $\Omega$, witnessing the non-emptyness of $\Omega$, and which in \cite{NeriPischke2025} serves a rather technical purpose in the proof of the associated logical metatheorems that plays no further role in this paper. We hence omit any further discussion thereof.} The axioms of $\mathcal{F}^\omega[\PP]$ then specify that
\[
x^\Omega=_\Omega y^\Omega \; :\equiv \; \mathrm{eq}xy=_00
\]
is an equivalence relation representing equality on $\Omega$, that the abstract element relation $\in$ behaves as expected relative to the set-theoretic operations $\cup$ and $(\cdot)^c$ for union and complement as well as the constant $\emptyset$ for the empty set and lastly that $\PP$ is probability content on the algebra represented by the type $\Algebra$.\footnote{While most axioms of probability contents are easily defined via computationally empty universal axioms, recognizing that the law of disjoint unions can be specified in a way that is admissible in the context of bound extraction theorems is not immediately straightforward and requires considerations on the uniform boundedness of the new abstract types. Further, the system $\mathcal{F}^\omega[\PP]$ as defined in \cite{NeriPischke2025} actually replaces the law of disjoint unions with the inclusion exclusion principle together with the monotonicity of the content for reasons of practicality in concrete formalizations. In fact, we want to highlight a necessary correction to \cite{NeriPischke2025} which misses the crucial (universal, and hence unproblematic) axiom that $\PP(\Omega)=_\mathbb{R}1$. While the precise axiomatization of $\PP$ therefore carries various subtleties with it, these do not feature in the present paper any further and so we simply refer the interested reader to the discussion in \cite{NeriPischke2025}.} As in \cite{NeriPischke2025}, we introduce appropriate abbreviations to make formal expressions less rigid: We typically write $A^c$ instead of $(A)^c$, $x\in A$ instead of $\in x A =_0 0$ and $x\not\in A$ instead of $\in x A\neq_0 0$. Also, we define $\Omega:=\emptyset^c$ and $A\cap B:=(A^c\cup B^c)^c$. Arbitrary finite unions and intersections are then definable from these operators via recursion, and we refer to \cite{NeriPischke2025} for further discussions. Equality on $\Algebra$ as well as inclusion are also derived notions which we obtain by setting 
\[
A =_\Algebra B \; :\equiv \; \forall x^\Omega(x\in A \leftrightarrow x \in B) \quad \text{and} \quad A\subseteq_\Algebra B :\equiv \forall x^\Omega \left(x \in A \rightarrow x \in B\right).
\]
Then, $=_\Algebra$ is provably an equivalence relation and $\subseteq_\Algebra$ is provably a partial order w.r.t.\ that equality. For the sole purpose of extending the inequality relation to all types, we define $x^\Omega\geq_\Omega y^\Omega := \PP(\Omega)\geq_\mathbb{R}\PP(\Omega)$ and $A^\Algebra\geq_\Algebra B^\Algebra:=\PP(A)\geq_\mathbb{R}\PP(B)$, following \cite{NeriPischke2025}.

The last addition to $\mathcal{F}^\omega[\PP]$ are two choice principles, the principle of quantifier-free choice $\QFAC$ as well as the principle of dependent choice $\DC$ (see e.g.\ Chapter 11 in \cite{Kohlenbach2008}), the latter of those giving access to full classical analysis.

\begin{remark}
In similarity to \cite{NeriPischke2025}, we remark that while the main system $\mathcal{F}^\omega[\PP]$, through $\DC$, contains a rather large amount of comprehension, this is just to highlight the potential strength of systems which are amenable to the present methods. All the while, as with \cite{NeriPischke2025}, no logical considerations made in this paper in any way depend on the presence of that comprehension, nor do they entail principles of such strength, and hence can be just as well developed over a range of weak fragments, such as those introduced in \cite{Kohlenbach1996a} based on the Grzegorczyk hierarchy \cite{Grzegorczyk1953}.
\end{remark}

\begin{remark}
The main system $\mathcal{F}^\omega[\PP]$ allows for various mathematically natural and proof-theoretically tame extensions to deal with various objects and notions from mathematical practice, e.g.\ random variables and their Lebesgue integrals as discussed in \cite{NeriPischke2025}. These extensions then, in particular, potentially populate the term structure of the type $\Algebra$ with more non-trivial ground terms. While we do not discuss these extensions here, and just phrase everything for the simplest system $\mathcal{F}^\omega[\PP]$, it should be noted that essentially all considerations of this paper extend mutatis mutandis to these broader settings.
\end{remark}

As we also consider semi-constructive variants of our results throughout, we will also rely on an intuitionistic variant of $\mathcal{F}^\omega[\PP]$. This variant, denoted by $\mathcal{F}^\omega_i[\PP]$, will be based on $\WEHAomega$, that is $\WEPAomega$ with the law of excluded middle removed, and additionally includes the principles the full axiom of choice $\AC$, Markov's principle $\MARK$ and the independence of premise principle for universal formulas $\IPU$ (see e.g.\ Chapters 5 and 8 in \cite{Kohlenbach2008}) instead of $\QFAC$ and $\DC$. Aside from these, $\mathcal{F}^\omega_i[\PP]$ arises from $\WEHAomega$ in exactly the same way as $\mathcal{F}^\omega[\PP]$ arises from $\WEPAomega$, that is with the same additional types, constants and axioms.

\subsection{Uniform boundedness, contra-collection and metatheorems on program extraction}

We now discuss the main results for the systems introduced previously, that is the logical metatheorems on bound extraction. While the basic versions of the classical variants of these results were established in \cite{NeriPischke2025}, the semi-constructive variants are for the first time stated in this paper. These, however, arise through a rather immediate combination of the constructions from \cite{NeriPischke2025} with the considerations on semi-constructive systems given by Gerhardy and Kohlenbach in \cite{GerhardyKohlenbach2006}, so that we are content with proof sketches for the resulting metatheorems tailored to probability theory.

Slightly more extensive modifications are required to incorporate Kohlenbach's uniform boundedness principle, as considered in \cite{Kohlenbach2006} (see also already the earlier works \cite{Kohlenbach1996b,Kohlenbach1998,Kohlenbach1999} for variants of this principle not phrased in the context of abstract types). While our arguments here follow the previous works \cite{GuenzelKohlenbach2016,Kohlenbach1998,Kohlenbach2006}, we actually consider this principle here for the first time in types higher than the base type $0$ in the context of the monotone functional interpretation. This requires some care in adapting the previous arguments, and we therefore provide additional details.

The most extensive modifications arise through our treatment of the so-called \emph{principle of contra-collection}, an abstract principle related to weak K\H{o}nig's lemma (arising, classically, as the contraposition of the associated uniform boundedness principle, itself an abstract extension of the FAN principle from intuitionism), isolated in this extended form for the first time in the context of the bounded functional interpretation \cite{FerreiraOliva2005}. Also this principle is treated here for the first time in the context of the monotone functional interpretation (outside of classical contexts), and although we only consider its variant for the base type $0$, together with some extensions to type $1$, this principle requires various non-trivial arguments which we provide in full detail.

The first key logical tool used in the context of the logical metatheorems, and in the rest of the paper for that matter, is G\"odel's functional (Dialectica) interpretation.

\begin{definition}[\cite{Goedel1958}]
The Dialectica interpretation $\varphi^D=\exists\underline{x}\forall\underline{y} \varphi_D(\underline{x},\underline{y})$ of a formula $\varphi$ in the language of $\mathcal{F}^\omega[\PP]$ (or its variants) is defined via the following recursion on the structure of the formula:
\begin{enumerate}
\item $\varphi^D:=\varphi_D:=\varphi$ for $\varphi$ being a prime formula.
\end{enumerate}
If $\varphi^D=\exists\underline{x}\forall\underline{y} \varphi_D(\underline{x},\underline{y})$ and $\psi^D=\exists\underline{u}\forall\underline{v} \psi_D(\underline{u},\underline{v})$, we set
\begin{enumerate}
\setcounter{enumi}{1}
\item $(\varphi\land \psi)^D:=\exists\underline{x},\underline{u}\forall\underline{y},\underline{v}(\varphi\land \psi)_D$\\ where $(\varphi\land \psi)_D(\underline{x},\underline{u},\underline{y},\underline{v}):=\varphi_D(\underline{x},\underline{y})\land \psi_D(\underline{u},\underline{v})$,
\item $(\varphi\lor \psi)^D:=\exists z^0,\underline{x},\underline{u}\forall\underline{y},\underline{v}(\varphi\lor \psi)_D$\\ where $(\varphi\lor \psi)_D(z^0,\underline{x},\underline{u},\underline{y},\underline{v}):=(z=0\rightarrow \varphi_D(\underline{x},\underline{y}))\land (z\neq 0\rightarrow \psi_D(\underline{u},\underline{v}))$,
\item $(\varphi\rightarrow \psi)^D:=\exists\underline{U},\underline{Y}\forall\underline{x},\underline{v}(\varphi\rightarrow \psi)_D$\\ where $(\varphi\rightarrow \psi)_D(\underline{U},\underline{Y},\underline{x},\underline{v}):=\varphi_D(\underline{x},\underline{Y}\underline{x}\underline{v})\to \psi_D(\underline{U}\underline{x},\underline{v})$,
\item $(\exists z^\tau \varphi(z))^D:=\exists z,\underline{x}\forall\underline{y}(\exists z^\tau \varphi(z))_D$\\ where $(\exists z^\tau \varphi(z))_D(z,\underline{x},\underline{y}):=\varphi_D(\underline{x},\underline{y},z)$,
\item $(\forall z^\tau \varphi(z))^D:=\exists\underline{X}\forall z,\underline{y}(\forall z^\tau \varphi(z))_D$\\ where $(\forall z^\tau \varphi(z))_D(\underline{X},z,\underline{y}):=\varphi_D(\underline{X}z,\underline{y},z)$.
\end{enumerate}
\end{definition}

Beyond this, in the classical context we will also rely on a negative translation to transfer classical to semi-constructive proofs. As in previous works, we consider the following rather minimal variant due to Kuroda.

\begin{definition}[\cite{Kuroda1951}]
The (Kuroda) negative translation $\varphi'$ of $\varphi$ is defined as $\varphi':=\neg\neg \varphi^*$, with $\varphi^*$ defined recursively via
\begin{enumerate}
\item $\varphi^*:=\varphi$, if $\varphi$ is atomic,
\item $(\varphi\circ \psi)^*:=\varphi^*\circ\psi^*$, for $\circ\in\{\land,\lor,\rightarrow\}$,
\item $(\exists x^\tau\varphi)^*:=\exists x^\tau \varphi^*$,
\item $(\forall x^\tau\varphi)^*:=\forall x^\tau\neg\neg\varphi^*$.
\end{enumerate}
\end{definition}

The last tool that will feature in the present paper is Kreisel's modified realizability interpretation, which can be employed in semi-constructive contexts \emph{without} Markov's principle.

\begin{definition}[\cite{Kreisel1959,Kreisel1962}]
For any formula $\varphi$ in the language of $\mathcal{F}^\omega[\PP]$ (or its variants), we define its modified realizability interpretation $\underline{x}\,mr\,\varphi$ by recursion on the structure of $\varphi$:
\begin{enumerate}
\item $\langle\rangle\,mr\,\varphi:=\varphi$ for a prime formula $\varphi$ where $\langle\rangle$ is the empty tuple.
\end{enumerate}
Further, if $\underline{x}\,mr\,\varphi$ and $\underline{y}\,mr\,\psi$ are the modified realizability interpretations of $\varphi$ and $\psi$, respectively, then:
\begin{enumerate}
\setcounter{enumi}{1}
\item $\underline{x},\underline{y}\,mr\,(\varphi\land \psi):=\underline{x}\,mr\,\varphi\land\underline{y}\,mr\, \psi$,
\item $z^0,\underline{x},\underline{y}\,mr\,(\varphi\lor \psi):=(z=_00\rightarrow \underline{x}\,mr\,\varphi)\land(z\neq_00\rightarrow\underline{y}\,mr\,\psi)$,
\item $\underline{Y}\,mr\, (\varphi\rightarrow \psi):=\forall\underline{x}(\underline{x}\,mr\, \varphi\rightarrow\underline{Y}\underline{x}\,mr\,\psi)$,
\item $\underline{X}\,mr\,\forall w^\rho \varphi(w):=\forall w^\rho(\underline{X}w\,mr\, \varphi(w))$,
\item $z^\rho,\underline{x}\,mr\,\exists w^\rho \varphi(w):=\underline{x}\,mr\,\varphi(z)$.
\end{enumerate}
\end{definition}

Next to the above proof interpretations, we will also crucially rely on the notion of \emph{majorizability}, originally introduced by Howard \cite{Howard1973} and later extended by Bezem \cite{Bezem1985} to the notion of \emph{strong majorizability}. While majorizability features crucially throughout, it is in particular also the latter notion which allows our classical system $\mathcal{F}^\omega[\PP]$ to include the strong dependent choice principle $\DC$, which is interpreted using bar-recursion for which the structure of all strongly majorizable functionals provides a model. Moreover, it generally allows us to include a very strong uniform boundedness principle later (see Definitions \ref{def:classUB} and \ref{def:constrUB}). We here consider an extension of the respective notion of majorizability to the abstract types $\Omega$ and $S$ for the use in the theory of probability contents in \cite{NeriPischke2025}. In that context, as in the case of the first logical metatheorems with abstract types considered in \cite{GerhardyKohlenbach2008,Kohlenbach2005}, majorants of functionals with types $\tau\in T^{\Omega,S}$ will have type $\widehat{\tau}\in T$, defined recursively via $\widehat{0}:=0$, $\widehat{\Omega}:=0$, $\widehat{S}:=0$ and $\widehat{\tau(\xi)}:=\widehat{\tau}(\widehat{\xi})$. We then introduce our extension of Bezem's notion of strong majorizability  semantically as follows, tied to the structure of all strongly majorizable functionals:

\begin{definition}[\cite{NeriPischke2025}, following \cite{GerhardyKohlenbach2008,Kohlenbach2005} and \cite{Bezem1985}]
Let $\Omega$ be a non-empty set, $S\subseteq 2^\Omega$ be an algebra and $\PP$ be a probability content on $S$. The structure $\mathcal{M}^{\omega,\Omega,S}$ and the majorizability relation $\gtrsim_\rho$ are defined by
\[
\begin{cases}
\mathcal{M}_0:=\mathbb{N}, n\gtrsim_0 m:=n,m\in\mathbb{N}\land n\geq m,\\
\mathcal{M}_\Omega:= \Omega, n\gtrsim_\Omega x:= n\in \mathbb{N},x\in \Omega\land n\geq \PP(\Omega),\\
\mathcal{M}_{S}:= S, n\gtrsim_{S} A:= n\in\mathbb{N},A\in S\land n\geq \PP(A),\\
x^*\gtrsim_{\tau(\xi)}x:=x^*\in \mathcal{M}_{\widehat{\tau}}^{\mathcal{M}_{\widehat{\xi}}}\land x\in \mathcal{M}_\tau^{\mathcal{M}_\xi}\\
\phantom{x^*\gtrsim_{\tau(\xi)}x:=}\land\forall y^*\in \mathcal{M}_{\widehat{\xi}},y\in \mathcal{M}_\xi(y^*\gtrsim_\xi y\rightarrow x^*y^*\gtrsim_\tau xy)\\
\phantom{x^*\gtrsim_{\tau(\xi)}x:=}\land\forall y^*,y\in \mathcal{M}_{\widehat{\xi}}(y^*\gtrsim_{\widehat{\xi}}y\rightarrow x^*y^*\gtrsim_{\widehat{\tau}}x^*y),\\
\mathcal{M}_{\tau(\xi)}:=\left\{x\in \mathcal{M}_\tau^{\mathcal{M}_\xi}\mid \exists x^*\in \mathcal{M}^{\mathcal{M}_{\widehat{\xi}}}_{\widehat{\tau}}:x^*\gtrsim_{\tau(\xi)}x\right\}.
\end{cases}
\]
\end{definition}

The other main structure featuring in the metatheorems is the structure of all set-theoretic functionals:

\begin{definition}
Let $\Omega$ be a non-empty set, $S\subseteq 2^\Omega$ be an algebra and $\PP$ be a probability content on $S$. The structure $\mathcal{S}^{\omega,\Omega,S}$ is defined via $\mathcal{S}_0:=\mathbb{N}$, $\mathcal{S}_\Omega:= \Omega$, $\mathcal{S}_{S}:=S$ and $\mathcal{S}_{\tau(\xi)}:=\mathcal{S}_{\tau}^{\mathcal{S}_{\xi}}$.
\end{definition}

The logical metatheorems, classical and constructive, then arise via a combination of a soundness result for any of the previous proof interpretations, by which one extracts realizing terms with types from $T^{\Omega,S}$ for a given $\forall\exists$-theorem, together with a subsequent majorization of these realizers, resulting in bounding terms with types from $T$ which are validated in a model based on $\mathcal{M}^{\omega,\Omega,S}$. From this, one then can transfer the results back to $\mathcal{S}^{\omega,\Omega,S}$ if the types occurring in the theorem are ``low enough'', formally encapsulated by the notion of small and admissible types as defined in \cite{GuenzelKohlenbach2016} (derived from the same notion introduced in \cite{GerhardyKohlenbach2008,Kohlenbach2005} under a different name). This combination of a proof interpretation and majorization effectively amounts to an application of the monotone variants of the interpretations in question, that is of the monotone functional interpretation \cite{Kohlenbach1996b} (recall also the previous references) or the monotone modified realizability interpretation \cite{Kohlenbach1998}. However, following e.g.\ \cite{Kohlenbach2005}, we here prefer the presentation where these two processes are firmly disentangled.

Now, as mentioned previously, the crucial ``novelty'' of the bound extraction theorem for the system $\mathcal{F}^\omega[\PP]$ as stated herein, compared to the one presented in \cite{NeriPischke2025}, is the inclusion of a uniform boundedness principle for the types $\Omega$ and $S$. This principle, pioneered in \cite{Kohlenbach1996b,Kohlenbach1996a} by Kohlenbach, and then developed further substantially in \cite{GuenzelKohlenbach2016,Kohlenbach1998,Kohlenbach1999,Kohlenbach2006,Kohlenbach2008}, has rather strong and even set-theoretically false conclusions in general. Nevertheless, it can be used admissibly in the context of bound extraction theorems to derive set-theoretically correct and uniform bounds. In the present paper, as already discussed in the introduction, we will in particular show how this principle allows our systems to formally use contents and associated outer contents, to a rather large degree, as if they were probability measures and associated outer measures, all while staying admissible in the context of bound extraction theorems, so that rather strikingly the resulting theorems will nevertheless still be true for contents. Even further, we in fact allow for uniform boundedness \emph{in higher types} in this paper, which has particularly interesting implications in the context of probability theory. 

In any case, the uniform boundedness principles considered in this paper take the following form:\footnote{As apparent, while extended to higher types, the uniform boundedness principles considered in this paper are only defined for pure types, as this allows for a quite smooth development of the associated theory while  being sufficient for the mathematical considerations of this paper. It is however quite possible that the present considerations also extend to arbitrary higher types.}

\begin{definition}[extending \cite{Kohlenbach2006}]\label{def:classUB}
Let $\rho$ be a pure type. The principle $\UBOmega{\rho}$ is defined as
\[
\begin{cases}
\forall y^{\alpha(0)}(\forall k^0, x^\alpha,\underline{z}^{\underline{\beta}}\, \exists w^\rho\, \varphi_\exists(y,k,\min_\alpha(x,yk),\underline{z},w)\\
\qquad\to \exists \chi^{\rho(0)}\,\forall k^0, x^\alpha,\underline{z}^{\underline{\beta}}\, \exists w \le_\rho \chi k\,\varphi_\exists(y,k,\min_\alpha(x,yk),\underline{z},w))
\end{cases}
\]
where $\alpha$ is a type of the form $0(\sigma_k)\ldots(\sigma_1)$, each type in $\underline{\beta}$ is of the form $\Omega(\tau_m)\ldots(\tau_1)$ and $\varphi_\exists$ is an existential formula where the types of all quantifiers are admissible (see \cite{GuenzelKohlenbach2016} and also \cite{GerhardyKohlenbach2008,Kohlenbach2005}). The principle $\UBF{\rho}$ is defined analogously, using $\Algebra$ instead of $\Omega$ in the conditions on the types in the tuple $\underline{\beta}$.
\end{definition}

This formulation is indeed very general and, following \cite{GuenzelKohlenbach2016}, represents a carefully defined intensional version of the ``usual'' uniform boundedness principle 
\[
\begin{cases}
\forall y^{\alpha(0)}(\forall k^0, x\leq_\alpha yk,\underline{z}^{\underline{\beta}}\, \exists w^\rho\, \varphi_\exists(y,k,x,\underline{z},w)\\
\qquad\to \exists \chi^{\rho(0)}\,\forall k^0, x\leq_\alpha yk,\underline{z}^{\underline{\beta}}\, \exists w \le_\rho \chi k\,\varphi_\exists(y,k,x,\underline{z},w)).
\end{cases}
\]
While a necessary restriction in order to stay admissible in the context of bound extraction theorems, these two formulations coincide for sentences that are extensional (see \cite{GuenzelKohlenbach2016}, to which we refer for further discussions in that vein). In any way, previous considerations of uniform boundedness in the context of the monotone functional interpretation were confined to principles with $\rho=0$ (in the above notation). It should however be stressed that this was not really out of a technical necessity, but rather since all mathematically meaningful applications of uniform boundedness were covered by uniform boundedness in that lowest type. In that way, a key contribution of the present paper is even the consideration of higher type uniform boundedness in the context of the monotone functional interpretation per se, but in particular that it uncovers probability theory as a mathematical area where these higher principles have mathematically useful consequences.

The following is now the classical metatheorem for probability contents, as established in \cite{NeriPischke2025}, extended to include the above uniform boundedness principles. For that purpose, we denote the union of all instances of $\UBOmega{\rho}$ and $\UBF{\rho}$, for pure types $\rho$, by $\UBOmega{\omega}$ and $\UBF{\omega}$, respectively. These principles can now be treated by a rather simple extension of the usual approach to the uniform boundedness principles in the context of the monotone functional interpretation as given in \cite{Kohlenbach2006} (see also \cite{GuenzelKohlenbach2016}). In particular, this extension (for the case $\rho=1$, with the higher cases following a similar pattern) is due to U. Kohlenbach, who kindly allowed us to include it here. In every other regard, the proof is a straightforward amalgamation of the proof given in \cite{NeriPischke2025} together with the discussions in \cite{GuenzelKohlenbach2016, Kohlenbach2006}, so that we confine ourselves to the treatment of $\UBOmega{\omega}$ and $\UBF{\omega}$.

\begin{theorem}[extending \cite{Kohlenbach2006} and \cite{NeriPischke2025}] \label{thm:metaClassical}
Let $\tau$ be admissible \emph{(}see \cite{GuenzelKohlenbach2016} and also \cite{GerhardyKohlenbach2008,Kohlenbach2005}\emph{)} and let $\varphi_\exists(x,n)$ be an existential formula of $\mathcal{F}^\omega[\PP]$ (with only $x,n$ free) such that all internal quantifiers have admissible types. If
\[
\mathcal{F}^{\omega}[\PP] + \UBOmega{\omega} +\UBF{\omega}\vdash\forall x^\tau\exists n^0 \varphi_\exists(x,n),
\]
then one can extract a partial functional $\Phi:\mathcal{S}_{\widehat{\tau}}\rightharpoonup\mathbb{N}$ which is total and bar-recursively computable on $\mathcal{M}_{\widehat{\tau}}$ (in the sense of Spector \cite{Spector1962}) and such that
\[
\mathcal{S}^{\omega,\Omega, \Algebra} \models \forall {x^*}^{\widehat{\tau}} \forall x \lesssim_{\tau} x^* \exists n \leq_0 \Phi(x^*) \, \varphi_\exists(x,n)
\]
holds for $\mathcal{S}^{\omega, \Omega, \Algebra}$ defined via any non-empty set $\Omega$, algebra $S\subseteq 2^\Omega$, and probability content $\PP$ on $\Algebra$ (with suitable interpretations of the additional constants).\footnote{For example, the constant $\PP$ is interpreted in the models $\mathcal{M}^{\omega,\Omega, \Algebra}$ or $\mathcal{S}^{\omega,\Omega, \Algebra}$ via $[\PP]_{\mathcal{M}^{\omega,\Omega,\Algebra}}=[\PP]_{\mathcal{S}^{\omega,\Omega,\Algebra}}=\lambda A^S.(\PP(A))_\circ$ for the given probability content $\PP$, where $(\cdot)_\circ$ selects a canonical type $1$ representation of a given real as defined in \cite{Kohlenbach2005}. While this operation is naturally non-effective, it is well-behaved w.r.t.\ majorization which serves all intents and purposes. In this paper, we do not rely on these aspects any further and generally refer to \cite{NeriPischke2025} for further discussions on this in relation to probability theory, as well as the interpretations of the other constants.} Further:
\begin{enumerate}
\item If $\widehat{\tau}$ is of degree $1$, then $\Phi$ is a total computable functional. 
\item We may have tuples instead of single variables $x,n$.
\item The system used in the premise may be extended with any further set of axioms $\beth$ of type $\Delta$ \emph{(}see e.g.\ \cite{NeriPischke2025} for a definition in the context of $\mathcal{F}^\omega[\PP]$\emph{)}, in which case the conclusion is true whenever $\mathcal{S}^{\omega,\Omega, \Algebra}\models \beth$.
\item If the claim is proved without $\DC$, then $\tau$ may further be arbitrary and $\Phi$ will be a total functional on $\mathcal{S}_{\widehat{\tau}}$ that is primitive recursive in the sense of G\"odel \cite{Goedel1958} and Hilbert \cite{Hilbert1926}.
\item If the claim can be proved without using either $\DC$ or uniform boundedness, then plain majorization (see e.g.\ Chapter 6 in \cite{Kohlenbach2008}) can be used instead of strong majorization.
\end{enumerate}
\end{theorem}
\begin{proof}
We only discuss the treatment of the higher uniform boundedness principle $\UBOmega{\rho}$ for a pure type $\rho=0(\xi)$, the principle $\UBF{\rho}$ is treated analogously. As with $\UBOmega{0}$ and any other uniform boundedness principle for that matter (see e.g.\ \cite{GuenzelKohlenbach2016,Kohlenbach2006}), we actually treat an associated stronger (type $\Delta$) axiom $F^{\Omega,\rho}$ defined as
\begin{gather*}
\forall \Phi^{\rho\underline{\beta}^t(\alpha)(0)},y^{\alpha(0)}\exists X^{\alpha(0)(\xi)}\leq_{\alpha(0)(0)} \lambda i^0.y\, \exists\underline{Z}^{\underline{\beta}(0)(\xi)}\, \forall u^\xi,k^0,x^\alpha,\underline{z}^{\underline{\beta}}\\
\left(\Phi(k,X(u,k),\underline{Z}(u,k),u)\geq_0 \Phi(k,\min_\alpha(x,yk),\underline{z},u)\right).
\end{gather*}
where, given $\beta=(\beta_1,\dots,\beta_m)$, we write $\underline{\beta}^t:=(\beta_m)\dots(\beta_1)$ and $\underline{\beta}(0):=(\beta_1(0),\dots,\beta_m(0))$. To see that treating $F^{\Omega,\rho}$ suffices, we have the following claim:\medskip

\begin{claim}$\mathcal{F}^{\omega}[\PP] +F^{\Omega,\rho}\vdash \UBOmega{\rho}$.
\end{claim}
\begin{claimproof} 
For that, let $y$ be given and assume that $\forall k^0, x^\alpha,\underline{z}^{\underline{\beta}}\, \exists w^\rho\, \varphi_\exists(y,k,\min_\alpha(x,yk),\underline{z},w)$. By $\QFAC$, we get a functional $\Phi$ such that
\[
\forall k^0, x^\alpha,\underline{z}^{\underline{\beta}}\, \varphi_\exists(y,k,\min_\alpha(x,yk),\underline{z},\Phi(k,x,z)).
\]
Provably, without any assumptions, one has $\min_\alpha(\min_\alpha(x,yk),yk)=_\alpha\min_\alpha(x,yk)$. Thus, by the quantifier-free extensionality rule $\QFER$, we have
\[
\forall k^0, x^\alpha,\underline{z}^{\underline{\beta}}\, \varphi_\exists(y,k,\min_\alpha(x,yk),\underline{z},\Phi(k,\min_\alpha(x,yk),z)).
\]
Using $F^{\Omega,\rho}$, we get an $X$ and $\underline{Z}$ such that 
\[
\forall u^\xi,k^0,x^\alpha,\underline{z}^{\underline{\beta}}\left(\Phi(k,X(u,k),\underline{Z}(u,k),u)\geq_0 \Phi(k,\min_\alpha(x,yk),\underline{z},u)\right).
\]
In particular, for any $k$, $x$ and $\underline{z}$, we have
\[
\lambda u^\xi.\Phi(k,X(u,k),\underline{Z}(u,k),u) \geq_\rho \Phi(k,\min_\alpha(x,yk),\underline{z}),
\]
so that $\chi:=\lambda k^0,u^\xi.\Phi(k,X(u,k),\underline{Z}(u,k),u)$ satisfies the conclusion of $\UBOmega{\rho}$.
\end{claimproof}
\medskip

\noindent To now formally deal with $F^{\Omega,\rho}$ in the context of a metatheorem, we proceed like with any other type $\Delta$ axiom (see e.g.\ \cite{GuenzelKohlenbach2016,NeriPischke2025}) and consider the witnessed Skolemization $\widetilde{F}^{\Omega,\rho}$ defined as
\begin{gather*}
\forall \Phi^{\rho\underline{\beta}^t(\alpha)(0)},y^{\alpha(0)}, u^\xi,k^0,x^\alpha,\underline{z}^{\underline{\beta}}\\
\left(\mathcal{X}(\Phi,y,u)\leq_{\alpha(0)} y \land \Phi(k,\mathcal{X}(\Phi,y,u,k),\underline{\mathcal{Z}}(\Phi,y,u,k),u)\geq_0 \Phi(k,\min_\alpha(x,yk),\underline{z},u)\right).
\end{gather*}
where $\mathcal{X}$ and $\underline{\mathcal{Z}}$ are new constants of type $\alpha(0)(\xi)(\alpha(0))(\rho\underline{\beta}^t(\alpha)(0))$ and $\underline{\beta}(0)(\xi)(\alpha(0))(\rho\underline{\beta}^t(\alpha)(0))$, respectively. Note that $\mathcal{X}$ and $\underline{\mathcal{Z}}$ solve the functional interpretation of $F^{\Omega,\rho}$ and, as $(F^{\Omega,\rho})'$ is intuitionistically implied by $F^{\Omega,\rho}$, they can also be used to solve that of $(F^{\Omega,\rho})'$. It remains to give an interpretation of these new constants in the model based on $\mathcal{M}^{\omega,\Omega,\Algebra}$ such that $\widetilde{F}^{\Omega,\rho}$ is satisfied and that they are majorized by closed terms from $\mathcal{A}^\omega$. For that, let $\Phi,y$ from $\mathcal{M}^{\omega,\Omega,\Algebra}$ be given with majorants $\Phi^*,y^*$. Then, for any $x,\underline{z}$ and $k$ in $\mathcal{M}^{\omega,\Omega,\Algebra}$, $\min_\alpha(x,yk)$ is majorized by $y^*k$ and $\underline{z}$ by the tuple of constant-$1$ functionals $\underline{z}^*$, as all elements of $\underline{z}$ have value type $\Omega$. Hence, we have
\[
\forall x\in\mathcal{M}_{\alpha},\underline{z}\in\mathcal{M}_{\underline{\beta}}\left( \Phi^*(k,y^*k,\underline{z}^*)\gtrsim_\rho\Phi(k,\min_\alpha(x,yk),\underline{z})\right).
\]
For any $u\in\mathcal{M}_{\xi}$, we can define
\[
\mathrm{Max}_{\Phi,y,k}(u):=\max\{\Phi(k,\min_\alpha(x,yk),\underline{z},u)\mid x\in\mathcal{M}_{\alpha}\text{ and }\underline{z}\in\mathcal{M}_{\underline{\beta}}\}
\]
which is well-defined as $\Phi^*(k,y^*k,\underline{z}^*,u^*)\geq_0\Phi(k,\min_\alpha(x,yk),\underline{z},u)$ for any such $x$ and $z$, where $u^*\gtrsim_\xi u$ (which exists as $u\in\mathcal{M}_{\xi}$). In particular, as the maximum is attained, we have
\[
\forall \Phi,y,u,k\in\mathcal{M}^{\omega,\Omega,\Algebra}\ \exists x,\underline{z}\in\mathcal{M}^{\omega,\Omega,\Algebra}\left( x\leq_\alpha yk\land \Phi(k,x,\underline{z},u)=_0\mathrm{Max}_{\Phi,y,k}(u)\right). 
\]
Using the axiom of choice, we get functionals $\mathfrak{X}$ and $\underline{\mathfrak{Z}}$ such that $\mathfrak{X}(\Phi,y,u)\leq y$ and 
\[
\Phi(k,\mathfrak{X}(\Phi,y,u,k),\underline{\mathfrak{Z}}(\Phi,y,u,k),u)=_0\mathrm{Max}_{\Phi,y,k}(u)\geq \Phi(k,\min_\alpha(x,yk),\underline{z},u)
\]
for any $\Phi,y,k,x,\underline{z}$ and any $i$. We set $[\mathcal{X}]_{\mathcal{M}^{\omega,\Omega,\Algebra}}:=\mathfrak{X}$ and $[\mathcal{Z}]_{\mathcal{M}^{\omega,\Omega,\Algebra}}:=\mathfrak{Z}$, which are well-defined as the constant-$1$ functionals majorize $\mathfrak{Z}$ and since $\lambda \Phi^*,y^*,u.y^*$ majorizes $\mathfrak{X}$. It is clear that these functionals satisfy $\widetilde{F}^{\Omega,\rho}$.
\end{proof}

Note in particular that the extractable bounds, as guaranteed by the above metatheorem, subscribe to a high degree of uniformity in the sense that they will be independent of all parameters relating to the underlying probability content space (see also again the discussion in \cite{NeriPischke2025}).

We now present the semi-constructive variants of the above metatheorem. In fact, we will discuss two variants of a semi-constructive metatheorem here, one based on the Dialectica interpretation and the other based on modified realizability. While we focus on the Dialectica interpretation throughout, we still highlight the metatheorem based on the modified realizability interpretation, as if a proof does not rely on those aspects of the theory underpinned by Markov's principle, then it provides a just as viable tool, with the added bonus that one can allow an even larger class of non-computational principles in proofs without voiding the bound extraction theorems (see in particular also the discussions in \cite{Kohlenbach2008}).

Crucially, also here we can allow a uniform boundedness principle which will similarly include types higher than $0$. In contrast to the classical case, the main benefit of the semi-constructive case even allows for this principle to be unrestricted regarding the formula (see e.g.\ the discussion in \cite{Kohlenbach1998}). In that way, the principles considered in the present paper take the following form:

\begin{definition}[extending \cite{Kohlenbach1996b} and \cite{Kohlenbach2006}]\label{def:constrUB}
Let $\rho$ be a pure type. The principle $\FullUBOmega{\rho}$ is defined as
\[
\begin{cases}
\forall y^{\alpha(0)}(\forall k^0, x^\alpha,\underline{z}^{\underline{\beta}}\, \exists w^\rho\, \varphi(y,k,\min_\alpha(x,yk),\underline{z},w)\\
\qquad\to \exists \chi^{\rho(0)}\,\forall k^0, x^\alpha,\underline{z}^{\underline{\beta}}\, \exists w \le_\rho \chi k\,\varphi(y,k,\min_\alpha(x,yk),\underline{z},w))
\end{cases}
\]
where $\alpha$ is a type of the form $0(\sigma_k)\ldots(\sigma_1)$, each type in $\underline{\beta}$ is of the form $\Omega(\tau_m)\ldots(\tau_1)$ and $\varphi$ is any formula. The principle $\FullUBF{\rho}$ is defined analogously, using $\Algebra$ instead of $\Omega$ in the conditions on the types in the tuple $\underline{\beta}$.
\end{definition}

As commented on above, in the presence of extensionality for $\varphi$, the above principle in fact implies the stronger uniform boundedness principle
\[
\begin{cases}
\forall y^{\alpha(0)}(\forall k^0, x\leq_\alpha yk,\underline{z}^{\underline{\beta}}\, \exists w^\rho\, \varphi(y,k,x,\underline{z},w)\\
\qquad\to \exists \chi^{\rho(0)}\,\forall k^0, x\leq_\alpha yk,\underline{z}^{\underline{\beta}}\, \exists w \le_\rho \chi k\,\varphi(y,k,x,\underline{z},w)).
\end{cases}
\]

We now state the resulting metatheorem as derived using the monotone Dialectica interpretation. While this metatheorem is new to the literature, we also here do not expand on the proof as it arises via an easy combination of the considerations from \cite{NeriPischke2025} with \cite{GerhardyKohlenbach2006} together with those of \cite{Kohlenbach1996b} for the uniform boundedness principle, adapted to the abstract types in a similar way as in \cite{Kohlenbach2006,GuenzelKohlenbach2016} and to higher types as in the proof of the previous Theorem \ref{thm:metaClassical}. Similar to before, we write $\FullUBOmega{\omega}$ for the union of the principles $\FullUBOmega{\rho}$, as well as $\FullUBF{\omega}$ for the union of the principles $\FullUBF{\rho}$, both for pure types $\rho$.

\begin{theorem}\label{thm:metaConstructive}
Let $\tau$ be admissible and let $\varphi(x,n)$ be any formula of $\mathcal{F}_i^{\omega}[\PP]$ (with only $x,n$ free) such that all positively occurring universal quantifiers and all negatively occurring existential quantifiers have small types and all other types are admissible. If
\[
\mathcal{F}_i^{\omega}[\PP] + \FullUBOmega{\omega} + \FullUBF{\omega} \vdash \forall x^\tau\exists n^0 \varphi(x,n),
\]
then one can extract a functional $\Phi:\mathcal{S}_{\widehat{\tau}}\to\mathbb{N}$ which is primitive recursive (in the sense of G\"odel \cite{Goedel1958} and Hilbert \cite{Hilbert1926}) and such that 
\[
\mathcal{S}^{\omega, \Omega, \Algebra} \models \forall {x^*}^{\widehat{\tau}} \forall x \lesssim_{\tau} x^*\exists n \leq_0 \Phi(x^*) \, \varphi(x,n)
\]
holds for $\mathcal{S}^{\omega, \Omega, \Algebra}$ defined via any non-empty set $\Omega$, algebra $S\subseteq 2^\Omega$, and probability content $\PP$ on $\Algebra$ (and with suitable interpretations of the additional constants). Further:

\begin{enumerate}
\item As before, we may have tuples instead of single variables $x,n$ and the system used in the premise may be extended with any further set of axioms $\beth$ of type $\Delta$ \emph{(}see e.g.\ \cite{NeriPischke2025} for a definition in the context of $\mathcal{F}^\omega[\PP]$\emph{)}, in which case the conclusion is true whenever $\mathcal{S}^{\omega,\Omega, \Algebra}\models \beth$.
\item If the proof does not use Markov's principle $\MARK$, then (using the monotone modified realizability interpretation) the conclusion remains valid even in the context of the independence of premise principle $\IPEF$ for existential-free formulas \emph{(}see e.g.\ Chapter 5 in \cite{Kohlenbach2008}\emph{)} and any further set of axioms $\beth$ of type $\Theta$ \emph{(}see e.g.\ Chapter 7 in \cite{Kohlenbach2008}\emph{)}, in which case the conclusion is true whenever $\mathcal{S}^{\omega,\Omega, \Algebra}\models \beth$.
\item If the claim is proved without uniform boundedness, then plain majorization (see e.g.\ Chapter 6 in \cite{Kohlenbach2008}) can be used instead of strong majorization.
\end{enumerate}
\end{theorem}

As before, also the bounds guaranteed by this semi-constructive metatheorem are highly uniform, being again independent of all parameters relating to the underlying probability content space.

We now provide further extensions in this semi-constructive context to also accommodate the so-called principle of \emph{contra-collection}, which has similar mathematically useful consequences than uniform boundedness later on. This principle is motivated from its use in the context of the bounded functional interpretation \cite{FerreiraOliva2005}, but introduced here in a style that is suitably adapted to the monotone functional interpretation for the first time.  

\begin{definition}
The principle $\CCOmega{0}$ is defined as
\[
\forall k^0 \exists \underline{z}^{\underline{\beta}} \forall n \le_0  k\,\varphi_0(\underline{z},n)
\to\exists \underline{z}^{\underline{\beta}} \forall n^0\, \varphi_0(\underline{z},n),
\]
where each type in $\underline{\beta}$ is of the form $\Omega(\tau_m)\ldots(\tau_1)$ and $\varphi_0$ is a quantifier-free formula. The principle $\CCF{0}$ is defined analogously, using $\Algebra$ instead of $\Omega$ in the conditions on the types in the tuple $\underline{\beta}$. 
\end{definition}

We shall show that under suitable restrictions on the formula, a version of Theorem \ref{thm:metaConstructive} holds in the presence of $\CCF{0}$ and $\CCOmega{0}$. However, to provide a treatment, we require a certain ``limited'' use of these principles, which nevertheless suffices for all applications.

Concretely, in what follows, we write $\oplus\, \CCOmega{0}$ (viz.\ $\oplus\, \CCF{0}$) in the specification of a system to signify that, in the context of a derivation, the principle $\CCOmega{0}$ (viz.\ $\CCF{0}$) must not be used in the proof of the premise of an application of the quantifier-free rule of extensionality $\QFER$. This in particular has the consequence that while systems with the quantifier-free extensionality rule, which do not satisfy the deduction theorem as a whole, nevertheless then satisfy the deduction theorem relative to the premise $\CCOmega{0}$ (viz.\ $\CCF{0}$), as can be immediately verified by an induction on the length of the proof. 

This availability of a fragment of the deduction theorem will be crucial for us in the following result to discharge these premises in the context of a bound extraction theorem. Indeed, we will be able to deal with this additional premise of contra-collection as follows: Using the deduction theorem, we move each such application to the premise of an implication and strengthen this principle to a formula $\beth$ of type $\Delta$, motived along the treatment of weak K\H{o}nig's lemma given by Avigad and Feferman \cite{AvigadFeferman1998} (which is slightly different compared to the first treatment of weak K\H{o}nig's lemma in the context of the (monotone) functional interpretation given by Kohlenbach \cite{Kohlenbach1992}). To this compound statement, we then apply the previous semi-constructive metatheorem, which extracts quantitative information on the conclusion that is verified in the context of a so-called $\varepsilon$-weakening $\beth_\varepsilon$ (see e.g.\ Chapter 10 in \cite{Kohlenbach2008}), which can then be discharged.

\begin{theorem}\label{thm:metaConstructiveCC}
Let $\tau$ be admissible and let $\varphi_\exists(x^\tau,n^0)$ be an existential formula of $\mathcal{F}^\omega_i[\PP]$ (with only $x,n$ free) such that all internal quantifiers have admissible types. If
\[
\mathcal{F}_i^{\omega}[\PP] + \FullUBOmega{\omega} + \FullUBF{\omega} \oplus\CCOmega{0} \oplus\CCF{0}\vdash \forall x^\tau\exists n^0 \varphi_\exists(x,n),
\]
then one can extract a functional $\Phi:\mathcal{S}_{\widehat{\tau}}\to\mathbb{N}$ which is primitive recursive  (in the sense of G\"odel \cite{Goedel1958} and Hilbert \cite{Hilbert1926}) and such that
\[
\mathcal{S}^{\omega, \Omega, \Algebra} \models \forall {x^*}^{\widehat{\tau}} \forall x \lesssim_{\tau} x^*\exists n \leq_0 \Phi(x^*) \, \varphi_\exists(x,n)
\]
holds for $\mathcal{S}^{\omega, \Omega, \Algebra}$ defined via any non-empty set $\Omega$, algebra $S\subseteq 2^\Omega$, and probability content $\PP$ on $\Algebra$ (and with suitable interpretations of the additional constants).

We have the same remark (1) as in Theorem \ref{thm:metaConstructive}.
\end{theorem}
\begin{proof}
We only discuss the treatment of $\CCOmega{0}$, the principle $\CCF{0}$ can be treated analogously and simultaneously. So suppose we have 
\[
\mathcal{F}_i^{\omega}[\PP] + \FullUBOmega{\omega} + \FullUBF{\omega} \oplus\CCOmega{0} \vdash \forall x^\tau\exists n^0 \varphi_\exists(x,n).
\]
As a proof is finite, only finitely many instances of the principle $\CCOmega{0}$ are used. For simplicity, we assume it is only one such principle, say for the formula $\psi_0$, to which we still refer by $\CCOmega{0}$. As $\oplus$ allows for the deduction theorem, we have 
\[
\mathcal{F}_i^{\omega}[\PP] + \FullUBOmega{\omega} + \FullUBF{\omega}  \vdash \CCOmega{0} \to \forall x^\tau\exists n^0 \varphi_\exists(x,n).
\]
Consider the formula
\[
\beth:= \exists \underline{z}^{\underline{\beta}}\forall k^0 (\exists \underline{z^*}^{\underline{\beta}}\, \forall m \le_0  k\, \psi_0(\underline{z^*},m)
\to \psi_0(\underline{z},k)).
\]
It is clear that $\beth \to \CCOmega{0}$ is provable in $\mathcal{F}_i^{\omega}[\PP]$, and so we have,
\[
\mathcal{F}_i^{\omega}[\PP] + \FullUBOmega{\omega} + \FullUBF{\omega}  \vdash \beth \to \forall x^\tau\exists n^0 \varphi_\exists(x,n).
\]
Thus, we have
\[
\mathcal{F}_i^{\omega}[\PP] + \FullUBOmega{\omega} + \FullUBF{\omega}  \vdash  \exists \underline{z}^{\underline{\beta}}\forall k^0\,\psi(\underline{z},k) \to \forall x^\tau\exists n^0 \varphi_0(x,n).
\]
where 
\[
\psi(\underline{z},k):= \exists \underline{z^*}^{\underline{\beta}}\forall n \le_0  k\,\varphi_0(\underline{z^*},n)
\to \varphi_0(\underline{z},k).
\]
Using additional epsilon-terms (see e.g.\ \cite{GuenzelKohlenbach2016}), which we suppress, we can regard $\psi$ as quantifier-free. By $\MARK$ and $\IPU$, we have that $\mathcal{F}_i^{\omega}[\PP] + \FullUBOmega{\omega} + \FullUBF{\omega}$ proves
  \[
\mathcal{F}_i^{\omega}[\PP] + \FullUBOmega{\omega} + \FullUBF{\omega}  \vdash \forall x^\tau\forall \underline{z}^{\underline{\beta}} \exists n^0 \exists k^0( \psi(\underline{z},k)\to \varphi_\exists(x,n)).
\]
Using Theorem \ref{thm:metaConstructive}, we get that there exists a functional $\Phi$ such that for all $x\in \mathcal{M}_\tau$, $x^*\in \mathcal{M}_{\widehat{\tau}}$, if $x^*\gtrsim x$, then
\[
\mathcal{M}^{\omega,\Omega, \Algebra} \models  \forall \underline{z}^{\underline{\beta}}\ \exists n\le_0 \Phi(x^*)\exists k\le_0 \Phi(x^*) (\psi(\underline{z},k) 
 \to \varphi_\exists(x,n)).
\]
This implies
\[
\mathcal{M}^{\omega,\Omega, \Algebra} \models   \beth_\varepsilon 
 \to \exists n\le_0 \Phi(x^*)\,\varphi_\exists(x,n),
\]
where we write
\[
\beth_\varepsilon:= \forall i^0 \exists \underline{z}^{\underline{\beta}} \forall k \le_0 i(\exists \underline{z^*}^{\underline{\beta}} \forall m \le_0  k\,\psi_0(\underline{z^*},m)
\to \psi_0(\underline{z},k)).
\]
Now, we have that $\mathcal{M}^{\omega,\Omega, \Algebra} \models \beth_\varepsilon$. To see this, let $i$ be given and let us take $k^* \le i$ maximal such that $\exists \underline{z}\forall m \le k^*\, \psi_0(\underline{z},m)$. We may suppose that such a $k^*$ exists, as if not, then we would in particular have $\forall\underline{z}\neg\psi_0(\underline{z},0)$, which yields that the premise of $\beth_\varepsilon$ is always false, so that $\beth_\varepsilon$ itself is satisfied and the claim already holds true. Given such a maximal $k^*$, take a corresponding $\underline{z}$ with $\forall m \le k^*\, \psi_0(\underline{z},m)$. Let $k \le i$ be arbitrary. If $k \le k^*$, then $\psi_0(\underline{z}, k)$ and so the conclusion of $\beth_\varepsilon$ is satisfied, so that $\beth_\varepsilon$ itself is satisfied. If $k > k^*$, then by maximality of $k^*$, we have $\forall \underline{z} \exists m \le k\, \neg\psi_0(\underline{z}, m)$, which yields that the premise of $\beth_\varepsilon$ is false, so that again $\beth_\varepsilon$ itself is satisfied. Thus, we have $\mathcal{M}^{\omega,\Omega, \Algebra} \models \exists n \le_0 \Phi(x^*)\varphi_\exists(x,n)$ and the result follows as the types are low.
\end{proof}

While the above can potentially be extended to principles $\CCOmega{1}$ and $\CCF{1}$ of contra-collection at type $1$, we do not pursue this here further and remain with the above result.

\section{A systematic approach to representing probabilistic statements}\label{sec:trans}

In this section, we present a general approach for formally representing almost sure statements in the language of the formal systems presented in Section \ref{sec:prelim}. To that end, we will rely on the notions of outer (and inner) content (or measure) of a formula in this language, inspired by the corresponding notion(s) from ordinary probability theory.

As we will illustrate in the following Section \ref{sec:meta}, this representation can then be used, in combination with a monotone variant of G\"odel's functional interpretation due to Kohlenbach, to systematically assign (and hence extract) the computational content of measure-theoretic statements from semi-constructive or even classical proofs (if chained with a negative translation), resulting ultimately in dedicated logical metatheorems for such statements in the style of proof mining.

To motivate our formal representation of probabilistic statements, consider a content space $(\Omega, \Algebra, \PP)$ and suppose that we are given a property $\varphi(\varpi)$ with a variable $\varpi$ over $\Omega$ which induces a measurable property
\[
A_\varphi:=\{\varpi\in \Omega\mid \varphi(\varpi)\}\in \Algebra.
\]
In that case, we can formally assign a probability to the statement $\varphi$ by setting
\[
\PP(\varphi):=\PP(A_\varphi)
\]
so that $\varphi$ being true almost surely is precisely captured by the statement that $\PP(\varphi)=1$. If we are interested in capturing this approach formally, we encounter an immediate logical problem.
The approach requires both a formal notion of a “measurable formula” and a comprehension principle to construct the corresponding measurable sets. However, this comprehension principle is logically highly complicated due to hidden quantifiers, and their inclusion would “taint” the translation with auxiliary complexity that stems only from this arbitrary representation.

Instead, we opt for a representation that does not attempt to eliminate the quantifiers via such comprehension, but rather leaves the formula (and, as we will see later, also its computational content) intact in an appropriate way. The device that facilitates this lift is essentially the outer measure (see e.g.\ \cite{Klenke2008}), or more appropriately the outer content since we are working over content spaces (see e.g.\ \cite{RaoRao1983}), a measure-theoretic device which assigns a probability to a generic set, measurable or not, only by putting it into relation with other measurable sets. Concretely, associated to the probability content $\PP$ is the outer content $\PP^*$ of $\PP$ that assigns to each set $A\subseteq \Omega$, measurable or not, the value
\[
\PP^*(A):=\inf\{\PP(B)\mid A\subseteq B\text{ and }B\in\Algebra\},
\]
which coincides with the content $\PP(A)$ when $A$ is measurable. Instead of considering $\PP(A_\varphi)=1$ as a representation of $\varphi$ being true almost surely, we may thus consider $\PP^*(A_\varphi)=1$. While mathematically the same, this notion $\PP^*(A_\varphi)=1$, based on the outer content, can, due to its simpler definition, be formally captured rather immediately by the statement
\[
\forall A^\Algebra(\forall \varpi^\Omega(\varphi(\varpi) \to \varpi \in A) \to \PP(A)=_\RR 1).
\]
Now, as with the outer content in general, it is actually irrelevant for this formulation whether $A_\varphi$ is measurable or not. In the same way, as also immediately apparent from its phrasing, the above property is suitable for arbitrary statements $\varphi(\varpi^\Omega)$. It is precisely this formulation that underlies our representation of probabilistic statements.

However, to arrive at a general working notion, we will first require some technical considerations. First, to be able to quantify over the degree of probability, we do not want to restrict to a representation of almost sure notions only. Therefore, we will consider the more general abbreviation
\[
\PP[\varphi] \ge \lambda \; :\equiv \; \forall A^\Algebra (\forall \varpi^\Omega(\varphi(\varpi) \to \varpi \in A) \to \PP(A) \ge_\RR \lambda)
\]
where $\varphi(\varpi^\Omega)$ is a given, arbitrary, formula as before and $\lambda^1$ is a free variable of type $1$. This notion, of which ``$\varphi$ being true almost surely'' is a special case, by setting $\lambda=1$, comes naturally equipped with a dual notion of low probability,
\[
\PP[\varphi] \leq \lambda\; :\equiv \;\PP[\neg\varphi]\geq 1-\lambda.
\]
Further, as we will later see, this dual notion can be equivalently captured via
\[
\PP[\varphi] \le \lambda \; \leftrightarrow \; \forall A^\Algebra (\forall \varpi^\Omega(\varpi \in A \to \varphi(\varpi) ) \to \PP(A) \le_\RR \lambda)
\]
and, similarly to the abbreviation $\PP[\varphi] \ge \lambda $, which is motivated using the outer content, this dual notion can be understood as being motivated by the inner content $\PP_*$ associated with $\PP$ (see e.g.\ \cite{RaoRao1983}), defined for generic sets $A\subseteq\Omega$ via
\[
\PP_*(A):=\sup\{\PP(B)\mid A\supseteq B\text{ and }B\in\Algebra\}.
\]
As these two abbreviations are dually linked via the negation of the formula in that way, it is in particular justified to drop the super- and subscripts used in the above notation of the outer and inner content as they yield a single underlying concept of probability, high or low, for a formula.

Now, for the purpose of dealing with classical reasoning, these formal representations of probabilistic formulas provide the precise internal notions needed to represent the associated probability-theoretic statements in a suitable way, which allows for applications in proof mining. However, in this paper, we will also be concerned with semi-constructive reasoning, in which the above property proves to be wholly unsuitable for this purpose. Indeed, the above formulation of the outer/inner content of a formula is inherently classical as $\varphi$, appearing in a double negative position, is treated by essentially all proof interpretations as being inside a double negation. This has the effect that, for example, the modified realizability interpretation of $\PP[\varphi]\geq 1$ (i.e.\ our representation of $\varphi$ being true almost surely) just extracts the empty realizer, while the functional interpretation extracts a realizer similar to that for the double negation of $\varphi$.

Thus, to obtain a notion that uniformly represents such probabilistic properties, and moreover preserves the computational meaning of $\varphi$ while suitably lifting it to the probabilistic level in both classical and constructive settings, we consider the following abbreviation as our formal representation:

\begin{definition} \label{def-abbreviations}
Let $\varphi(\varpi^\Omega)$ be a formula and $\lambda^1$ be a free variable. We define the abbreviation
\begin{equation} \label{def-P-gt}
    \PP[\varphi] \ge \lambda \; :\equiv \; \forall A^\Algebra (\PP(A) <_\RR \lambda \to \exists \varpi^\Omega\in A^c\,\varphi(\varpi)).\tag{P}
\end{equation}
Dually, we define
\begin{equation} \label{def-P-lt}
    \PP[\varphi] \leq \lambda \; :\equiv \; \PP[\overline{\varphi}]\geq 1-\lambda\tag{D}
\end{equation}
where $\overline{\varphi}$ is the De Morgan dual of $\varphi$.
\end{definition}

Classically, the above notion is equivalent to our previous outer/inner content, so that in that context, no expressivity is lost. Constructively, however, our ``official'' definition is much more productive as $\varphi$, and by duality also $\overline{\varphi}$ in the low probability variant, appear positively. It should be noted that the motivation for utilizing the De Morgan dual $\overline{\varphi}$ instead of $\neg\varphi$ is again due to constructive considerations, as $\overline{\varphi}$ provides a constructively much more productive notion than the object $\neg\varphi$, for which, for example, the modified realizability interpretation again just extracts the empty realizer.

Following standard probability theory notation, we generally omit parameters of type $\Omega$ in expressions of the form $\PP[\varphi]\geq\lambda$. E.g., we write $\PP[Qx^\tau\varphi(x)]\geq\lambda$ for $\PP[Qx^\tau\varphi(x,\varpi)]\geq\lambda$, where $Q\in\{\forall,\exists\}$ is some quantifier, and we write $\PP[A]\geq\lambda$ for $\PP[\varpi\in A]\geq\lambda$, where $A$ is of type $S$.

\section{The algebra of probabilistic statements}\label{sec:algebra}

In this section, we study the basic algebra of abbreviations (\ref{def-P-gt}) and (\ref{def-P-lt}). To that end, we begin by noting some trivial monotonicity properties and fixing some further notation in that context: Clearly, if $\lambda\geq_\mathbb{R}\mu$ and $\PP[\varphi]\geq\lambda$, then also $\PP[\varphi]\geq\mu$. Likewise, we have that $\lambda\leq_\mathbb{R}\mu$ and $\PP[\varphi]\leq\lambda$ imply $\PP[\varphi]\leq\mu$. Further, we write $\PP[\varphi]>\lambda$ or $\PP[\varphi]<\lambda$ for the respective non-strict variants resulting from replacing $\geq_{\mathbb{R}}$ or $\leq_{\mathbb{R}}$ in the above by $>_{\mathbb{R}}$ or $<_{\mathbb{R}}$, respectively. It should be noted, however, that these only serve a theoretical purpose later on as $>_{\mathbb{R}}$ and $<_{\mathbb{R}}$ are existential, which will be convenient in some places. The mathematically more appropriate phrasing of these strict relations of the outer and inner content would rather be given by
\[
\exists\mu^\mathbb{Q}>0\left( \PP[\varphi] \geq \lambda+\mu\right) \text{ and }\exists\mu^\mathbb{Q}>0\left( \PP[\varphi] \leq \lambda-\mu\right),
\]
but if such a relation is required in our context, we will always use the above formulas explicitly.

Now as mentioned briefly before, the definition \eqref{def-P-lt} of low probability, which is simply connected to the definition \eqref{def-P-gt} of high probability by duality, can be given a formal definition inspired by inner contents analogous to the official definition of \eqref{def-P-gt} inspired by outer contents:

\begin{proposition} \label{pro:duality} For any formula $\varphi(\varpi^\Omega)$:
\[
\mathcal{F}^\omega_i[\PP] \vdash \PP[\varphi] \le \lambda \leftrightarrow \forall A^\Algebra (\PP(A) >_\RR \lambda \to \exists \varpi^\Omega \in A\,\overline{\varphi}(\varpi)).
\]
\end{proposition}
\begin{proof} Working over $\mathcal{F}^\omega_i[\PP]$ and fixing $\lambda^1$, suppose that $\PP[\varphi] \le \lambda$. We have by definition that
\[
\forall A^\Algebra (\PP(A) <_\RR 1-\lambda \to \exists \varpi^\Omega\in A^c\,\overline{\varphi}(\varpi)).
\]
Instantiating $A$ with $A^c$, this implies 
\[
\forall A^\Algebra (\PP(A^c) <_\RR 1-\lambda \to \exists \varpi^\Omega\in (A^c)^c\,\overline{\varphi}(\varpi)).
\]
Since $\mathcal{F}^\omega_i[\PP]$ proves that $(A^c)^c =_\Algebra A$ and $1=_\mathbb{R}\PP(A^c)+\PP(A)$ we obtain
\[
\forall A^\Algebra (\PP(A) >_\RR \lambda \to \exists \varpi^\Omega\in A \; \overline{\varphi}(\varpi)).
\]
Similarly, we can show that $\forall A^\Algebra (\PP(A) >_\RR \lambda \to \exists \varpi^\Omega\in A \; \overline{\varphi}(\varpi))$ implies $\PP[\varphi] \le \lambda$.
\end{proof}

This, in particular, also justifies dropping the sub- and superscripts as discussed in the previous section.

Another key observation is that measurable properties cause the outer and inner interpretations to collapse to the underlying content. For that, we introduce the abbreviation 
\[
   \PP[\varphi]=\lambda \; :\equiv \; \PP[\varphi]\leq\lambda \; \land \; \PP[\varphi]\geq\lambda.
\]

\begin{proposition}\label{pro:measurable} Over $\mathcal{F}^\omega_i[\PP]$, suppose that $\varphi$ is measurable in the sense that for some $A_0^\Algebra$:
\[
\forall \varpi^\Omega\left(\varphi(\varpi)\leftrightarrow \varpi\in A_0\right).
\]
Then for any free variable $\lambda^1$:
\[
\PP[\varphi]\geq \lambda\leftrightarrow \PP(A_0)\geq_\mathbb{R}\lambda 
\quad \text{and} \quad 
\PP[\varphi]\leq \lambda\leftrightarrow \PP(A_0)\leq_\mathbb{R}\lambda.
\]
In particular $\PP[\varphi]=\PP(A_0)$.
\end{proposition}
\begin{proof} First, note that we have
\begin{align*}
\PP[\varphi]\geq\lambda & \equiv \forall A^\Algebra\left(\PP(A)<_\mathbb{R}\lambda\to \exists \varpi\in A^c\,\varphi(\varpi)\right)\\
&\rightarrow \PP(A_0)<_\mathbb{R}\lambda\to \exists \varpi\in A^c_0\,\varphi(\varpi)\\
&\leftrightarrow \PP(A_0)<_\mathbb{R}\lambda\to \exists \varpi\in A^c_0\,(\varpi\in A_0)\\
&\leftrightarrow \neg(\PP(A_0)<_\mathbb{R}\lambda)
\end{align*}
over $\mathcal{F}^\omega_i[\PP]$, as $A^c\cap A=_\Algebra\emptyset$. Conversely, assume $\neg(\PP(A_0)<_\mathbb{R}\lambda)$ and let $\PP(A)<_\mathbb{R}\lambda$. Now, we have that $\neg(\PP(A_0)<_\mathbb{R}\lambda)\leftrightarrow \PP(A_0)\geq_\mathbb{R}\lambda$. Then $A_0\not\subseteq A$, by monotonicity of $\PP$, and so there is an $\varpi\in A^c$ with $\varpi\in A_0$, that is $\varphi(\varpi)$. As $A$ was arbitrary, we have shown
\[
\neg(\PP(A_0)<_\mathbb{R}\lambda)\rightarrow \forall A^\Algebra\left(\PP(A)<_\mathbb{R}\lambda\to \exists \varpi\in A^c\,\varphi(\varpi)\right)\leftrightarrow \PP[\varphi]\geq\lambda
\]
over $\mathcal{F}^\omega_i[\PP]$. So, we have shown $\PP[\varphi]\geq\lambda\leftrightarrow \PP(A_0)\geq_\mathbb{R}\lambda$ over $\mathcal{F}^\omega_i[\PP]$. For the dual claim, simply note that $\forall\varpi^\Omega\left(\overline{\varphi}(\varpi)\leftrightarrow \varpi\not\in A_0\right)$ so that using \eqref{def-P-lt}, we get
\[
\PP[\varphi]\leq\lambda
\; \leftrightarrow \; 
\PP[\overline{\varphi}]\geq 1-\lambda
\; \leftrightarrow \;
\PP(A_0^c)\geq_{\mathbb{R}}1-\lambda
\; \leftrightarrow \;
\PP(A_0)\leq_\mathbb{R}\lambda,
\]
using again that $1=_\mathbb{R}\PP(A^c)+\PP(A)$, for any $A^\Algebra$.
\end{proof}

\begin{remark}
In Proposition \ref{pro:measurable}, note that already $\forall \varpi^\Omega\left(\varphi(\varpi)\rightarrow \varpi\in A_0\right)$ implies
\[
\PP[\varphi]\geq \lambda\rightarrow \PP(A_0)\geq_\mathbb{R}\lambda\text{ and }\PP(A_0)\leq_\mathbb{R}\lambda\rightarrow \PP[\varphi]\leq \lambda
\]
and $\forall \varpi^\Omega\left(\varpi\in A_0\rightarrow \varphi(\varpi)\right)$ implies
\[
\PP(A_0)\geq_\mathbb{R}\lambda\rightarrow \PP[\varphi]\geq \lambda\text{ and }\PP[\varphi]\leq \lambda\rightarrow \PP(A_0)\leq_\mathbb{R}\lambda.
\]
\end{remark}

The outer content considered in this paper is the outer content $\PP^*$ relative to the content $\PP$, and in that way a special instance of the general notion of an outer content (or measure) studied in measure theory. In the probabilistic case, this class of general outer contents on $\Omega$ comprises all set functions $\nu:2^\Omega\to [0,1]$ that satisfy the following three properties:
\begin{enumerate}
\item $\nu(\emptyset)=0$,
\item for all $A,B\in 2^\Omega$: if $A\subseteq B$, then $\nu(A)\leq\nu(B)$,
\item for all $A_0,\dots,A_m\in 2^\Omega$: $\nu(\bigcup_{n=0}^m A_n)\leq\sum_{n=0}^m \nu(A_n)$.
\end{enumerate}

It should be noted that outer measures, compared to outer contents, are generally required to be sub-$\sigma$-additive (see e.g.\ \cite{Halmos1950}) instead of just subadditive, but this property does not hold true already for the outer content $\PP^*$ relative to a content $\PP$ if the latter is only formulated over a field (see e.g.\ Remark 4.1.5 in \cite{RaoRao1983}).

For our particular outer content $\PP^*$ formalized using $\PP[\cdot]$, we can replicate these key properties formally in the system $\mathcal{F}^\omega_i[\PP]$, identifying general sets $A \in 2^\Omega$ with arbitrary formulas $\varphi(\varpi)$ in the language of $\mathcal{F}^\omega_i[\PP]$.

In particular, a formal variant of the first property (1) already follows from the preceding Lemma \ref{pro:measurable}, by which any set specified by $\varphi(\varpi)$ that equals $\emptyset$, that is $\forall \varpi^\Omega\neg\varphi(\varpi)$, already satisfies $\PP[\varphi]=0$.

The next lemma now gives a formal account of the second property, formalizing inclusions $A\subseteq B$ among general sets $A,B\in 2^\Omega$ via implications $\forall \varpi^\Omega (\varphi(\varpi)\to\psi(\varpi))$ for general formulas $\varphi(\varpi)$, $\psi(\varpi)$ in the language of $\mathcal{F}^\omega_i[\PP]$. To model a comparison between outer contents, which in our system are not numbers but actually formal expressions involving comparisons, we introduce the notation 
\[
\PP[\varphi] \leq \PP[\psi] \; :\equiv \; \forall \lambda^1 (\PP[\varphi] \geq \lambda \to \PP[\psi] \geq \lambda).
\]

\begin{proposition}\label{pro:implicationOuter}
Over $\mathcal{F}^\omega_i[\PP]$, assume $\varphi(\varpi^\Omega)$ and $\psi(\varpi^\Omega)$ are formulas such that $\varphi(\varpi) \to \psi(\varpi)$ outside of a null set $A_0$, i.e.
\[ \PP(A_0)=_\mathbb{R} 0 \wedge \forall \varpi^\Omega\in A_0^c (\varphi(\varpi)\to\psi(\varpi)). \]
Then $\PP[\varphi]\leq\PP[\psi]$.
\end{proposition}
\begin{proof} Let $\lambda^1$ be arbitrary and suppose $\PP[\varphi] \geq \lambda$, i.e.
\[
\forall A^\Algebra (\PP(A) <_\RR \lambda \to \exists \varpi^\Omega\in A^c\,\varphi(\varpi)).
\]
Let $A^\Algebra$ be arbitrary and suppose $\PP(A)<_\mathbb{R}\lambda$. Then we have
\[
\PP(A\cup A_0)\leq_\RR \PP(A)+\PP(A_0)<_\RR\lambda
\]
so that the above implies $\exists \varpi^\Omega\in (A\cup A_0)^c\,\varphi(\varpi)$. Provably in $\mathcal{F}^\omega_i[\PP]$, we have $(A\cup A_0)^c=_\Algebra A^c\cap A_0^c$, and hence, we get that $\varphi(\varpi)$ implies $\psi(\varpi)$ as $\varpi\in A_0^c$. Since $\varpi \in A^c$ and $A$ was arbitrary, we have shown that $\PP[\psi] \geq \lambda$, i.e.
\[
\forall A^\Algebra (\PP(A) <_\RR \lambda \to \exists \varpi^\Omega\in A^c\,\psi(\varpi)).
\]
As $\lambda$ was arbitrary, it follows that $\PP[\varphi]\leq\PP[\psi]$.
\end{proof}

In particular, the above Proposition \ref{pro:implicationOuter} establishes a type of \emph{almost-sure} extensionality of the expression $\PP[\varphi]\geq\lambda$ in $\varphi$.

Let us now consider the last property, $\nu(\bigcup_{n=0}^m A_n)\leq\sum_{n=0}^m \nu(A_n)$, for arbitrary $A_0,\dots,A_m\in 2^\Omega$. In the following, we will show how we can formally recognize this property, phrased using $\PP[\cdot]$, in our system $\mathcal{F}^\omega_i[\PP]$, using a formula $\varphi(n,\varpi)$ to represent the sets $A_n$ and $\exists n\leq_0 m\, \varphi(n,\varpi)$ to represent their finite union $\bigcup_{n=0}^m A_n$.

To formally encompass this property, we require a representation of a finite sum of outer contents, which is again not completely immediate as outer contents are not numbers but formal comparisons. Given a formula $\varphi(n,\varpi)$, we in that vein write
\[
\sum_{n=0}^m\PP[\varphi(n)]\geq\lambda:\equiv \forall\tilde{\lambda}^{1(0)}_{(\cdot)}\left(\lambda\geq_\mathbb{R}\sum_{n=0}^m\tilde{\lambda}_n\to \exists n\leq_0 m(\PP[\varphi(n)]\geq \tilde{\lambda}_n)\right).
\]
We now obtain the subadditivity of the outer content:

\begin{lemma}\label{lem:subAddOuter}
Over $\mathcal{F}^\omega[\PP]$, let $\varphi(n^0,\varpi^\Omega)$ be a formula. Then 
\[ \PP[\exists n\leq_0 m\,\varphi(n)]\leq\sum_{n=0}^m\PP[\varphi(n)], \]
that is
\[
\forall \lambda^1 \left(\PP[\exists n\leq_0 m\,\varphi(n)] \geq \lambda \to \sum_{n=0}^m\PP[\varphi(n)] \geq \lambda\right).
\]
\end{lemma}
\begin{proof}
Suppose for a contradiction that $\PP[\exists n\leq_0 m\,\varphi(n)]\geq\lambda$ but $\sum_{n=0}^m\PP[\varphi(n)]\geq\lambda$ fails. Then there is a $\tilde{\lambda}^{1(0)}_{(\cdot)}$ such that $\lambda\geq_\RR\sum_{n=0}^m\tilde{\lambda}_n$ but $\neg\PP[\varphi(n)]\geq \tilde{\lambda}_n$, for all $n\leq_0 m$. In particular, for any $n\leq_0 m$ there exists an $A_n$ such that $\PP(A_n)<_\RR\tilde{\lambda}_n$ but $\neg\varphi(n,\varpi)$ for all $\varpi\in A_n^c$. However we then have
\[
\PP\left(\bigcup_{n=0}^m A_n\right)\leq_\RR \sum_{n=0}^m\PP(A_n)<_\RR \sum_{n=0}^m\tilde{\lambda}_n\leq_\RR \lambda.
\]
As $\PP[\exists n\leq_0 m\,\varphi(n)]\geq\lambda$ holds, we get
\[
\exists\varpi\in \left(\bigcup_{n=0}^m A_n\right)^c\exists n\leq_0 m\,\varphi(n,\varpi).
\]
In particular, we have $\varpi\in A_n^c$ and $\varphi(n,\varpi)$. This contradicts that $\neg\varphi(n,\varpi)$ for all $\varpi\in A_n^c$ as established previously.
\end{proof}

We now move to the quantifier structure of a formula. Here, we begin with the following partial prenexation results:

\begin{proposition}\label{pro:initialPrenex} For any type $\tau$, any formula $\varphi(x^\tau,\varpi^\Omega)$, and any free variable $\lambda^1$, $\mathcal{F}^\omega_i[\PP]$ proves:
\begin{enumerate}
    \item $\PP[\forall x^\tau\varphi(x)]\ge \lambda \to \forall x^\tau(\PP[\varphi(x)]\ge \lambda)$,
    \item $\forall \mu^\mathbb{Q}> 0\, \exists x^\tau(\PP[\varphi(x)]\leq \lambda + \mu) \to \PP[\forall x^\tau\varphi(x)]\leq \lambda$.
\end{enumerate}
Further, $\mathcal{F}^\omega_i[\PP]$ also proves:
\begin{enumerate}
    \setcounter{enumi}{2}
    \item $\PP[\exists x^\tau\varphi(x)]\leq \lambda\to \forall x^\tau(\PP[\varphi(x)]\leq \lambda)$,
    \item $\forall \mu^\mathbb{Q}> 0\, \exists x^\tau(\PP[\varphi(x)]\ge \lambda - \mu) \to \PP[\exists x^\tau\varphi(x)]\ge \lambda$.
\end{enumerate}
\end{proposition}
\begin{proof} We are working over $\mathcal{F}^\omega_i[\PP]$ and we only show (1) and (2). Items (3) and (4) follow from these by \eqref{def-P-lt}.
\begin{enumerate}
    \item We have
    \begin{align*}
    \PP[\forall x^\tau\varphi (x)]\geq\lambda 
    & \; \equiv \; \forall A^\Algebra (\PP(A) <_\RR \lambda \to \exists \varpi^\Omega\in A^c\,\forall x^\tau\varphi (x,\varpi))\\
    & \;\rightarrow\; \forall A^\Algebra (\PP(A) <_\RR \lambda \to \forall x^\tau\,\exists \varpi^\Omega\in A^c\,\varphi (x,\varpi))\\
    & \;\leftrightarrow\; \forall x^\tau\,\forall A^\Algebra (\PP(A) <_\RR \lambda \to \exists \varpi^\Omega\in A^c\,\varphi (x,\varpi))\\
    & \;\equiv\; \forall x^\tau\left(\PP[\varphi (x)]\geq\lambda\right),
    \end{align*}
    where all the (bi-)implications are constructively true.
    \item We have, by Proposition \ref{pro:duality}, that
    \begin{align*}
    \forall \mu^\mathbb{Q}> 0\, \exists x^\tau(\PP[\varphi(x)]\leq \lambda + \mu) 
    & \; \equiv \;  \forall \mu^\mathbb{Q}> 0\, \exists x^\tau\, \forall A^\Algebra (\PP(A) >_\RR \lambda+\mu \to \exists \varpi^\Omega\in A\,\overline{\varphi} (x,\varpi))\\
    & \; \rightarrow \; \forall \mu^\mathbb{Q}> 0\, \forall A^\Algebra (\PP(A) >_\RR \lambda+\mu \to \exists \varpi^\Omega\in A\,\exists x^\tau\,\overline{\varphi} (x,\varpi))\\
    & \; \leftrightarrow \; \forall A^\Algebra (\exists \mu^\mathbb{Q}> 0\left( \PP(A) >_\RR \lambda+\mu\right) \to \exists \varpi^\Omega\in A\,\exists x^\tau\overline{\varphi} (x,\varpi))\\
    & \; \leftrightarrow \; \forall A^\Algebra (\PP(A) >_\RR \lambda \to \exists \varpi^\Omega\in A\,\overline{\forall x^\tau\varphi(x,\varpi)})\\
    & \; \equiv \; \PP[\forall x^\tau\varphi(x)]\leq\lambda,
    \end{align*}
    where again all the (bi-)implications are constructively true.\qedhere
\end{enumerate}
\end{proof}

While the above principles are generally valid, the converse variants of these results are only valid under a monotonicity assumption, and in particular only provable in the presence of a uniform boundedness principle. Moreover, in the intuitionistic context, the converse statements subtly rely either on uniform boundedness or contra-collection, depending on polarity, which are conflated in the classical context. As the picture here starts to become more complicated, we hence first focus on the classical case. As only a fragment of the full uniform boundedness principle is permissible in a classical context while aiming for computable bounds, we have to make restrictions on the quantifier complexity of the formulas in that context.

\begin{proposition} \label{pro:furtherPrenex}
Over $\mathcal{F}^\omega[\PP]$, let $\tau$ be a pure type, let $\varphi(x^\tau, \varpi^\Omega)$ be a formula and let $\lambda^1$ be an arbitrary free variable. Then:
\begin{enumerate}
\item If $\varphi$ is a $\forall$-formula and if $\varphi(x,\varpi)$ is anti-monotone in $x$, i.e.
\[
\forall \varpi^\Omega\,\forall x^\tau,y^\tau(y\geq_\tau x\land\varphi(y,\varpi) \to \varphi(x,\varpi)),
\]
then $\mathcal{F}^\omega[\PP] + \UBOmega{\tau}$ proves
\[
\PP[\forall x^\tau\varphi(x)]\ge \lambda \; \leftrightarrow \; \forall x^\tau(\PP[\varphi(x)]\ge \lambda)
\]
and $\mathcal{F}^\omega[\PP] + \UBF{\tau}$ proves
\[
\forall \mu^\mathbb{Q}> 0\, \exists x^\tau(\PP[\varphi(x)]< \lambda + \mu) \; \leftrightarrow \; \PP[\forall x^\tau\varphi(x)]\leq \lambda.
\]
\item If $\varphi$ is an $\exists$-formula and if $\varphi(x,\varpi)$ is monotone in $x$, i.e.
\[
\forall \varpi^\Omega\,\forall x^\tau,y^\tau(y\geq_\tau x\land\varphi(x,\varpi) \to \varphi(y,\varpi)),
\]
then $\mathcal{F}^\omega[\PP] + \UBOmega{\tau}$ proves
\[
\PP[\exists x^\tau\varphi(x)]\leq \lambda \; \leftrightarrow \; \forall x^\tau(\PP[\varphi(x)]\leq \lambda)
\]
and $\mathcal{F}^\omega[\PP] + \UBF{\tau}$ proves
\[
\forall \mu^\mathbb{Q}> 0\, \exists x^\tau(\PP[\varphi(x)]> \lambda - \mu) \; \leftrightarrow \; \PP[\exists x^\tau\varphi(x)]\ge \lambda.
\]
\end{enumerate}
\end{proposition}
\begin{proof}
We only prove (1). Item (2) again follows using \eqref{def-P-lt}. For the first part of item (1), by Proposition \ref{pro:initialPrenex} we are left with showing
\[
\forall x^\tau(\PP[\varphi(x)]\ge \lambda)\to \PP[\forall x^\tau\varphi(x)]\ge \lambda. 
\]
Suppose $\forall x^\tau(\PP[\varphi(x)]\ge \lambda)$, i.e.
\[ 
    \forall x^\tau\,\forall A^\Algebra(\PP(A) <_\RR \lambda \to \exists \varpi^\Omega \in A^c \, \varphi(x,\varpi)).
\]
Using classical logic, this is equivalent to
\[
    \forall A^\Algebra(\exists x^\tau\,\forall \varpi^\Omega(\varphi(x,\varpi) \to \varpi \in A) \to \PP(A) \ge_\RR \lambda).
\]
Now fix $A$ and assume $\forall \varpi^\Omega\, \exists x^\tau(\varphi(x,\varpi) \to \varpi \in A)$. Using $\UBOmega{\tau}$, we get
\[
\exists \chi^\tau\,\forall\varpi^\Omega\, \exists x\leq_\tau\chi\left(\varphi(x,\varpi)\to \varpi\in A\right)
\]
as $\varphi$ is a $\forall$-formula. As $\varphi$ is anti-monotone, we get 
\[
\exists x^\tau\,\forall\varpi^\Omega\left(\varphi(x,\varpi)\to \varpi\in A\right).
\]
By the above, we get $\PP(A) \geq_\mathbb{R} \lambda$. Hence, we have shown
\[ \forall A^\Algebra(\forall \varpi^\Omega(\forall x^\tau\, \varphi(x,\varpi) \to \varpi \in A) \to \PP(A) \ge_\RR \lambda) \]
which is classically equivalent to $\PP[\forall x^\tau\varphi(x)]\ge \lambda$. \\[2mm]
For the second part of item (1), by Proposition \ref{pro:initialPrenex}, we are similarly left with showing 
\[
\PP[\forall x^\tau\varphi(x)]\leq \lambda\to \forall \mu^\mathbb{Q}> 0\, \exists x^\tau(\PP[\varphi(x)]< \lambda + \mu). 
\]
So suppose $\PP[\forall x^\tau\varphi(x)]\leq \lambda$. Using classical logic and Proposition \ref{pro:duality}, this is equivalent to
\[
    \forall A^\Algebra(\forall \varpi^\Omega(\varpi \in A \to \forall x^\tau\,\varphi(x,\varpi) ) \to \PP(A) \le_\RR \lambda)
\]
and hence
\[
    \forall \mu^\mathbb{Q}> 0\, \forall A^\Algebra\,\exists x^\tau(\forall \varpi^\Omega(\varpi \in A \to \varphi(x,\varpi) ) \to \PP(A) <_\RR \lambda+\mu).
\]
As $\varphi$ is a $\forall$-formula, $\UBF{\tau}$ yields 
\[
\forall \mu^\mathbb{Q}> 0\, \exists \chi^\tau\, \forall A^\Algebra\,\exists x\leq_\tau\chi(\forall \varpi^\Omega(\varpi \in A \to \varphi(x,\varpi) ) \to \PP(A) <_\RR \lambda+\mu)
\]
and anti-monotonicity of $\varphi$ yields
\[
\forall \mu^\mathbb{Q}> 0\, \exists x^\tau\, \forall A^\Algebra(\forall \varpi^\Omega(\varpi \in A \to \varphi(x,\varpi) ) \to \PP(A) <_\RR \lambda+\mu)
\]
which is classically equivalent, using Proposition \ref{pro:duality}, to $\forall \mu^\mathbb{Q}> 0\, \exists x^\tau(\PP[\varphi(x)]< \lambda + \mu)$.
\end{proof}

We now consider the constructive case. There, the second equivalence of both items of Proposition \ref{pro:furtherPrenex} remains admissible in the presence of the constructively (at least as far as the Dialectica interpretation is concerned) acceptable principle of independence of premise for universal formulas. Further, the restrictions previously in place on the uniform boundedness principle can now be lifted as uniform boundedness is admissible without restrictions on the formula in a semi-constructive context while aiming for the extraction of computable bounds.

\begin{proposition}\label{pro:furtherPrenexConst}
Over $\mathcal{F}^\omega_i[\PP]$, let $\tau$ be a pure type, let $\varphi(x^\tau, \varpi^\Omega)$ be any formula and let $\lambda^1$ be an arbitrary free variable. Then:
\begin{enumerate}
\item If $\varphi$ is anti-monotone in $x$, then $\mathcal{F}^\omega_i[\PP] + \FullUBF{\tau}$ proves
\[
\forall \mu^\mathbb{Q}> 0\, \exists x^\tau(\PP[\varphi(x)]< \lambda + \mu) \leftrightarrow \PP[\forall x^\tau\varphi(x)]\leq \lambda.
\]
\item If $\varphi$ is monotone in $x$, then $\mathcal{F}^\omega_i[\PP] + \FullUBF{\tau}$ proves
\[
\forall \mu^\mathbb{Q}> 0\, \exists x^\tau(\PP[\varphi(x)]> \lambda - \mu) \leftrightarrow \PP[\exists x^\tau\varphi(x)]\ge \lambda.
\]
\end{enumerate}
\end{proposition}
\begin{proof}
As before, we only need to prove one of the items as the other follows using \eqref{def-P-lt}. Here, we only show (2). By Proposition \ref{pro:initialPrenex}, also similar to before, we are only left to show 
\[
\PP[\exists x^\tau\,\varphi(x)]\ge \lambda\rightarrow 
\forall \mu>_\RR 0\, \exists x^\tau(\PP[\varphi(x)]> \lambda - \mu).
\]
For that, note that we have
\begin{align*}
\PP[\exists x^\tau\varphi(x)]\ge \lambda
& \;\equiv\; \forall A^\Algebra (\PP(A) <_\RR \lambda \to \exists \varpi^\Omega\in A^c\,\exists x^\tau\varphi(x,\varpi))\\
& \;\leftrightarrow\; \forall A^\Algebra (\exists \mu^\mathbb{Q}>0\left(\PP(A) \leq_\RR \lambda-\mu\right) \to \exists x^\tau\,\exists \varpi^\Omega\in A^c\,\varphi(x,\varpi))\\
& \;\leftrightarrow\; \forall \mu^\mathbb{Q}>0\,\forall A^\Algebra (\PP(A) \leq_\RR \lambda-\mu \to \exists x^\tau\,\exists \varpi^\Omega\in A^c\,\varphi(x,\varpi))\\
& \;\leftrightarrow\; \forall \mu^\mathbb{Q}>0\,\forall A^\Algebra\, \exists x^\tau(\PP(A) \leq_\RR \lambda-\mu \to \exists \varpi^\Omega\in A^c\,\varphi(x,\varpi)),
\end{align*}
where all equivalences are constructive but the last, which follows using $\mathrm{IP}^\omega_\forall$. Using $\UBF{\tau}$, we get 
\[
\forall \mu^\mathbb{Q}>0\,\exists \chi^\tau\, \forall A^\Algebra\, \exists x\leq_\tau \chi(\PP(A) \leq_\RR \lambda-\mu \to \exists \varpi^\Omega\in A^c\,\varphi(x,\varpi))
\]
and as $\varphi$ is monotone, we have
\[
\forall \mu^\mathbb{Q}>0\,\exists x^\tau\, \forall A^\Algebra(\PP(A) \leq_\RR \lambda-\mu \to \exists \varpi^\Omega\in A^c\,\varphi(x,\varpi))
\]
which is $\forall \mu^\mathbb{Q}> 0\, \exists x^\tau(\PP[\varphi(x)]> \lambda - \mu)$.
\end{proof}

Now, the first equivalence of both items of Proposition \ref{pro:furtherPrenex} was previously derived using classical logic, essentially invoking the converse of the uniform boundedness principle. Intuitionistically, this essentially just amounts to an application of the principle of contra-collection, which in the context of the present paper is, however, restricted not only in terms of the permissible formulas but also regarding the type, resulting in the following rather restricted result:

\begin{proposition}\label{pro:evenFurtherPrenexConst}
Over $\mathcal{F}^\omega_i[\PP]$, let $\varphi_0(x^0, \varpi^\Omega)$ be a quantifier-free formula and let $\lambda^1$ be an arbitrary free variable. Then:
\begin{enumerate}
\item If $\varphi_0$ is anti-monotone in $x$, then $\mathcal{F}^\omega_i[\PP] \oplus \CCOmega{0}$ proves
\[
\PP[\forall x^0\varphi_0(x)]\ge \lambda \; \leftrightarrow \; \forall x^0(\PP[\varphi_0(x)]\ge \lambda).
\]
\item If $\varphi_0$ is monotone in $x$, then $\mathcal{F}^\omega_i[\PP] \oplus \CCOmega{0}$ proves
\[
\PP[\exists x^0\varphi_0(x)]\leq \lambda \; \leftrightarrow \; \forall x^0(\PP[\varphi_0(x)]\leq \lambda).
\]
\end{enumerate}
\end{proposition}
\begin{proof}
As before, we only need to prove one of the items as the other follows using \eqref{def-P-lt}. Here, we only show (1). By Proposition \ref{pro:initialPrenex}, also similar to before, we are only left to show 
\[
\forall x^0(\PP[\varphi_0(x)]\ge \lambda)\rightarrow \PP[\forall x^0\varphi_0(x)]\ge \lambda.
\]
For that, note that we have
\begin{align*}
\forall x^0(\PP[\varphi_0(x)]\ge \lambda)& \;\equiv\; \forall x^0\forall A^\Algebra (\PP(A) <_\RR \lambda \to \exists \varpi^\Omega\in A^c\,\varphi_0(x,\varpi))\\
& \;\leftrightarrow\; \forall A^\Algebra (\PP(A) <_\RR \lambda \to \forall x^0\,\exists \varpi^\Omega\in A^c\,\varphi_0(x,\varpi))\\
& \;\leftrightarrow\; \forall A^\Algebra (\PP(A) <_\RR \lambda \to \forall x^0\,\exists \varpi^\Omega\in A^c\,\forall y\leq_0 x\,\varphi_0(y,\varpi)),
\end{align*}
where the last line follows from the anti-monotonicity of $\varphi_0$. Using $\CCOmega{0}$, we get 
\[
\forall x^0\exists \varpi^\Omega\in A^c\forall y\leq_0 x\,\varphi_0(y,\varpi)\to \exists \varpi^\Omega\in A^c\,\forall x^0\varphi_0(x,\varpi),
\]
so that $\forall x^0(\PP[\varphi_0(x)]\ge \lambda)$ implies 
\[
\forall A^\Algebra (\PP(A) <_\RR \lambda \to \exists \varpi^\Omega\in A^c\,\forall x^0\varphi_0(x,\varpi)),
\]
i.e. $\PP[\forall x^0\varphi_0(x)]\ge \lambda$, as claimed.
\end{proof}

While the restriction of the above result to quantifier-free formulas seems quite essential, the restriction to type $0$ is entirely dependent on the fact that we only treated $\CCOmega{0}$ and $\CCF{0}$. Indeed, it is easy to see that the above proof immediately generalizes whenever contra-collection of higher types is available. This would be an advantage of the bounded functional interpretation \cite{FerreiraOliva2005}, where an analogous contra-collection principle is immediately available for arbitrary types without arduous pre- and post-processing, at the expense of using an intensional majorizability relation of course.

\begin{remark}\label{rem:prenex}
Such prenexation laws as the above are formulated mathematically, for general (measurable) formulas, for example in Lemma 3.1 in \cite{NeriPowell2025}. Concretely, let $(\Omega,\Algebra,\PP)$ be a probability space and write $\forall n\,\varphi(n)=\bigcap_{n\in\mathbb{N}}\varphi(n)$ as well as $\exists n\,\varphi(n)=\bigcup_{n\in\mathbb{N}}\varphi(n)$ for any measurable $\varphi$, that is $\varphi(n)\in\Algebra$ for all $n\in\mathbb{N}$. Then, Lemma 3.1 in \cite{NeriPowell2025} proves semantically that 
\begin{enumerate}
\item $\PP(\forall n\, \varphi(n))\geq p$ if and only if $\forall n(\PP(\varphi(n))\geq p)$,
\item $\PP(\forall n\, \varphi(n))\leq p$ if and only if $\forall \lambda >0\exists n(\PP(\varphi(n))<p+\lambda)$,
\end{enumerate}
for any measurable anti-monotone $\varphi$ and probability $p\in [0,1]$, and 
\begin{enumerate}
\item $\PP(\exists n\, \varphi(n))\leq p$ if and only if $\forall n(\PP(\varphi(n))\leq p)$,
\item $\PP(\exists n\, \varphi(n))\geq p$ if and only if $\forall \lambda >0\exists n(\PP(\varphi(n))>p-\lambda)$,
\end{enumerate}
for any measurable and monotone $\varphi$ and probability $p\in [0,1]$. However, this result only holds when $\PP$ is a measure and is indeed established in \cite{NeriPowell2025} using limit theorems for $\PP$, so that these principles are generally not true for contents. Even further, the laws formulated above are, in their generality, not even true for the outer content associated with a probability content. Strikingly, the previous propositions hence show that variants of such prenexation laws are admissible in proofs for the purpose of extracting computable information. As our systems are based on contents, the resulting consequences will, further, be valid for general contents. So, uniform boundedness allows one to falsely treat contents, to some degree, like a measure in proofs, all the while retaining meaningful bound extraction theorems with valid conclusions over content spaces. In these contexts, we are in particular not restricted to the type $0$ in our quantifiers, and can even allow the use of uncountable quantifiers.
\end{remark}

\begin{remark}\label{rem:MarkovRemark}
Note that out of the semi-constructive results above, only Proposition \ref{pro:furtherPrenexConst} requires the use of $\IPU$. Further, none of the results use $\AC$ or $\MARK$ in any form. Also, none of the classical results use $\DC$.
\end{remark}

As a last result of the present section, we now provide a type of conservativity result for the systems considered here regarding the principle of $\sigma$-additivity by relating it, in a classical context, to the uniform boundedness principles. Here, $\sigma$-additivity concretely takes the form of
\[
\PP\left(\bigcup_{n\in\mathbb{N}} A_n \right) = \sum_{n\in\mathbb{N}} \PP (A_n),
\]
where $A_n$ are pairwise disjoint events. First note that, using the monotonicity of the series $\sum_{n=0}^l \PP(A_n)$, the above statement is equivalent to the conjunction of the two principles
\[
\forall l\in\mathbb{N} \left(\PP\left(\bigcup_{n\in\mathbb{N}}  A_n \right) \geq \sum_{n=0}^l \PP ( A_n)\right)\text{ and }\forall k\in\mathbb{N} \exists l\in\mathbb{N} \left(\PP\left(\bigcup_{n\in\mathbb{N}}  A_n \right) \leq \sum_{n=0}^l \PP ( A_n) + 2^{-k}\right).
\]
To formulate this property formally, we first recall from \cite{NeriPischke2025} the following operation on terms of type $S(0)$ that allows for the implicit quantification over a disjoint countable family of sets: Given $A^{S(0)}$, we set $(A\!\uparrow)_0:=A_0$ and
\[
(A\!\uparrow)_{n+1}:=A_{n+1}\cap \left(\bigcup^n_{i=0} A_i\right)^c.
\]
This operation thus turns $A$, representing an arbitrary sequence of measurable sets, into a sequence of disjoint sets $A\!\uparrow$ with the same (partial) union(s). Further, if $A$ was already a disjoint family, then it is left unchanged by the operation (see also the discussion in \cite{NeriPischke2025}).

The fact that $\varpi$ is included in the union $\bigcup_{n\in\mathbb{N}} A_n$ can now be expressed using the outer content via the formula $\exists n^0 (\varpi \in A_n)$. The first inequality above then corresponds to the statement
\[
\forall A^{S(0)} \forall l^0 \left(\PP\left[\exists n^0 A_n\right] \geq \sum_{n=0}^l \PP ((A\!\uparrow)_n)\right)
\]
while the second can be formally recognized via
\[
\forall A^{S(0)}\forall k^0 \exists l^0 \left( \PP\left[\exists n^0 A_n\right] \leq \sum_{n=0}^l\PP ((A\!\uparrow)_n) + 2^{-k}\right).
\]
We denote the conjunction of these two principles by ($\sigma$-\textsf{additivity}). The key result regarding this principle is then the following:

\begin{proposition}\label{Prop:elim:sigma} The system $\mathcal{F}^\omega[\PP] +\UBOmega{0}$ proves ($\sigma$-\textsf{additivity}).
\end{proposition}
\begin{proof}
Abusing notation slightly, note that
\[
\sum_{n=0}^l \PP ((A\!\uparrow)_n)=_\RR\PP\left(\bigcup_{n=0}^lA_n\right)\leq \PP[\exists n\leq_0 l A_n]
\]
is provable in $\mathcal{F}^\omega[\PP]$ (recall Proposition \ref{pro:measurable}), so that the first part of ($\sigma$-\textsf{additivity}) is similarly provable in $\mathcal{F}^\omega[\PP]$, using Proposition \ref{pro:implicationOuter} and that $\exists n\leq_0 l (\varpi\in A_n)$ implies $\exists n^0 (\varpi\in A_n)$. Hence, we turn to the second part of ($\sigma$-\textsf{additivity}). Let $A^{S(0)}$ and $k^0$ be given. The partial sums $\sum_{n=0}^{l} \PP ((A\!\uparrow)_n)=_\RR\PP(\bigcup_{n=0}^{l} A_n )$ now form a monotone sequence of nonnegative real numbers, bounded above by $1$. Since $\mathcal{F}^\omega[\PP]$ (already without $\DC$) proves the Cauchy formulation of the monotone convergence theorem (see e.g.\ Proposition 4.5 in \cite{NeriPischke2025}), there exists an $l^0$ such that 
\[
\forall m^0\left( \PP\left(\bigcup_{n=0}^{l+ m} A_n \right) \leq_\RR \sum_{n=0}^{l} \PP ((A\!\uparrow)_n) + 2^{-k}\right).
\]
Using Proposition \ref{pro:measurable}, we in particular get 
\[
\forall m^0\left( \PP[\exists n \le_0 (l+m)\, A_n] \leq \sum_{n=0}^{l} \PP ((A\!\uparrow)_n) + 2^{-k}\right).
\]
Now part (2) of Proposition \ref{pro:furtherPrenex} yields (observing the required monotonicity property in $m$), that with $\UBOmega{0}$ we have 
\[
\PP[\exists m^0\exists n \le_0 (l+m)\, A_n] \leq \sum_{n=0}^{l} \PP ((A\!\uparrow)_n)) + 2^{-k}
\]
which, combined with Proposition \ref{pro:implicationOuter}, since $\exists m^0\exists n\leq_0 (l+m) \, (\varpi\in A_n)$ is equivalent to $\exists n^0 (\varpi\in A_n)$, yields 
\[
\PP[\exists n^0 A_n] \leq \sum_{n=0}^{l} \PP ((A\!\uparrow)_n) + 2^{-k},
\]
as required.
\end{proof}

Proposition \ref{Prop:elim:sigma} in particular shows that any abstract use of $\sigma$-additivity, that is $\sigma$-additivity formulated for measurable sets via the abstract type $\Algebra$, is eliminated via a proof mining analysis at the expense of a simple functional iteration akin to the monotone convergence theorem. Naturally, for more complex measurable sets which are not immediately represented via terms of type $\Algebra$ (i.e.\ via quantifier-free formulas), the resulting principle of $\sigma$-additivity might of course still carry a lot of computational strength, akin to the limit theorems utilized previously to interchange quantifiers with the (outer) content.

Nevertheless, this crucial observation that an \emph{abstract} use as above remains computationally tame, finally provides a formal explanation and justification for the (up to this point only empirically observed) phenomenon that proof mining seems to eliminate such uses of this axiom in the course of an analysis. As such, it in particular formally illustrates that probability contents indeed seem to be the finitary core of probability theory through the lens of proof mining, as conjectured in \cite{NeriPischke2025} based on the case studies available at the time. In particular, in light of Theorem \ref{thm:metaClassical}, we have that probabilistic theorems provably equivalent over $\mathcal{F}^\omega[\PP]+(\sigma\text{-}\textsf{additivity})$ to $\forall\exists$-statements remain true over finitely additive spaces.

\section{Logical metatheorems for probabilistic theorems}\label{sec:meta}

In this section, we illustrate how our formal translation based on the outer measure can be used to systematically associate computational content to probabilistic statements, and how this association can be used in combination with the logical machinery underlying the logical metatheorems for proof mining in probability theory due to the first and third author to extract that content from proofs.

Our discussion will primarily be focused on probabilistic $\forall\exists$-theorems, that is theorems of the form
\[
\text{``$\forall x^\tau\exists n^0\varphi(x,n)$ with probability $\geq\lambda$''},\tag{$*$}\label{originalProblem}
\]
where, at first, we want to think of $\tau$ as a generic admissible type and $\varphi$ as an arbitrary formula, although we later also discuss various special cases.

There are three different, but a priori equally natural, ways to express this property formally in our system, and we now want to motivate them in what follows. The most naive way is to simply wrap the $\forall\exists$-statement in the outer content introduced earlier, leading to
\[
\PP[\forall x^\tau\exists n^0\varphi(x,n)]\geq\lambda.\tag*{$(*)_1$}\label{strongProblem}
\]
From a purely mathematical perspective, at least over measure spaces when $\tau=0$, and if the statement $\varphi(x,n)$ is measurable and anti-monotone in $x$, the leading universal quantifier of that statement can be prenexed out of the probability (recall Remark \ref{rem:prenex}), leading to a statement of the form 
\[
\forall x^\tau \left( \PP[\exists n^0\varphi(x,n)]\geq\lambda\right). \tag*{$(*)_2$}\label{problem}
\]
If, moreover, $\varphi(x,n)$ is assumed to be monotone in $n$, then this statement can be further prenexed at the expense of a probabilistic error, leading to a statement of the form
\[
\forall m^0\forall x^\tau\exists n^0 \left( \PP[\varphi(x,n)]\geq\lambda-2^{-m} \right).\tag*{$(*)_3$}\label{uniformProblem}
\]
While thereby naturally related in the mathematical context of measures, for suitably monotone and measurable properties and sufficiently low types, these statements all provide, in one way or another, suitable formal representations of the previously mentioned theorems of the type \eqref{originalProblem}. 

Presented with three choices for a formal access towards probabilistic $\forall\exists$-theorems, we are now faced with three immediate questions:
\begin{enumerate}
\item What is the computational content of theorems of the form \ref{strongProblem} -- \ref{uniformProblem} and when can it be extracted from corresponding (classical or constructive) proofs?
\item How are the statements \ref{strongProblem} -- \ref{uniformProblem} and their computational interpretations related, potentially irrespective of monotonicity assumptions as used above?
\item Which formulations naturally occur in practical case studies from probability theory and statistics?
\end{enumerate}

The next three parts of the present section provide a comprehensive answer to the first question by presenting corresponding metatheorems on the extraction of bounds tailored to theorems of the form \ref{strongProblem} -- \ref{uniformProblem}. In that way, we extend the existing arsenal of proof mining metatheorems with dedicated tools to handle all the different kinds of probabilistic theorems illustrated above. As in the work \cite{NeriPischke2025}, and as already highlighted before, the abstract treatment of the sample and event spaces, and in particular their uniform majorizability, guarantees that the resulting bounds will be highly uniform, not depending on any parameters related to the underlying probability content space, such as its measure. In the classical case, $\varphi$ will be restricted to be purely existential, while there are no such formula restrictions in the constructive case. In general, the type $\tau$ of $x$ will be allowed to exceed $0$, as long as it remains reasonably low (as is typical for proof mining metatheorems, recall Section \ref{sec:prelim}). Throughout, we only consider the case of the outer content and high probability, as the case of the inner content and low probability can be treated using duality. Strikingly, the statements \ref{strongProblem} -- \ref{uniformProblem} exhibit very different quantitative behavior, illustrating how the transformations (based on the continuity of the measure) used above to motivate the various formulations carry heavy computational strength in general. As such, the corresponding metatheorems get progressively more difficult to establish, requiring more and more theory beyond \cite{NeriPischke2025}. In particular the second question on the relation between \ref{strongProblem} -- \ref{uniformProblem} in that way, turns out to be rather subtle and complex, so that the last part of the present section is dedicated solely to this discussion.

The third question will be addressed in the following and final Section \ref{sec:examples}, where various exemplary applications from mathematical practice are discussed, and where in particular these formulations and their computational content will be linked with the by now quite considerable range of quantitative theorems from probability theory and statistics derived via the proof mining program, as already briefly mentioned in the introduction.

\subsection{Probabilistic theorems of the form $\forall x^\tau \left( \PP[\exists n^0\varphi(x,n)]\geq\lambda\right)$}

We begin with probabilistic bound extraction metatheorems for statements of the form \ref{problem}. In both the classical and semi-constructive context, the results are rather immediate applications of the resulting classical and semi-constructive metatheorems presented in Section \ref{sec:prelim}.

Let us consider first the classical setting, where $\varphi(x, n)$ is assumed to be an existential formula $\varphi_\exists(x, n)$. 

\begin{theorem}\label{thrm:pulloutpi2}
In the context of the assumptions of Theorem \ref{thm:metaClassical}, suppose that
\[
\mathcal{F}^{\omega}[\PP]+\UBOmega{\omega}+\UBF{\omega}\vdash \forall x^\tau \left( \PP[\exists n^0 \varphi_\exists(x,n)]\geq\lambda \right),
\] 
where $\tau$ is admissible and $\varphi_\exists(x,n,\varpi)$ is an existential formula such that all internal quantifiers have admissible types. Then, from this proof one can extract a partial functional $\Phi:\mathbb{N}\times\mathcal{S}_{\widehat{\tau}}\rightharpoonup\mathbb{N}$ which is total and bar-recursively computable on $\mathbb{N}\times\mathcal{M}_{\widehat{\tau}}$ and such that
\[
\mathcal{S}^{\omega, \Omega, \Algebra} \models \forall m^0\forall {x^*}^{\widehat{\tau}} \forall x\lesssim_\tau x^*\left(\PP[\exists n\leq_0\Phi(m,x^*)\, \varphi_\exists(x,n)]\geq \lambda-2^{-m}\right)
\]
holds for $\mathcal{S}^{\omega, \Omega, \Algebra}$ defined via any non-empty set $\Omega$, algebra $S\subseteq 2^\Omega$, and probability content $\PP$ on $\Algebra$ (and with suitable interpretations of the additional constants).

The same remarks (1) -- (5) as in Theorem \ref{thm:metaClassical} apply.
\end{theorem}
\begin{proof} Expanding the abbreviation, we have that $\forall x^\tau (\PP[\exists n^0 \varphi_\exists(x,n)]\geq\lambda)$ is equivalent to 
\[
\forall x^\tau\forall m^0\forall A^\Algebra\exists n^0(\PP(A)\leq_\RR \lambda -2^{-m}\to \exists \varpi\in A^c \, \varphi_\exists(x,n,\varpi)).
\]
As the inner formula is existential with suitable type restrictions, we can then apply Theorem \ref{thm:metaClassical} to extract a suitable partial functional $\Phi:\mathbb{N}\times\mathcal{S}_{\widehat{\tau}}\rightharpoonup\mathbb{N}$ such that
\[
\mathcal{S}^{\omega, \Omega, \Algebra} \models \forall m^0\forall {x^*}^{\widehat{\tau}} \forall x\lesssim_\tau x^* \forall A^\Algebra\exists n\leq_0\Phi(m,x^*) (\PP(A)\leq_\RR \lambda -2^{-m}\to \exists \varpi\in A^c\,\varphi_\exists(x,n,\varpi))
\] 
for all suitable models $\mathcal{S}^{\omega, \Omega, \Algebra}$. This in particular implies
\[
\mathcal{S}^{\omega, \Omega, \Algebra} \models \forall m^0\forall {x^*}^{\widehat{\tau}} \forall x\lesssim_\tau x^* \left(\PP[\exists n\leq_0\Phi(m,x^*)\, \varphi_\exists(x,n)]\geq \lambda-2^{-m}\right)
\]
as claimed.
\end{proof}

Note that Theorem \ref{thrm:pulloutpi2} also applies to statements of the form \ref{strongProblem}, that is of the form $\PP[\forall x^\tau\exists n^0 \varphi_\exists(x,n)]\geq\lambda$, since, by Proposition \ref{pro:initialPrenex}, this implies $\forall x^\tau \left( \PP[\exists n^0 \varphi_\exists(x,n)]\geq\lambda \right)$ already over $\mathcal{F}^\omega[\PP]$.

Let us now consider the case when statements of the form \ref{problem} have been proved in a semi-constructive setting. As in Theorem \ref{thm:metaConstructive}, the semi-constructive context allows for greater flexibility in the assumptions.

\begin{theorem}\label{thrm:pulloutpi2-const}
In the context of the assumptions of Theorem \ref{thm:metaConstructive}, suppose that
\[
\mathcal{F}_i^{\omega}[\PP]+\FullUBOmega{\omega}+ \FullUBF{\omega} \vdash \forall x^\tau \left( \PP[\exists n^0 \varphi(x,n)]\geq\lambda \right),
\] 
where $\tau$ is admissible and $\varphi(x,n,\varpi)$ is any formula such that all positively occurring universal quantifiers and all negatively occurring existential quantifiers have small types and all other types are admissible.

Then one can extract a primitive-recursive functional $\Phi:\mathbb{N}\times\mathcal{S}_{\widehat{\tau}}\to\mathbb{N}$ such that 
\[
\mathcal{S}^{\omega, \Omega, \Algebra}\models \forall m^0\forall {x^*}^{\widehat{\tau}}\forall x\lesssim_\tau x^*\left(\PP[\exists n\leq_0\Phi(m,x^*)\, \varphi(x,n)]\geq \lambda-2^{-m}\right)
\]
holds for $\mathcal{S}^{\omega, \Omega, \Algebra}$ defined via any non-empty set $\Omega$, algebra $S\subseteq 2^\Omega$, and probability content $\PP$ on $\Algebra$ (and with suitable interpretations of the additional constants).

If the theorem is established over $\mathcal{F}_i^{\omega}[\PP] + \FullUBOmega{\omega} + \FullUBF{\omega} \oplus\CCOmega{0} \oplus\CCF{0}$ instead, then the result remains true if $\varphi(x,n,\varpi)$ is a quantifier-free formula.

We have the same remarks (1) -- (3) as in Theorem \ref{thm:metaConstructive}.
\end{theorem}
\begin{proof} As in the proof of Theorem \ref{thrm:pulloutpi2}, note that, in $\mathcal{F}_i^{\omega}[\PP]+\FullUBOmega{\omega} + \FullUBF{\omega}$, the statement $\forall x^\tau \left(\PP[\exists n^0 \varphi(x,n)]\geq\lambda\right)$ is equivalent to
\[
\forall x^\tau\forall m^0\forall A^\Algebra\exists n^0(\PP(A)\leq_\RR \lambda -2^{-m}\to \exists \varpi\in A^c \, \varphi(x,n,\varpi)).
\]
The first result now follows from Theorem \ref{thm:metaConstructive} and the second result, regarding contra-collection, follows from Theorem \ref{thm:metaConstructiveCC}.
\end{proof}

As before, Theorem \ref{thrm:pulloutpi2-const} also applies to statements of the form \ref{strongProblem}, i.e. $\PP[\forall x^\tau\exists n^0 \varphi(x,n)] \geq \lambda$, since, by Proposition \ref{pro:initialPrenex}, this provably implies $\forall x^\tau \left(\PP[\exists n^0 \varphi(x,n)]\geq\lambda\right)$ even in a constructive setting.

\subsection{Probabilistic theorems of the form $\forall m^0\forall x^\tau\exists n^0 \left( \PP[\varphi(x,n)]\geq\lambda-2^{-m} \right)$}

We now move to statements of the form \ref{uniformProblem}, which, in general (without any monotonicity assumptions), are stronger than theorems of the form \ref{problem} (recall Proposition \ref{pro:initialPrenex}). Nevertheless, they obey similarly nice bound extraction theorems, which are also readily established. However, in the classical case we do rely on epsilon-terms in the style of \cite{GuenzelKohlenbach2016}, and here provide a bit more detail than before, were they were only mentioned in passing. In that classical context, we then in particular assume that $\varphi$ is quantifier-free, although this requirement can be somewhat relaxed.

\begin{theorem}\label{thrm:pulloutpi2Uniform}
In the context of the assumptions of Theorem \ref{thm:metaClassical}, suppose that
\[
\mathcal{F}^{\omega}[\PP]+\UBOmega{\omega}+\UBF{\omega}\vdash \forall m^0 \forall x^\tau \exists n^0 \left( \PP[\varphi_0(x,n)] \geq \lambda - 2^{-m} \right),
\] 
where $\tau$ is admissible and $\varphi_0(x,n,\varpi)$ is quantifier-free. 

Then one can extract a partial functional $\Phi:\mathbb{N}\times\mathcal{S}_{\widehat{\tau}}\rightharpoonup\mathbb{N}$ which is total and bar-recursively computable on $\mathbb{N}\times\mathcal{M}_{\widehat{\tau}}$ and such that
\[
\mathcal{S}^{\omega, \Omega, \Algebra} \models \forall m^0\forall {x^*}^{\widehat{\tau}} \forall x\lesssim_\tau x^* \exists n\leq_0\Phi(m,x^*) \left(\PP[\varphi_0(x,n)]\geq \lambda-2^{-m}\right),
\]
holds for $\mathcal{S}^{\omega, \Omega, \Algebra}$ defined via any non-empty set $\Omega$, algebra $S\subseteq 2^\Omega$, and probability content $\PP$ on $\Algebra$ (and with suitable interpretations of the additional constants).

The same remarks (1) -- (5) as in Theorem \ref{thm:metaClassical} apply.
\end{theorem}
\begin{proof}
The abbreviation $\forall m^0 \forall x^\tau \exists n^0 \left( \PP[\varphi_0(x,n)] \geq \lambda - 2^{-m} \right)$ is equivalent to 
\[
\forall x^\tau\forall m^0\exists n^0\forall A^\Algebra(\PP(A)\leq_\RR \lambda -2^{-m}\to \exists \varpi\in A^c \, \varphi_0(x,n,\varpi)).
\]
Using additional epsilon-terms (see e.g.\ \cite{GuenzelKohlenbach2016}), we can ``remove'' the bounded quantifier of type $\Algebra$ and regard the formula
\[
\forall A^\Algebra(\PP(A)\leq_\RR \lambda -2^{-m}\to \exists \varpi\in A^c \, \varphi_0(x,n,\varpi))
\]
as purely existential, with respective type restrictions. Concretely, we add a constant $\chi^{S(0)(\tau)(0)(1)}$ to the language that is governed by the following axiom of type $\Delta$:
\[
\begin{cases}
\forall\lambda^1,m^0,x^\tau,n^0\Big((\PP(\chi_{\lambda,m,x,n})\leq_\mathbb{R}\lambda-2^{-m}\rightarrow \exists \varpi\in (\chi_{\lambda,m,x,n})^c\, \varphi_0(x,n,\varpi) )\\
\qquad\rightarrow \forall A^\Algebra(\PP(A)<_\mathbb{R}\lambda-2^{-m}\rightarrow \exists \varpi\in A^c\, \varphi_0(x,n,\varpi))\Big).
\end{cases}\tag{$\chi$}\label{firstepsilonAx}
\]
The intended semantics of that constant is that $\chi_{\lambda,m,x,n}:=A$ for $A$ such that $\PP(A)\leq \lambda-2^{-m}$ and $\forall \varpi\in A^c\, \neg\varphi_0(x,n,\varpi)$, if existent, and $\chi_{\lambda,m,x,n}:=\emptyset$ otherwise. In that way, the constant immediately satisfies the above axiom (for suitably large $m$), and as $\chi$ maps into the type $\Algebra$, it is immediately majorizable. In the theory extended with the constant $\chi$ and the axiom \eqref{firstepsilonAx}, we can then in particular derive
\[
\forall x^\tau\forall m^0\exists n^0(\PP(\chi_{\lambda,m,x,n})\leq_\RR \lambda -2^{-m}\to \exists \varpi\in (\chi_{\lambda,m,x,n})^c \, \varphi_0(x,n,\varpi)).
\]
Applying Theorem \ref{thm:metaClassical} to that statement in this extension, we can extract a suitable partial functional $\Phi:\mathbb{N}\times\mathcal{S}_{\widehat{\tau}}\rightharpoonup\mathbb{N}$ such that, using \eqref{firstepsilonAx}, we have
\[
\mathcal{S}^{\omega, \Omega, \Algebra} \models \forall m^0\forall {x^*}^{\widehat{\tau}} \forall x\lesssim_\tau x^*\exists n\leq_0\Phi(m,x^*) \forall A^\Algebra (\PP(A)<_\RR \lambda -2^{-m}\to \exists \varpi\in A^c\,\varphi_0(x,n,\varpi))
\] 
for all suitable models $\mathcal{S}^{\omega, \Omega, \Algebra}$, which yields the claim.
\end{proof}

\begin{remark}\label{rem:monMeta}
In the context of Theorem \ref{thrm:pulloutpi2}, if $\varphi_\exists(x, n)$ is monotone in $n$ then (by Proposition \ref{pro:implicationOuter}) the functional $\Phi$ would in fact be a witness for $n$, i.e.
\[
\mathcal{S}^{\omega, \Omega, \Algebra} \models \forall m^0\forall {x^*}^{\widehat{\tau}} \forall x\lesssim_\tau x^* \left( \PP[\varphi_\exists(x,\Phi(m,x^*))]\geq \lambda-2^{-m}\right),
\]
which is even stronger than the conclusion of Theorem \ref{thrm:pulloutpi2Uniform}. In the context of Theorem \ref{thrm:pulloutpi2}, this result could then also have been established by first pulling the existential quantifier out of the outer content using uniform boundedness and Proposition \ref{pro:furtherPrenex}. 
\end{remark}

We now move to the semi-constructive case. Here, the result is generally a bit more immediate and we have greater flexibility in the assumptions, as before.

\begin{theorem}\label{thrm:pulloutpi2-constUniform}
In the context of the assumptions of Theorem \ref{thm:metaConstructive}, suppose that
\[
\mathcal{F}_i^{\omega}[\PP]+\FullUBOmega{\omega}+ \FullUBF{\omega} \vdash \forall m^0 \forall x^\tau \exists n^0 \left( \PP[\varphi(x,n)] \geq \lambda - 2^{-m} \right),
\] 
where $\tau$ is admissible and $\varphi(x,n,\varpi)$ is any formula such that all positively occurring universal quantifiers and all negatively occurring existential quantifiers have small types and all other types are admissible.

Then one can extract a primitive-recursive functional $\Phi:\mathbb{N}\times\mathcal{S}_{\widehat{\tau}}\to\mathbb{N}$ such that 
\[
\mathcal{S}^{\omega, \Omega, \Algebra}\models \forall m^0\forall {x^*}^{\widehat{\tau}}\forall x\lesssim_\tau x^* \, \exists n\leq_0\Phi(m,x^*) \left(\PP[\varphi(x,n)]\geq \lambda-2^{-m}\right)
\]
holds for $\mathcal{S}^{\omega, \Omega, \Algebra}$ defined via any non-empty set $\Omega$, algebra $S\subseteq 2^\Omega$, and probability content $\PP$ on $\Algebra$ (and with suitable interpretations of the additional constants).

If the theorem is established over $\mathcal{F}_i^{\omega}[\PP] + \FullUBOmega{\omega} + \FullUBF{\omega} \oplus\CCOmega{0} \oplus\CCF{0}$ instead, then the result remains true if $\varphi(x,n,\varpi)$ is an existential formula with the respective type restrictions.

We have the same remarks (1) -- (3) as in Theorem \ref{thm:metaConstructive}.
\end{theorem}
\begin{proof} 

As in the proof of Theorem \ref{thrm:pulloutpi2Uniform}, note that, in $\mathcal{F}_i^{\omega}[\PP]+\FullUBOmega{\omega} + \FullUBF{\omega}$, the statement $\forall m^0 \forall x^\tau \exists n^0 \left( \PP[\varphi(x,n)] \geq \lambda - 2^{-m} \right)$ is equivalent to
\[
\forall x^\tau\forall m^0\exists n^0\forall A^\Algebra(\PP(A)\leq_\RR \lambda -2^{-m}\to \exists \varpi\in A^c \, \varphi(x,n,\varpi)).
\]
The first result then immediately follows from Theorem \ref{thm:metaConstructive}. For the second result, we instead first require similar considerations on epsilon-terms as in the proof of Theorem \ref{thrm:pulloutpi2Uniform}, after which Theorem \ref{thm:metaConstructiveCC} then applies.
\end{proof}

\begin{remark}
As in Remark \ref{rem:monMeta}, suppose in the context of Theorem \ref{thrm:pulloutpi2-const} additionally that $\varphi$ is monotone in $n$. Then, by Proposition \ref{pro:implicationOuter}, for that same functional $\Phi$ we in fact have
\[
\mathcal{S}^{\omega, \Omega, \Algebra} \models \forall m^0\forall {x^*}^{\widehat{\tau}}\forall x\lesssim_\tau x^* \left( \PP[\varphi(x,\Phi(m,x^*))]\geq \lambda-2^{-m} \right),
\]
so that also here the monotonicity allows us to derive an even stronger result. As before, in the context of Theorem \ref{thrm:pulloutpi2-const}, this result could then also have been established by first pulling the existential quantifier out of the outer content using uniform boundedness, now as in Proposition \ref{pro:furtherPrenexConst}.
\end{remark}

\subsection{Probabilistic theorems of the form $\PP[\forall x^\tau\exists n^0\varphi(x,n)]\geq\lambda$}

We now consider theorems of the form \ref{strongProblem}. This formulation of a probabilistic $\forall\exists$-theorem actually proves to be quite subtle in various ways. Indeed, both the classical and semi-constructive metatheorems will prove to be considerably more difficult to establish than the previous results, with greatest difficulty in the classical case. In particular, we rely rather extensively on the use of uniform boundedness as well as epsilon-terms in the style of \cite{GuenzelKohlenbach2016} and as the result seems to be rather sensitive, we here in particular are very explicit on the use of epsilon-terms, compared to some of the previous results, where they were only mentioned in passing or briefly sketched. With these difficulties also come various slightly more severe restrictions on the reach of the metatheorem. However, while discussed and exemplified in Section \ref{sec:examples} later on, we want to already remark here that, currently, apart from one notable recent case study \cite{NeriPischkePowell2026}, no other use of formulations of the form \ref{strongProblem} has ever been observed in proof mining practice, with essentially all theorems naturally confining to the forms \ref{problem} and \ref{uniformProblem}. In that way, the resulting complications so far had a rather minimal impact on proof mining practice.

We now begin with the classical case. To simplify the presentation, we here restrict ourselves to $\tau\in\{0,1\}$ and we assume that $\varphi$ is quantifier-free, although both of these requirements can be relaxed in some ways.

\begin{theorem}\label{thrm:pulloutpi2Strong}
In the context of the assumptions of Theorem \ref{thm:metaClassical}, assume that
\[
\mathcal{F}^{\omega}[\PP]+\UBOmega{\omega}+\UBF{\omega}\vdash \PP[\forall x^\tau\exists n^0 \varphi_0(x,n)]\geq\lambda,
\] 
where $\tau$ is of degree $1$ and $\varphi_0(x,n,\varpi)$ is quantifier-free.

Then one can extract a total bar-recursive functional $\Phi:\mathbb{N}\times\mathcal{M}_{\tau(0(\tau))}\rightarrow \mathcal{M}_{0(\tau)}$ such that
\[
\mathcal{M}^{\omega,\Omega,\Algebra}\models\forall m\forall x^* \forall \tilde{x} \lesssim x^* \left(\PP[\exists N\lesssim \Phi(m,x^*)\forall x\leq \tilde{x} (N)\, \varphi_0(x,N(x))]\geq\lambda -2^{-m}\right)
\]
holds for $\mathcal{M}^{\omega, \Omega, \Algebra}$ defined via any non-empty set $\Omega$, algebra $S\subseteq 2^\Omega$, and probability content $\PP$ on $\Algebra$ (and with suitable interpretations of the additional constants).

If $\tau=0$, then the conclusion moreover holds in a model based on $\mathcal{S}^{\omega,\Omega,\Algebra}$, with $\Phi$ then being a partial function $\mathbb{N}\times\mathcal{S}_{2}\rightharpoonup \mathcal{M}_{1}$.

We have the same remarks (2) -- (4) as in Theorem \ref{thm:metaClassical}.
\end{theorem}
\begin{proof}
We focus on the more complex case $\tau=1$. Expanding the abbreviation, we have that $\mathcal{F}^{\omega}[\PP]+\UBOmega{\omega}+\UBF{\omega}$ proves
\[
\forall m^0\forall A^\Algebra(\PP(A)\leq_\RR \lambda -2^{-m}\to \exists \varpi\in A^c\forall x^\tau\exists n^0\varphi_0(x,n,\varpi)).
\]
In particular, using $\QFAC$, the above is equivalent to
\[
\forall m^0\forall A^\Algebra(\PP(A)\leq_\RR \lambda -2^{-m}\to \exists N^{0(\tau)}\exists \varpi\in A^c\forall x^\tau\varphi_0(x,Nx,\varpi)).
\]
Using $\UBOmega{\tau}$ and classical logic, this is further equivalent to
\[
\forall m^0\forall A^\Algebra(\PP(A)\leq_\RR \lambda -2^{-m}\to \exists N^{0(\tau)}\forall {x^*}^\tau\exists \varpi\in A^c\forall x\leq_\tau x^*\,\varphi_0(x,Nx,\varpi)).
\]
This yields
\[
\forall m^0\forall A^\Algebra\exists N^{0(\tau)}\forall {x^*}^\tau(\PP(A)\leq_\RR \lambda -2^{-m}\to \exists \varpi\in A^c\forall x^\tau\varphi_0(\min_\tau(x,x^*),N\min_\tau(x,x^*),\varpi)).\tag{\%}\label{firstClass}
\]
We now add an epsilon-term $\phi^{1(0(1))(1)}$ in the style of \cite{GuenzelKohlenbach2016} to the language with the intended semantics that $\phi(y^1,z^{0(1)}):=x\leq_1 y$ for $x$ such that $z(x)=_00$, if existent, and $\phi(y^1,z^{0(1)}):=\lambda n^0.0$ otherwise. This constant is governed by the following universal axiom:
\[
\forall x^1,y^1,z^{0(1)}\left(z(\min_1(x,y))=_00\to z(\min_1(\phi(y,z),y))=_00\right).\tag{$\phi$}\label{epsilonAx}
\]
The formal semantics, which later assures majorizability of that constant, is now as follows: over $\mathcal{S}^{\omega,\Omega,\Algebra}$, it is interpreted as 
\[
\phi(y^1,z^{0(1)}):=_1\begin{cases}\min_1(x,y)&\text{for }x^1\text{ with }z(\min_1(x,y))=_00\text{ if existent},\\\lambda n^0.0&\text{otherwise}.\end{cases}
\]
As $\min_1(\min_1(x,y),y)=\min_1(x,y)$, this satisfies the axiom. Let now $t_{\neg\varphi}$ be the term such that
\[
t_{\neg\varphi}(x,N,\varpi)=_00\leftrightarrow \neg\varphi_0(x,Nx,\varpi).
\]
Applying \eqref{epsilonAx} to $\lambda x^1.t_{\neg\varphi}(x,N,\varpi)$ and $x^*$ yields
\begin{gather*}
\varphi_0(\min_1(\phi(x^*,\lambda x^1.t_{\neg\varphi}(x,N,\varpi)),x^*),N\min_1(\phi(x^*,\lambda x^1.t_{\neg\varphi}(x,N,\varpi)),x^*),\varpi)\\
\rightarrow \varphi_0(\min_1(x,x^*),N\min_1(x,x^*),\varpi)
\end{gather*}
so that 
\begin{gather*}
\varphi_0(\min_1(\phi(x^*,\lambda x^1.t_{\neg\theta}(x,N,\varpi)),x^*),N\min_1(\phi(x^*,\lambda x^1.t_{\neg\theta}(x,N,\varpi)),x^*),\varpi)\\
\leftrightarrow \forall x^1\varphi_0(\min_1(x,x^*),N\min_1(x,x^*),\varpi).
\end{gather*}
Using another ``epsilon''-term, we hence get a quantifier-free formula $\theta_0(x^*,N)$ such that
\[
\theta_0(x^*,N)\leftrightarrow \exists \varpi\in A^c\forall x^\tau\varphi_0(\min_\tau(x,x^*),N\min_\tau(x,x^*),\varpi).
\]
The above formula \eqref{firstClass} is hence equivalent
\[
\forall m^0\forall A^\Algebra\exists N^{0(\tau)}\forall {x^*}^\tau(\PP(A)\leq_\RR \lambda -2^{-m}\to \theta_0(x^*,N))
\]
and using $\QFAC$ to 
\[
\forall m\forall A\forall {\tilde{x}}^*\exists N(\PP(A)\leq_\RR \lambda -2^{-m}\to \theta_0({\tilde{x}}^*(N),N)).
\]
The (negative) functional interpretation (see Lemma 8.3 in \cite{NeriPischke2025}) combined with majorization in the model $\mathcal{M}^{\omega,\Omega,\Algebra}$ (see Lemma 8.8 in \cite{NeriPischke2025} and recall the proof of Theorem \ref{thm:metaClassical}) now yields a functional $\Phi$ such that
\[
\mathcal{M}^{\omega,\Omega,\Algebra}\models\forall m\forall A\forall x^* \forall \tilde{x} \lesssim x^* \exists N\lesssim \Phi(m, x^*)(\PP(A)\leq_\RR \lambda -2^{-m}\to \theta_0(\tilde{x}(N),N))
\]
or, expanded, 
\begin{gather*}
\forall m\forall A\forall x^* \forall \tilde{x} \lesssim x^* \exists N\lesssim \Phi(m,x^*)(\PP(A)\leq_\RR \lambda -2^{-m}\\
\to \exists \varpi\in A^c\forall x\,\varphi_0(\min(x, \tilde{x}(N)), N(\min(x, \tilde{x}(N))),\varpi))
\end{gather*}
which is equivalent to
\[
\forall m\forall x^* \forall \tilde{x} \lesssim x^* \PP[\exists N\lesssim \Phi(m, x^*)\forall x\, \varphi_0(\min(x, \tilde{x}(N)),N(\min(x, \tilde{x}(N))))]\geq\lambda -2^{-m}.
\]
As we are working over the model $\mathcal{M}^{\omega,\Omega,\Algebra}$, we now get (using extensionality) that the above is equivalent to 
\[
\mathcal{M}^{\omega,\Omega,\Algebra}\models\forall m\forall x^* \forall \tilde{x} \lesssim x^* \PP[\exists N\lesssim \Phi(m, x^*)\forall x\leq \tilde{x}(N)\, \varphi_0(x,N(x))]\geq\lambda -2^{-m},
\]
as claimed. That this reduces to truth in a model based on $\mathcal{S}^{\omega,\Omega,\Algebra}$ if $\tau=0$ follows immediately as the types are then low-enough (see in particular the discussion in the proof of Theorem 17.52 in \cite{Kohlenbach2008}). 
\end{proof}

Note that already for $\tau=0$, the above modulus $\Phi$ is a functional of type $3$, which is quite high. Indeed, such a modulus has never been observed in practice, as will also be discussed in more details in the next section.

As before, we now also consider the special case of semi-constructively proven theorems of the form $\PP[\forall x^\tau\exists n^0\varphi(x,n)]\geq\lambda$. Again, the corresponding metatheorem is harder to establish than that given in Theorem \ref{thrm:pulloutpi2-const}, though only slightly more difficult. In particular, we here rely on a model-theoretic result due to Kohlenbach \cite{Kohlenbach2011}.

\begin{theorem}\label{thrm:pulloutpi2-constStrong}
In the context of the assumptions of Theorem \ref{thm:metaConstructive}, assume that
\[
\mathcal{F}_i^{\omega}[\PP]+\FullUBOmega{\omega}+ \FullUBF{\omega} \vdash \PP[\forall x^\tau\exists n^0 \varphi(x,n)]\geq\lambda,
\] 
where $\tau$ is admissible and $\varphi(x,n,\varpi)$ is any formula such that all positively occurring universal quantifiers and all negatively occurring existential quantifiers have small types and all other types are admissible.

Then one can extract a primitive-recursive functional $\Phi:\mathbb{N}\times\mathcal{S}_{\widehat{\tau}}\to\mathbb{N}$ such that
\[
\mathcal{S}^{\omega, \Omega, \Algebra}\models \forall m^0\left( \PP[\forall x^* \forall x\lesssim_\tau x^*\exists n\leq_0\Phi(m,x^*)\, \varphi(x,n)]\geq \lambda-2^{-m}\right)
\]
holds for $\mathcal{S}^{\omega, \Omega, \Algebra}$ defined via any non-empty set $\Omega$, algebra $S\subseteq 2^\Omega$, and probability content $\PP$ on $\Algebra$ (and with suitable interpretations of the additional constants).

We have the same remarks (1) -- (3) as in Theorem \ref{thm:metaConstructive}.
\end{theorem}
\begin{proof}
Expanding the abbreviation, we have that $\mathcal{F}^{\omega}_i[\PP]+\FullUBOmega{\omega}+\FullUBF{\omega}$ proves
\[
\forall m^0\forall A^\Algebra(\PP(A)\leq_\RR \lambda -2^{-m}\to \exists \varpi\in A^c\forall x^\tau\exists n^0\varphi(x,n,\varpi)).
\]
Using $\AC$, we get
\[
\forall m^0\forall A^\Algebra(\PP(A)\leq_\RR \lambda -2^{-m}\to \exists N^{0(\tau)}\exists \varpi\in A^c\forall x^\tau\varphi(x,Nx,\varpi)).
\]
The principle $\IPU$ yields
\[
\forall m^0\forall A^\Algebra\exists N^{0(\tau)}(\PP(A)\leq_\RR \lambda -2^{-m}\to \exists \varpi\in A^c\forall x^\tau\varphi(x,Nx,\varpi)).
\]
Using the (negative) functional interpretation (see again Lemma 8.3 in \cite{NeriPischke2025}) combined with majorization in the model $\mathcal{M}^{\omega,\Omega,\Algebra}$ (see again Lemma 8.8 in \cite{NeriPischke2025} and Theorem \ref{thm:metaClassical}), we get a majorizable functional $\Phi:\mathbb{N}\to\mathcal{M}_{0(\widehat{\tau})}$ such that:
\[
\mathcal{M}^{\omega,\Omega, \Algebra} \models \forall m^0 \left(\PP[\exists N\lesssim_{0(\tau)}\Phi(m)\forall x^\tau\varphi(x,Nx,\varpi)]\geq\lambda-2^{-m}\right).
\]
By Theorem 2.5 from \cite{Kohlenbach2011}, the above is equivalent to
\[
\mathcal{M}^{\omega,\Omega, \Algebra} \models \forall m^0 \left(\PP[\forall {x^*}^{\widehat{\tau}}\forall x\lesssim_\tau x^*\exists n\leq_0\Phi(m,x^*)\, \varphi(x,n,\varpi)]\geq\lambda-2^{-m}\right).
\]
As $\tau$ is admissible, we get that the resulting claim is also true in $\mathcal{S}^{\omega,\Omega, \Algebra}$ (recall again the discussion in the proof of Theorem 17.52 in \cite{Kohlenbach2008}). 
\end{proof}

\subsection{Relationships between \ref{strongProblem}, \ref{problem} and \ref{uniformProblem}} 

We now want to comment on the relationship between the statements \ref{strongProblem} -- \ref{uniformProblem}, and their respective computational interpretations, under various natural assumptions on the types and associated problem matrices. First, let us recall  the various types of computational content that these different statements generate:

The computational content of statements of the form $\PP[\forall x^\tau\exists n^0\,\varphi(x,n)]\geq\lambda$, that is of the form \ref{strongProblem}, was dependent in form already on whether this statement is considered in a classical or semi-constructive context, even when the internal matrix $\varphi$ is quantifier-free. Concretely, in a semi-constructive context this content takes the form of a functional $\Phi$ such that 
\[
\forall m^0\left(\PP[\forall x^* \forall x\lesssim_\tau x^*\exists n\leq_0\Phi(m,x^*)\, \varphi(x,n)]\geq \lambda-2^{-m}\right),\tag*{($+$)$^i_1$}\label{IntStrongConst}
\]
while, in a classical context, this content takes the form of a functional $\Phi$ such that 
\[
\forall m\forall x^* \forall \tilde{x} \lesssim x^* \left(\PP[\exists N\lesssim \Phi(m,x^*)\forall x\leq \tilde{x}(N)\, \varphi_0(x,N(x))]\geq\lambda -2^{-m}\right).\tag*{($+$)$^c_1$}\label{IntStrongClass}
\]
The associated statement of the form \ref{problem} where the universal quantifier is prenexed, that is $\forall x^\tau\left(\PP[\exists n^0\,\varphi(x,n)]\geq\lambda\right)$, was not so sensitive in form, resulting generally (that is classically or semi-constructively) in functionals $\Phi$ such that
\[
\forall m^0\forall x^*\forall x\lesssim_\tau x^*
\left(\PP[\exists n\leq_0\Phi(m,x^*)\, \varphi(x,n)]\geq \lambda-2^{-m}\right).\tag*{($+$)$_2$}\label{IntWeak}
\]
Lastly, if further prenexed appropriately by also drawing out the existential quantifier, at the expense of a probabilistic error, that is moving to the statement $\forall m^0\forall x^\tau\exists n^0\left(\PP[\varphi(x,n)]\geq\lambda-2^{-m}\right)$ of the form \ref{uniformProblem}, we obtain a functional $\Phi$ such that
\[
\forall m^0\forall x^* \forall x\lesssim_\tau x^*\exists n\leq_0\Phi(m,x^*)\left(\PP[\varphi(x,n)]\geq \lambda-2^{-m}\right).\tag*{($+$)$_3$}\label{IntUniform}
\]

The relation between the different problem formulations and their resulting computational interpretations proves to be quite subtle, so that instead of providing a comprehensive answer, we are restricted to a collection of specialized observations.

First, note that by Proposition \ref{pro:initialPrenex} both \ref{strongProblem} and \ref{uniformProblem} imply \ref{problem}, without any additional assumptions on the matrix $\varphi$ or the type $\tau$. In particular, the resulting interpretations \ref{IntStrongClass} and \ref{IntUniform} primitive recursively imply \ref{IntWeak} (as is perhaps apparent from their formulation already). Further, intuitionistically \ref{strongProblem} implies its own negative translation, so that \ref{IntStrongConst} primitive recursively implies \ref{IntStrongClass}. Putting these relations into a picture, we get the following overview:
\begin{center}
\begin{tikzpicture}
\node at (0,0) (2) {$(*)_2$};
\node [below left of=2,xshift=-.5cm] (1) {$(*)_1$};
\node [above left of=2,xshift=-.5cm] (3) {$(*)_3$};
\draw[->] (1) -- (2);  
\draw[->] (3) -- (2);
\node at (8,0) (n2) {$(+)_2$};
\node [below left of=n2,xshift=-.6cm] (n1c) {$(+)^c_1$};
\node [left of=n1c,xshift=-.4cm] (n1i) {$(+)^i_1$};
\node [above left of=n2,xshift=-.6cm] (n3) {$(+)_3$};
\draw[->] (n1i) -- (n1c);  
\draw[->] (n1c) -- (n2);  
\draw[->] (n3) -- (n2); 
\node [xshift=-1cm] at ($(2)!0.5!(n2)$) (viz) {viz.};
\end{tikzpicture}
\end{center}
In other words, pulling the first universal quantifier out of the (outer) measure makes the problem considerably weaker, while pulling the second existential quantifier out too, with a probabilistic error, strengthens the problem again. In particular, we currently know of \emph{no} relation between \ref{strongProblem} and \ref{uniformProblem} (viz.\ \ref{IntStrongClass}, \ref{IntStrongConst} and \ref{IntUniform}) and it seems that no further relations can be directly inferred without further assumptions.

Now, if $\varphi(x,n,\varpi)$ is both monotone in $n$ and quantifier-free, then Propositions \ref{pro:furtherPrenex} and \ref{pro:furtherPrenexConst} guarantee  that \ref{problem} implies back \ref{uniformProblem}, a formal version of their mathematical equivalence for measures (recall Remark \ref{rem:prenex}), established, however, only under the set-theoretically false principle of uniform boundedness, which nevertheless has no impact on bound extractions so that in that context, both statements can be used interchangeably.

An immediate question is whether this equivalence somehow extends to more complex matrices. Here, we can give an answer at least for the nearest case of (measurable and suitably monotone) $\Pi^0_1$-matrices, by generalizing a mathematically motivated construction of Avigad, Dean and Rute \cite{AvigadDeanRute2012} to a more general context. In \cite{AvigadDeanRute2012}, the authors show how specific modes of finitary probabilistic convergence relate to each other, corresponding to specific instantiations of \ref{problem} and \ref{uniformProblem} if seen in the context of the present paper, and by that give a quantitative variant of Egorov's theorem \cite{Egoroff1911} on the equivalence between pointwise and uniform almost-sure convergence of random variables, as will also be discussed later in Section \ref{sec:examples}. This result was formally recognized in the context of the present systems in \cite{NeriPischke2025}, in which it was in particular observed that the corresponding result already holds over contents. While all of this was tailored to specific modes of convergence, we can now actually recognize this formal representation of the results from \cite{AvigadDeanRute2012} given in \cite{NeriPischke2025} as the blueprint of a general construction for the equivalence of \ref{problem} and \ref{uniformProblem} for $\Pi^0_1$-matrices, beyond convergence.

As hinted on before, the key assumption for this will be the measurability of the matrix. Concretely, we will in the following assume that $\varphi(x^0,n^0,\varpi^\Omega):=\forall k^0\varphi_0(x,n,k,\varpi)$ and that $\varphi_0(x,n,k,\varpi)$ is quantifier-free and measurable in the sense that there is a term $P_{x,n,k}$ of type $\Algebra(0)(0)(0)$ such that
\[
\forall x^0\forall n^0\forall k^0\forall\varpi^\Omega\left(\varpi\in P_{x,n,k}\leftrightarrow \varphi_0(x,n,k,\varpi)\right)\tag{$M$}\label{measAx}
\]
is provable in the system.\footnote{As such, this can immediately be ensured formally by adding a corresponding (majorizable) constant $P$ of the respective type together with the above principle as an additional (universal) axiom.} Then, instead of working with the formula $\varphi_0$, the construction given in \cite{AvigadDeanRute2012} and further abstracted in \cite{NeriPischke2025} relies on working with the sets $P_{x,n,m}$ and their probabilities more abstractly. The key result for that is the following combinatorial result established in \cite{AvigadDeanRute2012} and then formalized over $\mathcal{F}^\omega[\PP]$ in \cite{NeriPischke2025}:

\begin{lemma}[Theorem 9.7 in \cite{NeriPischke2025}]\label{lem:ADRCombinatorial}
The system $\mathcal{F}^\omega[\PP]$ proves:
\[
\forall A^{\Algebra(0)}, N^{0(1)}, u^0,v^0>_0 u\exists i^0\left(\forall g^1\left(\PP\left(\bigcap_{l=0}^{N(g)}\bigcup_{j=l}^{g(l)}A_j\right)\leq_\mathbb{R}2^{-v}\right) \to \PP(A_i)<_\mathbb{R}2^{-u}\right).
\]
\end{lemma}

This result is far from trivial and is in particular the source of the computational complexity of the resulting quantitative variant of our result, as it crucially relies on $\Pi^0_1\mbox{-}\mathsf{AC}$, which is contained in $\DC$ and hence in $\mathcal{F}^\omega[\PP]$.

The result we then get is the following, which closely follows the proof strategy of Theorem 9.9 in \cite{NeriPischke2025}.

\begin{proposition}\label{pro:AvigadEquiv}
Let $\varphi_0(x^0,n^0,k^0,\varpi^\Omega)$ be measurable and quantifier-free, and monotone in $n$ as well as anti-monotone in $k$. Over $\mathcal{F}^\omega[\PP]+\UBOmega{0}+\UBF{0}$, the following are equivalent:
\begin{enumerate}
\item $\forall x^0\left(\PP[\exists n^0\forall k^0\varphi_0(x,n,k)]\geq 1\right)$,
\item $\forall m^0\forall x^0\exists n^0\left(\PP[\forall k^0\varphi_0(x,n,k)]\geq 1-2^{-m}\right)$.
\end{enumerate}
\end{proposition}
\begin{proof}
By Proposition \ref{pro:initialPrenex}, (2) implies (1). To see the converse, consider (1), which by $\QFAC$ is equivalent to $\forall x^0\left(\PP[\forall g^1\exists n^0\,\varphi_0(x,n,g(n))]\geq 1\right)$ and so in particular, by Proposition \ref{pro:initialPrenex}, implies
\[
\forall x^0\forall g^1\left(\PP[\exists n^0\varphi_0(x,n,g(n))]\geq 1\right).
\]
By $\UBF{0}$, we get
\[
\forall m^0\forall x^0\forall g^1\exists n^0\left(\PP[\exists i\leq_0 n\,\varphi_0(x,i,g(i))]\geq 1-2^{-m}\right).
\]
As $\varphi_0$ is measurable by \eqref{measAx}, so is $\exists i\leq_0 n\,\varphi_0(x,i,g(i))$ and so we get
\[
\forall m^0\forall x^0\forall g^1\exists n^0\left(\PP\left(\bigcup_{i=0}^{n} P_{x,i,g(i)}\right)\geq_\mathbb{R} 1-2^{-m}\right)
\]
by Proposition \ref{pro:measurable}. Written in its contrapositive, we get we get
\[
\forall m^0\forall x^0\forall g^1\exists n^0\left(\PP\left(\bigcap_{i=0}^{n}P^c_{x,i,g(i)}\right)\leq_\mathbb{R} 2^{-m}\right).
\]
Hence, using $\Pi^0_1\mbox{-}\mathsf{AC}$, we get that there exists a function $N$ with
\[
\forall m^0\forall x^0\forall g^1\left(\PP\left(\bigcap_{i=0}^{N_{m,x}(g)} P^c_{x,i,g(i)}\right)\leq_\mathbb{R} 2^{-m}\right).
\tag{$\dagger$}\label{keyStepAvigad}
\]
Now consider (2). By Proposition \ref{pro:furtherPrenex} and using $\UBOmega{0}$, this is equivalent to
\[
\forall m^0\forall x^0\exists n^0\forall k^0\left(\PP[\varphi_0(x,n,k)]\geq 1-2^{-m}\right).
\]
For a contradiction, suppose that (2) fails, so that in particular, also using \eqref{measAx} and Proposition \ref{pro:measurable}, there are $m_0$ and $x_0$ such that $\forall n^0\exists k^0\left(\PP(P^c_{x_0,n,k})>_\mathbb{R} 2^{-m_0}\right)$. Using $\Pi^0_1\mbox{-}\mathsf{AC}$, we get a function $g_0$ such that
\[
\forall n^0\left(\PP(P^c_{x_0,n,g_0(n)})>_\RR 2^{-m_0}\right).
\]
Now, define $A_i:=P^c_{x_0,i,g_0(i)}$. Then, using the monotonicity of $\varphi_0$, and hence of $P^c$, we have
\[
\bigcup_{j=l}^{g(l)}A_j = \bigcup_{j=l}^{g(l)}P^c_{x_0,j,g_0(j)}\subseteq P^c_{x_0,l,\widetilde{g_0}(g(l))}
\]
for all $l$ and $g$, where $\widetilde{g_0}(i):=\max_{j\leq i}g_0(j)$. In particular, this implies
\[
\forall g^1\left( \bigcap_{l=0}^{N'_{m_0+1,x_0}(g)}\bigcup_{j=l}^{g(l)}A_j \subseteq \bigcap_{l=0}^{N'_{m_0+1,x_0}(g)}P^c_{x_0,l,\widetilde{g_0}(g(l))}\right),
\]
where $N'_{m_0+1,x_0}(g):=N_{m_0+1,x_0}(\lambda n.\widetilde{g_0}(g(n)))$. By \eqref{keyStepAvigad} and the monotonicity of $\PP$, we get
\[
\forall g^1\left( \PP\left(\bigcap_{l=0}^{N'_{m_0+1,x_0}(g)}\bigcup_{j=l}^{g(l)}A_j\right) \leq_\RR 2^{-(m_0+1)}\right),
\]
so that Lemma \ref{lem:ADRCombinatorial} implies that there exists an $i$ such that $\PP(P^c_{x_0,i,g_0(i)})=_\RR\PP(A_i) <_\RR 2^{-m_0}$, which is a contradiction.
\end{proof}

Note by inspection of the above proof that we actually have established that the implication
\[
\begin{cases}\forall m^0\forall x^0\forall g^1\exists n^0\left(\PP[\exists i\leq_0 n\,\varphi_0(x,i,g(i))]\geq 1-2^{-m}\right)\\
\qquad\to \forall m^0\forall x^0\exists n^0\forall k^0\left(\PP[\varphi_0(x,n,k)]\geq 1-2^{-m}\right)\end{cases}
\]
already holds over $\mathcal{F}^\omega[\PP]$, that is without any uniform boundedness principles. This is essentially what was established in \cite{NeriPischke2025} for the special case of a convergence statement and what was presented in a quantitative form in the paper \cite{AvigadDeanRute2012}. In particular, even this result crucially relies on $\Pi^0_1\mbox{-}\mathsf{AC}$, so that we can here only guarantee that there is a bar-recursive witness for the above implication (and in that way, essentially, mapping solutions of \ref{IntWeak} to solutions of \ref{IntUniform}) for measurable and suitably monotone $\Pi^0_1$-matrices. Moreover, such a witness, for this generalized result, can be explicitly constructed by more or less directly following the arguments of \cite{AvigadDeanRute2012}, although we do not explore this here any further. We know of no formal relationship beyond $\Pi^0_1$-matrices.

Moving to the other branch of the picture, we know a bit less. Recall that for quantifier-free and anti-monotone formulas $\psi(x^0,\varpi^\Omega)$ (not necessarily measurable), we provably have that $\PP[\forall x^0\psi(x)]\geq \lambda$ is equivalent to $\forall x^0\PP[\psi(x)]\geq\lambda$, which (in the classical context) even extends to $x$ of higher type (recall Propositions \ref{pro:furtherPrenex} and \ref{pro:evenFurtherPrenexConst}). 

The question on the formal relationship between \ref{problem} and \ref{strongProblem} thereby reduces to the question if, and how, this formal equivalence extends to formulas which are not decidable anymore but have the next highest complexity, that is $\Sigma^0_1$-matrices. The key result in that direction is the following, relating the respective interpretations \ref{IntWeak} and \ref{IntStrongConst} (and by that also \ref{IntStrongClass}) with each other.

\begin{lemma}\label{lem:equivRemark}
Let $\varphi_0(x^0,n^0,\varpi^\Omega)$ be measurable and quantifier-free. Over $\mathcal{F}^\omega[\PP] +(\sigma\text{-}\textsf{additivity})$, we have the following implication: For any functional $\Phi(m,x)$ with
\[
\forall m^0\forall x^0\left(\PP[\exists n\leq_0\Phi(m,x)\, \varphi_0(x,n)] \geq 1-2^{-m}\right),
\]
the functional $\Phi'(m,x):=\Phi(m+x+1,x)$ satisfies
\[
\forall m^0\left(\PP[\forall {x}^0\exists n\leq_0\Phi'(m,x)\, \varphi_0(x,n)] \geq 1-2^{-m}\right).
\]
\end{lemma}
\begin{proof}
Fix $m^0$. Note that $\forall n\leq_0\Phi(m+x+1,x)\, \neg\varphi_0(x,n)$ is provably measurable in $\mathcal{F}^\omega[\PP]$, and (as an abuse of notation) we in the following simply write the above expression for that set. Using $(\sigma\text{-}\textsf{additivity})$, for any $k^0$ there exists an $l^0$ such that (again, slightly abusing notation):
\begin{align*}
\PP\left[\exists {x}^0\forall n\leq_0\Phi(m+x+1,x)\, \neg\varphi_0(x,n)\right]&\leq \sum_{x=0}^l\PP\left(\forall n\leq_0\Phi(m+x+1,x)\, \neg\varphi_0(x,n)\right)+2^{-k}\\
&\leq_\RR \sum_{x=0}^l 2^{-m-(x+1)}+2^{-k}\leq_\RR 2^{-m}+2^{-k}.
\end{align*}
Since $k^0$ was arbitrary, we have
\[
\PP\left[\exists {x}^0\forall n\leq_0\Phi(m+x+1,x)\, \neg\varphi_0(x,n)\right]\leq 2^{-m}
\]
which yields the result using duality.
\end{proof}

Note that through Proposition \ref{Prop:elim:sigma}, the above result in particular holds over $\mathcal{F}^\omega[\PP] +\UBOmega{0}$ and so can be used in the context of bound extraction theorems without affecting extracted bounds. Regardless of that, we want to further emphasize that Lemma \ref{lem:equivRemark} is actually mathematically true for probability spaces $(\Omega,\Algebra,\PP)$.

Further, as statements of the form \ref{problem} can be formally related to their quantitative variants \ref{IntWeak} using uniform boundedness, we obtain the following equivalence:

\begin{proposition}
Let $\varphi_0(x^0,n^0,\varpi^\Omega)$ be measurable and quantifier-free. Over $\mathcal{F}^\omega[\PP]+\UBOmega{0}+\UBF{0}$, the following are equivalent:
\begin{enumerate}
\item $\PP[\forall x^0\exists n^0\varphi_0(x,n)]\geq 1$,
\item $\forall x^0\left(\PP[\exists n^0\varphi_0(x,n)]\geq 1\right)$.
\end{enumerate}
\end{proposition}
\begin{proof}
By Proposition \ref{pro:initialPrenex}, we only have to establish the direction from (2) to (1). We first show that over $\mathcal{F}^\omega[\PP]+\UBF{0}$, (2) is equivalent to
\[
\exists\Phi^{0(0)(0)}\forall m^0\forall {x}^0\left(\PP[\exists n\leq_0\Phi(m,x)\,\varphi_0(x,n)]\geq 1-2^{-m}\right).
\]
To see this, note that $\forall x^0\left(\PP[\exists n^0\varphi_0(x,n)]\geq 1\right)$ is equivalent to
\[
\forall m^0\forall x^0\forall A^S\exists n^0\left(\PP(A)\leq_\mathbb{R} 1-2^{-m}\to\exists\varpi\in A^c\,\varphi_0(x,n,\varpi)\right).
\]
Using $\UBF{0}$, we get that 
\[
\exists\Phi^{0(0)(0)}\forall m^0\forall x^0\forall A^S\exists n\leq_0\Phi(m,x)\left(\PP(A)\leq_\mathbb{R} 1-2^{-m}\to\exists\varpi\in A^c\,\varphi_0(x,n,\varpi)\right)
\]
which amounts to the above. Hence, over $\mathcal{F}^\omega[\PP]+\UBOmega{0}+\UBF{0}$, we can now proceed as follows: Using Lemma \ref{lem:equivRemark} (and Proposition \ref{Prop:elim:sigma}), (2) now implies
\[
\exists{\Phi'}^{0(0)(0)}\forall m^0\left(\PP[\forall {x}^0\exists n\leq_0\Phi'(m,x)\, \varphi_0(x,n)] \geq 1-2^{-m}\right).
\]
Dropping the quantitative information, we obtain $\forall m^0\left(\PP[\forall {x}^0\exists n^0\, \varphi_0(x,n)] \geq 1-2^{-m}\right)$, which in turn yields (1).
\end{proof}

Based on the use of the above mathematical result from Lemma \ref{lem:equivRemark}, we are here restricted to quantifiers of type $0$. It would be very interesting to see whether the above can be extended to e.g.\ type $1$ using continuity arguments, but we leave the investigation of this for future work.

In general, we see already that the formula structure and its interaction with $\PP$ has a subtle impact on the computational strength of a statement, and that realizing the previously mentioned mathematical equivalences over measures can be far from computationally trivial. However, perhaps surprisingly, in the context of proof mining practice, these issues rarely surface, if at all. Indeed, as will be discussed in detail in the following Section \ref{sec:examples}, essentially all statements ever considered in the context of proof mining case studies are naturally formalized via statements of the form \ref{problem} or \ref{uniformProblem}, by virtue of most proofs of theorems of the form \eqref{originalProblem} actually being of a sufficiently ``local'' nature that keeps universal parameters fixed throughout the argument.

As such, essentially all moduli ever considered in the context of proof mining practice have been of the form \ref{IntWeak} or \ref{IntUniform}. That means that neither are these stronger functionals extracted in case studies nor are they commonly required to resolve an implication where a statement like \ref{strongProblem} features in the premise. Indeed, only one recent case study \cite{NeriPischkePowell2026} involved such a modulus, as will be discussed in detail later.

In any case, if such formulations of probabilistic theorems will be required in the future, the present section provided the necessary logical tools and remarks on their relation to deal with them in practice.

\section{Examples from mathematical practice}\label{sec:examples}

We now illustrate the rather wide applicability of our probabilistic logical metatheorems by giving a range of examples from mathematical practice where they yield exactly the type of computational content expected and observed in applications of proof mining to probability theory. We focus on some key types of theorems, covering central examples. Nevertheless, also essentially all other (perhaps more minor) quantitative notions considered in the context of proof mining case studies in probability can similarly be systematically justified from a proof-theoretic perspective via the present paper, including in particular notions like moduli of almost sure finiteness or tightness as considered in \cite{NeriPowell2025}, moduli of absolute continuity as considered in \cite{PischkePowell2024} and the notion and related theory of finitary martingales developed in \cite{NeriPischkePowell2026}. It should be noted that for all the quantitative notions discussed in this section, the present paper in particular illustrates for the first time that these notions arise naturally through a systematic proof-theoretic approach towards assigning computational content to (classical or semi-constructive) almost sure statements, and in that way provides the first formal proof-theoretic justification for their extraction in the various case studies mentioned throughout. In this section, if not stated otherwise, all mathematical results are formulated over a respective probability content space $(\Omega,\Algebra,\PP)$.

\subsection{$A_n$ infinitely often almost surely} 

Let $(A_n)_{n \in \NN}$ be a sequence of events. The fact that $\varpi$ belongs to the event $\limsup_{n \to \infty} {A_n} := \bigcap_{k\in\mathbb{N}} \bigcup_{n \geq k} A_n$ can be expressed via the formula
\[
\forall k^0 \exists n^0 \exists i^0 \in [k; n] (\varpi \in A_i).
\]
We intentionally introduced the bounded quantification $\exists i^0 \in [k; n]$ here in order to make the formula $\exists i^0 \in [k; n] (\varpi \in A_i)$ monotone in $n$ and anti-monotone in $k$. The statement that \emph{$A_n$ occurs infinitely often almost surely}, that is $\PP(\limsup_{n \to \infty} {A_n})=1$, can now be represented formally in different ways in our setting. 

The representation with the weakest computational interpretation is given by the statement
\[
\forall k^0 \left( \PP[\exists n^0 \exists i^0 \in [k; n] \, A_i] \geq 1 \right),
\]
that is a statement of the form \ref{problem}. By Theorems \ref{thrm:pulloutpi2} and \ref{thrm:pulloutpi2-const}, from a (classical or semi-constructive) proof of this, we are be able to extract a modulus $\Phi$ such that
\[
\forall m\in\mathbb{N} \forall k\in\mathbb{N} \left(\PP\left(\bigcup_{i \in [k;\Phi(m, k)]} A_i\right) \geq 1 - 2^{-m}\right).
\]
We call such $\Phi$ a \emph{pointwise modulus of $A_n$ infinitely often almost surely}. Such a modulus has been used in \cite[Theorem 2.2]{ArthanOliva2021} in connection with the quantitative interpretation of the second Borel-Cantelli lemma or in the context of stochastic optimization \cite{NeriPischkePowell2025a,NeriPischkePowell2026,NeriPowell2024}.  Moreover, note that since $\exists i^0 \in [k; n] (\varpi \in A_i)$ is anti-monotone in $k$, this just coincides with the quantitative interpretation of the statement 
\[
\forall m^0\forall k^0 \exists n^0 \left( \PP[\exists i^0 \in [k; n] \, A_i] \geq 1-2^{-m} \right),
\]
that is a corresponding statement of the form \ref{uniformProblem}.

Now, if we instead were to consider a representation via 
\[
\PP[\forall k^0 \exists n^0 \exists i^0 \in [k; n] \, A_i] \geq 1,
\]
that is a statement of the form \ref{strongProblem}, then we get quite different computational interpretations based on whether we are in a classical or semi-constructive context, as illustrated in the previous Section \ref{sec:meta} (recall Theorems \ref{thrm:pulloutpi2Strong} and \ref{thrm:pulloutpi2-constStrong}). In particular, in a semi-constructive context, as illustrated in Theorem \ref{thrm:pulloutpi2-constStrong}, the resulting computational interpretation is given via a functional $\Phi$ satisfying
\[
\forall m\in\mathbb{N} \left(
\PP\left[\bigcap_{k\in\mathbb{N}} \bigcup_{i \in [k; \Phi(m, k)]} \, A_i\right] \geq 1 - 2^{-m} \right).
\]
Compared to the above modulus, which is of a more pointwise nature, we will call such a modulus $\Phi$ a \emph{uniform modulus of $A_n$ infinitely often almost surely}, which by the aforementioned theorem can be extracted from large classes of semi-constructive proofs of the above property. Such a modulus seems to be needed only rarely. Indeed, there is so far only one case study that actually relied on a modulus with this particularly strong property, given recently in \cite{NeriPischkePowell2026} (for a specific choice of sets $A_i$ arising in the context of stochastic optimization, see Definition 4.3 therein, that we do not discuss here any further). Indeed, therein we find a critical mismatch between two parts of a proof usually bridged by corresponding limit theorems for the measure, with the first part of the proof naturally establishing a theorem of the form $\forall k^0 \left( \PP[\exists n^0 \exists i^0 \in [k; n] \, A_i] \geq 1 \right)$ and the latter part of the proof crucially relying on formulating this theorem in the form of $\PP[\forall k^0 \exists n^0 \exists i^0 \in [k; n] \, A_i] \geq 1$. Even though there generally might be no immediate quantitative relation between statements of the form \ref{problem} and \ref{strongProblem} as discussed in the previous Section \ref{sec:meta}, as the types here were low enough, this gap in the particular case study \cite{NeriPischkePowell2026} could be bridged by Lemma \ref{lem:equivRemark}, by which, over a measure space, a solution to the latter is definable from a solution to the former.

We know of no other case where such an argument was ever needed. Ultimately, this might also be the reason why the classical interpretation of the above as given in Theorem \ref{thrm:pulloutpi2Strong}, which is actually represented via a functional of type 3, might never have been observed in mathematical practice before, as in cases such as the above it is simply subsumed by this stronger property over measure spaces.

\subsection{$A_n$ infinitely often almost never} 

Dually to the above, we can also consider the statement that \emph{$A_n$ occurs infinitely often almost never}, that is $\PP(\limsup_{n\to\infty}A_n)=0$. By following the above approach, we can represent this statement via duality through
\[
\PP[\forall k^0 \exists n^0 \exists i^0 \in [k; n] \, A_i] \leq 0 \;\; \equiv \;\; \PP[\exists k^0 \forall n^0 \forall i^0\in [k; n] \, A^c_i]\geq 1.
\]
As the inner matrix of the outer measure above is $\Sigma^0_2$, it will make a difference whether this statement is considered in a classical or a semi-constructive context.

In a semi-constructive context, as illustrated by Theorem \ref{thrm:pulloutpi2-const}, the computational content of the above statement is given by a function $\Psi$ such that
\[
\forall m\in\mathbb{N} \left( \PP\left[\bigcup_{i \geq \Psi(m)} A_i \right] \leq 2^{-m} \right).
\]
Over measure spaces, rewriting the infinite union as a sequence of finite unions, we in particular get (recall Remark \ref{rem:prenex}) that the above is equivalent to 
\[
\forall m\in\mathbb{N} \forall n\in\mathbb{N} \left( \PP\left(\bigcup_{i\in [\Psi(m); n]} A_i \right) \leq 2^{-m} \right),
\]
which has the benefit of being well-defined and suitable also for contents. We will call such a $\Psi$ a \emph{modulus of $A_n$ infinitely often almost never}. This modulus has been used in \cite[Theorem 2.1]{ArthanOliva2021} in connection with the quantitative interpretation of the first Borel-Cantelli lemma, but generally, by the aforementioned theorem, can be extracted from large classes of semi-constructive proofs of the above property.

In a classical context, however, one would have to first transform this statement into a $\forall\exists$-form in the style of the negative translation of the respective statement. This turns $\exists k^0 \forall n^0 \forall i^0 \in [k; n] (\varpi\in A^c_i)$ into $\forall g^1 \exists k^0 \forall i^0 \in [k; g(k)] (\varpi\in A^c_i)$. Therefore, there are various options for representing the resulting probabilistic statement formally. Here, we now just focus on the representation via
\[
\forall m^0\forall g^1 \exists k^0
\left( \PP[ \forall i^0 \in [k; g(k)] \, A_i^c] \geq 1-2^{-m} \right), 
\]
that is in the form of \ref{uniformProblem}, as the other two immediate representations of the associated property, that is $\forall g^1 \left(\PP[\exists k^0 \forall i^0 \in [k; g(k)] \, A_i^c] \geq 1\right)$ in the form of \ref{problem} and $\PP[\forall g^1\exists k^0 \forall i^0 \in [k; g(k)] \, A_i^c] \geq 1$ in the form of \ref{strongProblem}, both classically and constructively, have never been observed in practice. Note that here, the above property is not even mathematically equivalent to the formulation given by $\forall g^1 \left(\PP[\exists k^0 \forall i^0 \in [k; g(k)] \, A_i^c] \geq 1\right)$ as the inner statement, by the presence of the term $g(k)$, is not monotone in $k$.

Now, by Theorem \ref{thrm:pulloutpi2}, from a classical proof of the above, we would be guaranteed to obtain a modulus $\Psi$ such that
\[
\forall m\in\mathbb{N} \forall g:\mathbb{N}\to\mathbb{N} \exists k \leq \Psi(m, g)\left( \PP\left( \bigcup_{i \in [k; g(k)]} A_i \right) \leq 2^{-m} \right).
\]
Such a modulus, which we call a \emph{uniform metastable modulus of $A_n$ infinitely often almost never}, features prominently in the recent work by Powell and Wan on quantitative aspects of zero-one laws in probability theory \cite{PowellWan2025}, but generally, by the aforementioned theorem, can be extracted from large classes of classical proofs of the above property.

Also, such a modulus could have already featured in the interpretation of the first Borel-Cantelli lemma as given by Arthan and the second author in \cite{ArthanOliva2021}, if the result would have been analysed classically instead of constructively. To make this more clear, recall that the first Borel-Cantelli lemma states that for a sequence of events $(A_n)_{n\in\mathbb{N}}$ in a probability space $(\Omega,\Algebra,\PP)$, if $\sum_{n\in\mathbb{N}}\PP(A_n)<\infty$, then $\PP(\limsup_{n\to\infty}A_n)=0$. The paper \cite{ArthanOliva2021} presents a constructive interpretation of this implication by extracting from its (constructive) proof a transformation that maps any \emph{rate of convergence} $\phi(m)$ realizing the premise, that is such that
\[
\forall m\in\mathbb{N}\left(\sum_{n\geq\phi(m)}\PP(A_n)\leq 2^{-m}\right),
\]
into a modulus of $A_n$ infinitely often almost never $\Psi$ as defined above.

If we, however, were to analyse the proof classically, that is by first considering the negative translations, we are actually able to provide a transformation that maps any so-called rate of metastability $\phi(m,g)$ realizing the premise, a weaker computational interpretation given by 
\[
\forall m\in\mathbb{N}\forall g:\mathbb{N}\to\mathbb{N}\exists k\leq\phi(m,g)\left(\sum_{n=k}^{g(k)}\PP(A_n)\leq 2^{-m}\right),
\]
into a uniform metastable modulus of $A_n$ infinitely often almost never $\Psi$ as defined above. Indeed, the argument is essentially trivial. Simply note that 
\[
\PP\left( \bigcup_{n \in [k; g(k)]} A_n \right)\leq \sum_{n=k}^{g(k)}\PP(A_n)
\]
holds using (finite) subadditivity, so that one can immediately set $\Psi=\phi$. Note that this result actually has the benefit that it immediately and elementarily, that is without any additional proof-specific arguments, implies the previous result on transformations of rates of convergence to moduli of $A_n$ infinitely often almost never: If $\phi$ is a rate of convergence for $\sum_{n\in\mathbb{N}}\PP(A_n)<\infty$, it is also a rate of metastability for $\sum_{n\in\mathbb{N}}\PP(A_n)<\infty$, so that we obtain 
\[
\forall m\in\mathbb{N}\forall g:\mathbb{N}\to\mathbb{N}\exists k\leq\phi(m)\left(\sum_{n=k}^{g(k)}\PP(A_n)\leq 2^{-m}\right).
\]
A simple argument via contradiction shows that this is elementarily equivalent to $\phi$ being a modulus of $A_n$ infinitely often almost never, as constructed in \cite{ArthanOliva2021}.\footnote{While such an argument may seem somewhat elaborate for a result with such a trivial proof, we nonetheless presented it in some detail above in order to emphasize a recurring theme in proof mining: The classical analysis of a constructive proof often yields a more general computational result that, in particular, contains the constructive interpretation as a special case.}

\subsection{Almost sure convergence}\label{sec:ASC}

Let $(X_n)_{n \in \NN}$ be a sequence of random variables. One of the most fundamental statements about such a stochastic process is that these variables converge almost surely, where we here focus on the corresponding Cauchy formulation. That $(X_n(\varpi))_{n\in\mathbb{N}}$ is Cauchy can be expressed via the formula
\[
\forall k^0\exists n^0\forall l^0 \forall i^0,j^0\in [n;n+l]\left( \vert X_{i}(\varpi)-X_j(\varpi)\vert \leq_\mathbb{R} 2^{-k} \right).
\]
We are now interested in different formal representations of the statement that \emph{$X_n$ converges almost surely}, together with the computational content that they entail.

Indeed, such convergence results so far are the most common type of theorems analysed in proof mining case studies, as aside from various infinitely often statements already discussed in the previous parts of this section, essentially all other major quantitative results in proof mining and probability theory so far focus on such a probabilistic convergence theorem in one way or the other.

Similarly to the previous case, with events occurring infinitely often almost never, we are here confronted with a $\Pi^0_3$-property and so it will generally make a difference whether the statement is proven constructively or not, in addition to the different interpretations already induced by the concrete formalizations of the associated properties in the style of \ref{strongProblem} -- \ref{uniformProblem}.

Let us again begin with the formulation with the weakest computational content in the sense of Section \ref{sec:meta}, that is with
\[
\forall k^0\left(\PP[\exists n^0\forall l^0 \forall i^0,j^0\in [n;n+l]\left( \vert X_{i}-X_j\vert \leq_\mathbb{R} 2^{-k} \right)]\geq 1\right).
\]
In a semi-constructive context, as illustrated by Theorem \ref{thrm:pulloutpi2-const}, the computational content of the above statement is given by a modulus $\Phi:\mathbb{N}^2\to\mathbb{N}$ such that\footnote{Note that such a bound would also depend on a majorant of the sequence of random variables $(X_n)$. In the context of the system $\mathcal{F}^\omega[\PP]$, random variables are represented as objects of type $1(\Omega)$ and hence, due to the uniform majorizability of the type $\Omega$, are all assumed to be bounded (see \cite{NeriPischke2025} for further discussions of this). In that way, all extracted moduli concerned with (sequences) of random variables actually retain such a dependence on a (sequence of) almost-sure bounds for the variable(s). This creates a mismatch with some applications, where it is possible to derive bounds which are even more uniform on the variables in the sense that they only depend on their $L_1$-norm or other measures related to integrability. These uniform bounds can be systematically accounted for via a slight extension of the systems presented in \cite{NeriPischke2025} and used here, based on the use of an additional abstract type for random variables and a point-free approach to the content space. While this will be fully explored in upcoming work with Thomas Powell, we want to highlight that this merely amounts to a modular extension of the present paper, tailored to a specific, more narrow notion of point-free proofs and corresponding uniform bounds, and in particular builds on the general translation developed in this paper. The resulting arguments for extracting such uniform bounds remain almost exactly the same as the ones given here, so that for the rest of this section, we reserve the right to be vague regarding the quantitative dependence of the bounds on the random variables, as we merely want to illustrate their structure. Regardless, we want to stress that the bounds are still guaranteed to be independent of all parameters relating to the underlying probability content space.}
\[
\forall m\in\mathbb{N} \forall k\in\mathbb{N}\exists n\leq\Phi(m,k)\forall l\in\mathbb{N}\left(\PP(\forall i,j\in [n;n+l]\left(\vert X_{i} - X_{j}\vert \le 2^{-k}\right))\geq 1-2^{-m}\right).
\]
We call such a modulus $\Phi$ a \emph{rate of almost sure convergence for $X_n$}, which, by the aforementioned theorem, can be extracted from large classes of semi-constructive proofs of the previous property. Such rates are abundant in the probability theory literature, and they feature crucially in recent case studies of proof mining in that area, such as laws of large numbers \cite{Neri2024,Neri2023,Neri2025a} and in particular stochastic optimization \cite{NeriPischkePowell2025a,Pischke2025,Pischke2026,PischkePowell2024}.  

In particular, as also highlighted before, these theorems will guarantee that the resulting rates will hold true for all (suitable) finitely additive probability spaces, which is a rather surprising result as in general, almost sure convergence does not imply the existence of a rate of almost sure convergence over finitely additive probability space.

In a classical context, one would as before first have to translate the respective statement into a $\forall\exists$-theorem in the style of the negative translation of the respective statement, turning the previous formulation of convergence into
\[
\forall k^0\forall g^1\exists n^0 \forall i^0,j^0\in [n;n+g(n)]\left( \vert X_{i}(\varpi)-X_j(\varpi)\vert \leq_\mathbb{R} 2^{-k} \right).
\]
Hence, there are again various options for representing the resulting probabilistic statement formally.

Let us first consider the statement 
\[
\forall k^0\forall g^1\PP[\exists n^0\forall i^0,j^0\in [n;n+g(n)]\left( \vert X_{i}-X_j\vert \leq_\mathbb{R} 2^{-k} \right)]\geq 1,
\]
that is a statement of the form \ref{problem}. By Theorem \ref{thrm:pulloutpi2}, from a proof of this we would be able to extract a modulus  $\Phi:\mathbb{N}^2\times \mathbb{N}^\mathbb{N}\to\mathbb{N}$ such that
\[
\forall m\in\mathbb{N} \forall k\in\mathbb{N} \forall g:\mathbb{N}\to\mathbb{N}\left(\PP[\exists n \le \Phi(m,k,g)\forall i,j\in [n;n+g(n)] ( \vert X_{i} - X_j\vert \le 2^{-k})]\geq 1-2^{-m}\right).
\]
We call such a modulus $\Phi$ a \emph{rate of metastable pointwise convergence for $X_n$}. This notion, as far as we know, first appeared in the work of Avigad, Dean and Rute \cite{AvigadDeanRute2012} (extending a notion of Tao \cite{Tao2008b}). In recent proof mining case studies, it appears prominently in \cite{NeriPischkePowell2026} for a quantitative version of a central convergence theorem in stochastic optimization.

The other immediate representation of the associated property is given by
\[
\forall m^0\forall k^0\forall g^1\exists n^0 \left(\PP[\forall i^0,j^0\in [n;n+g(n)]\left( \vert X_{i}-X_j\vert \leq_\mathbb{R} 2^{-k} \right)]\geq 1-2^{-m}\right)
\]
in the form of \ref{uniformProblem}. Note that also here, similar to before in the case of infinitely often almost never statements, the above property is not immediately mathematically equivalent to the formulation given by $\forall k^0\forall g^1\PP[\exists n^0\left( \vert X_{n+g(n)}(\varpi)-X_n(\varpi)\vert \leq_\mathbb{R} 2^{-k} \right)]\geq 1$ as the inner statement, by the presence of the term $g(n)$, is not monotone in $n$. Again, by Theorem \ref{thrm:pulloutpi2}, from a proof of this we would be able to extract a modulus $\Phi:\mathbb{N}^2\times \mathbb{N}^\mathbb{N}\to\mathbb{N}$ such that
\begin{gather*}
\forall m\in\mathbb{N}\ \forall k\in\mathbb{N}\ \forall g:\mathbb{N}\to\mathbb{N}\ \exists n \le \Phi(m,k,g)\\
\left( \PP[\forall i,j\in [n;n+g(n)]\left(\vert X_{i} - X_j\vert \le 2^{-k}\right)]\geq 1-2^{-m}\right).
\end{gather*}
We call such a modulus $\Phi$ a \emph{rate of metastable uniform convergence for $X_n$}, which again, by the aforementioned theorem, can be extracted from large classes of proofs. This notion, as far as we know, first appears in the work of Avigad, Gerhardy and Towsner \cite{AvigadGerhardyTowsner2010} on applications of proof mining to ergodic theory and was later more systematically synthesized in the work of Avigad, Dean and Rute \cite{AvigadDeanRute2012}. 

In particular, the work \cite{AvigadDeanRute2012} is concerned with the relation that rates of metastable uniform convergence have with rates of metastable pointwise convergence. Now, as already discussed in the preceding Section \ref{sec:meta} more generally, but as also immediately apparent from their formulation here, it is clear that a rate of metastable uniform convergence is a rate of metastable pointwise convergence, so only the converse direction is of interest. Now, in the context of probability measures, having a rate of metastable uniform convergence is immediately equivalent to the well-known notion of almost uniform convergence, while having a rate of metastable pointwise convergence is equivalent to pointwise convergence. However, almost sure and almost uniform convergence are equivalent over probability measures by the well-known theorem of Egorov (originally phrased in \cite{Egoroff1911} and later extended by Riesz and Sierpi\'nski), but simple counterexamples show that this equivalence does not extend to probability contents.

Nevertheless, these two quantitative notions of metastable pointwise and metastable uniform convergence remain equivalent over probability contents, as essentially already implicitly shown in \cite{AvigadDeanRute2012}, but only fully highlighted through a proof-theoretic lens in \cite{NeriPischke2025}. In particular, the work \cite{AvigadDeanRute2012} provides a quantitative version of Egorov's theorem on the level of these two metastable notions by means of an explicitly constructed functional which transforms a modulus of metastable pointwise convergence into a modulus of metastable uniform convergence. However, this functional has high computational complexity (bar-recursive in this case), reflecting the complexity of the proof of Egorov's theorem, as also justified from a logical perspective in \cite{NeriPischke2025}. So, the proof-theoretic methodology furnishes the two notions of almost sure and almost uniform convergence with uniform quantitative variants for which a variant of Egorov's theorem can be retained for probability contents. It is exactly this result of \cite{AvigadDeanRute2012} that we abstracted into a general relationship between problems of the form \ref{problem} and \ref{uniformProblem} in the context of measurable and suitably monotone $\Pi^0_1$-matrices in Section \ref{sec:meta}.

Lastly, we want to stress that a rate of almost sure convergence, as described above, is also a direct realiser of the formulation of almost uniform convergence as given above. Therefore, although almost sure convergence and almost uniform convergence are distinct notions for probability contents, if one can establish that a sequence of random variables converges almost surely via a semi-constructive proof as formulated by the systems used in this paper, then one can not only extract a rate but also derive the qualitative result that the sequence converges almost uniformly in the context of probability contents.

Despite the presence of this high-complexity translation relating metastable uniform convergence and metastable pointwise convergence, essentially all quantitative convergence results for stochastic processes obtained via proof mining that do not yield a rate of almost sure convergence already natively provide the stronger rate of metastable uniform convergence. In particular, these results give rise to such rates directly and with low complexity, rather than first establishing a rate of metastable pointwise convergence and then appealing to the translation from \cite{AvigadDeanRute2012}. This seems to be essentially because the respective proofs of the convergence results could be immediately formalized with the convergence explicitly phrased in the form 
\[
\forall m^0\forall k^0\forall g^1\exists n^0\left(\PP[ \forall i^0,j^0\in [n;n+g(n)]\left(\vert X_{i}(\varpi)-X_j(\varpi)\vert \leq_\mathbb{R} 2^{-k}\right) ]\geq 1-2^{-m}\right).
\]
Examples of this arise in the context of general stochastic processes, martingales and stochastic optimization \cite{NeriPowell2025, NeriPowell2024}, laws of large numbers \cite{Neri2023,Neri2025a}, and ergodic theory in finitely additive probability spaces \cite{Neri2025b}. In the latter case, the general considerations on quantitative modes of convergence for random variables, as obtained via proof mining in the present paper, have been instrumental in deriving a new notion of almost sure convergence for probability contents. Indeed, there has so far been only one proof-mining application to stochastic processes that directly yields the weaker rate of metastable pointwise convergence, which is given in the recent work \cite{NeriPischkePowell2026} on stochastic optimization already mentioned previously for its peculiar features regarding almost sure infinitely often statements.

In all these other cases, not only are these rates $\Phi$ of metastable uniform convergence of low complexity but they are in particular often of the special form
\[
\Phi(m,k,g):=\tilde{g}^{(\phi(m,k))}(0)
\]
for $\tilde{g}(n):=n+g(n)$ and another function $\phi(m,k)$ (commonly itself also a simple function), that is they are \emph{learnable}, following the terminology of the work of Kohlenbach and Safarik \cite{KohlenbachSafarik2014} (which actually studies a much broader concept, in a deterministic setting however).\footnote{In general, even the notion of a rate of metastable uniform/pointwise convergence can be related to as a notion of generalized learnability as discussed in \cite{NeriPischkePowell2025b}.} Indeed, it is shown in \cite{KohlenbachSafarik2014} that under a restriction on the amount of law-of-excluded-middle in a proof, one can guarantee the extractability of learnable rates of metastability for deterministic sequences, and we deem it highly likely that this result extends to the present probabilistic context, so that results such as the above can be a priori guaranteed by means of our present probabilistic metatheorems together with a restriction on the amount of law-of-excluded-middle in a proof. We, however, leave this for future work.

Lastly, we just want to remark that statements such as
\[
\forall k^0\left(\PP[\forall g^1\exists n^0\left( \vert X_{n+g(n)}(\varpi)-X_n(\varpi)\vert \leq_\mathbb{R} 2^{-k} \right)]\geq 1\right)
\]
or even $\PP[\forall k^0\forall g^1\exists n^0\left( \vert X_{n+g(n)}(\varpi)-X_n(\varpi)\vert \leq_\mathbb{R} 2^{-k} \right)]\geq 1$, in the form of \ref{strongProblem}, have never been observed in practice, so that we do not expand on them here.

\subsection{Finite fluctuations}\label{sec:FLUC}

The final important property we want to discuss is the property that a stochastic process $(X_n)_{n\in\mathbb{N}}$ has \emph{finite fluctuations almost surely}, that is that
\[
\forall k\in\mathbb{N}\left( \PP(J_{k}(X_n)<\infty)=1\right),
\]
where $J_{k}(X_n)$ is the total number of $2^{-k}$-fluctuations that the process $(X_n)_{n\in\mathbb{N}}$ experiences, that is $J_{k}(X_n(\varpi)) := \lim_{N\to\infty}J_{N,k}(X_n(\varpi))$ where $J_{N,k}(X_n)$ is the number of $2^{-k}$-fluctuations that the process $(X_n)_{n\in\mathbb{N}}$ experiences during the first $N$ elements, i.e.\ the maximal $n\in\mathbb{N}$ such that there are
\[
i_1<j_1\leq i_2<j_2\leq\dots\leq i_n<j_n<N \text{ with }\vert X_{i_l}(\varpi)-X_{j_l}(\varpi)\vert> 2^{-k},
\]
for all $l=1,\dots,n$. Fluctuations are a central tool in the study of stochastic processes and in particular their quantitative behaviour. In particular, they have featured in quite a few case studies on proof mining and probability theory, notably \cite{NeriPischkePowell2026,NeriPowell2025} (see also \cite{AvigadRute2015} and \cite{Neri2025b}). In fact, having finite fluctuations is (non-effectively) equivalent to the process converging almost-surely (see e.g.\ the discussion in \cite{NeriPowell2025}). However, the quantitative property of having a bound on the number of fluctuations, that is a so-called \emph{modulus of finite fluctuations almost surely} $\phi:\mathbb{N}^2\to\mathbb{N}$ with the property that
\[
\forall m\in\mathbb{N}\forall k\in\mathbb{N}\left( \PP(J_{k}(X_n)\leq\phi(m,k))\geq 1-2^{-m}\right),
\]
falls, as a quantitative property, strictly between having a rate of almost-sure convergence and having a learnable rate of metastable pointwise convergence, which is in fact already true for deterministic sequences (see \cite{KohlenbachSafarik2014}). We refer to \cite{KohlenbachSafarik2014} for further discussions on this and related notions in the deterministic case and to \cite{NeriPowell2025} for the stochastic case.

The question when the extractability of moduli of finite fluctuations almost surely from convergence proofs can be guaranteed is quite subtle, as already discussed at length in the work of Kohlenbach and Safarik \cite{KohlenbachSafarik2014} for the deterministic case. Indeed, it seems difficult to find any a priori guarantee for this, and \cite{KohlenbachSafarik2014} only provides a type of a posteriori condition where, if one obtains a learnable rate of metastability that satisfies a certain so-called gap condition, then such a rate can be converted to a modulus of finite fluctuations. However, as far as is known, this gap condition cannot be guaranteed a priori based on additional structure on the proof, such as restricted uses of the law-of-excluded-middle.

The circumstances in the case studies mentioned above are, however, more specific. On a closer look, we find that these case studies actually all analyse proofs where the property of having finite fluctuations is \emph{directly} derived from an associated strong inequality bounding the number of crossings for the stochastic process in question (see e.g.\ Theorem 5.1 in \cite{NeriPowell2025} and generally the exposition therein for further information on crossing inequalities). While the resulting proof of the almost-sure convergence of the process is not constructive, as having finite fluctuations almost surely only implies almost sure convergence non-effectively, this derivation of the property of finite fluctuations is in these cases always semi-constructive and can be analysed on its own, which provides a priori guarantees for the existence of modulus of finite fluctuations almost surely in these cases.

Indeed, we can formally recognize these cases as follows: Adapting Definition 2.2 from \cite{KohlenbachSafarik2014}, we write
\[
J^X_{k,\varpi}(n,i,j) \;:\equiv\; \begin{cases}\mathrm{lh}(i)=_0\mathrm{lh}(j)=_0n \\ \qquad \land \; \forall l<_0n( i_l<_0j_l)\\
\qquad\land \; \forall l<_0n-1( j_l\leq_0 i_{l+1}) \\ \qquad \land \; \forall l<_0n( \vert X_{i_l}(\varpi)-X_{j_l}(\varpi)\vert>_\mathbb{R}2^{-k}),\end{cases}
\]
where $\mathrm{lh}(i)$ and $\mathrm{lh}(j)$ represent the length of the finite sequence coded by $i$ and $j$ (under some canonical coding of finite sequences) and $i_l,j_l$ access their respective elements at index $l$. Therefore, $J^X_{k,\varpi}(n,i,j)$ formally expresses the property that $(X_n(\varpi))_{n\in\mathbb{N}}$ has $n$-many $2^{-k}$-fluctuations with indices coded by $i$ and $j$.  We write 
\[
J^X_{k,\varpi}\leq b \;:\equiv\; \forall n>_0 b\forall i^0\forall j^0 \neg J^X_{k,\varpi}(n,i,j),
\]
expressing that $(X_n(\varpi))_{n\in\mathbb{N}}$ has $\leq b$-many $2^{-k}$-fluctuations. That $(X_n(\varpi))_{n\in\mathbb{N}}$ has finite fluctuations can now be formally expressed via 
\[
\forall k^0\exists b^0\left( J^X_{k,\varpi}\leq b\right),
\]
so that the resulting property of having finite fluctuations almost surely translates to
\[
\forall k^0\left(\PP[\exists b^0\left( J^X_{k}\leq b\right)]\geq 1\right),
\]
that is a statement of type \ref{problem}. By Theorem \ref{thrm:pulloutpi2-constUniform}, from a semi-constructive proof of this (which is a necessary requirement as $J^X_{k}\leq b$ is, essentially, universal), we are then be able to extract a modulus $\phi:\mathbb{N}^2\to\mathbb{N}$ such that
\[
\forall m\in\mathbb{N}\forall k\in\mathbb{N}\left(\PP\left[\bigcup_{b\leq\phi(m,k)}\left( J^X_{k}\leq b\right)\right]\geq 1-2^{-m}\right).
\]
Noting that using the monotonicity of $J^X_k\leq b$ the above is equivalent to
\[
\forall m\in\mathbb{N}\forall k\in\mathbb{N}\left(\PP\left[J^X_{k}\leq \phi(m,k)\right]\geq 1-2^{-m}\right),
\]
this directly represents that $\phi$ is a modulus of finite fluctuations almost surely, as defined above. It is exactly such analyses of semi-constructive proofs of this property that give rise to the moduli of finite fluctuations almost surely constructed in the case studies mentioned previously, giving the first systematic proof-theoretic explanation thereof. \\

\noindent{\textbf{Acknowledgments:}} The authors are grateful to Ulrich Kohlenbach for numerous helpful comments on various drafts of this paper. In particular, he provided crucial insights on the treatment of uniform boundedness principles in higher types and of contra-collection principles in the context of the monotone functional interpretation. The authors also thank Thomas Powell for valuable remarks on a draft of the paper.

\bibliographystyle{plain}
\bibliography{ref}

\end{document}